\documentclass[11pt]{article}

\usepackage{mathtools}
\usepackage{amssymb}
\usepackage{amsthm}
\mathtoolsset{showonlyrefs}
\usepackage[margin=1in]{geometry}

\usepackage{graphicx}
\graphicspath{{./art/}}
\usepackage{subcaption}
\usepackage{multirow}
\usepackage{adjustbox}
\usepackage{booktabs}
\usepackage{array}
\usepackage{arydshln}
\usepackage[plain,noend]{algorithm2e}
\usepackage{enumitem}

\usepackage{authblk}
\usepackage[authoryear,round]{natbib}

\usepackage{xcolor}

\newcommand{\colsp}{\mathrm{colsp}}
\newcommand{\rowsp}{\mathrm{rowsp}}
\newcommand{\vspan}{\mathrm{span}}

\newcommand{\cV}{\mathcal{V}}

\newcommand{\NN}{\mathbb{N}}
\newcommand{\RR}{\mathbb{R}}
\newcommand{\Rb}{\mathbb{R}}
\newcommand{\bone}{\boldsymbol{\mathbf{1}}}
\newcommand{\bbE}{\mathbb{E}}

\newcommand{\bc}{\boldsymbol{c}}
\newcommand{\bd}{\boldsymbol{d}}
\newcommand{\be}{\boldsymbol{e}}

\newcommand{\bh}{\boldsymbol{h}}

\newcommand{\bk}{\boldsymbol{k}}

\newcommand{\bv}{\boldsymbol{v}}
\newcommand{\bw}{\boldsymbol{w}} 
\newcommand{\bx}{\boldsymbol{x}}
\newcommand{\by}{\boldsymbol{y}} 
 
\newcommand{\bA}{\boldsymbol{A}}
\newcommand{\bB}{\boldsymbol{B}}
\newcommand{\bC}{\boldsymbol{C}}
\newcommand{\bD}{\boldsymbol{D}}

\newcommand{\bG}{\boldsymbol{G}}
\newcommand{\bH}{\boldsymbol{H}}
\newcommand{\bI}{\boldsymbol{I}}

\newcommand{\bL}{\boldsymbol{L}}
\newcommand{\bM}{\boldsymbol{M}}
\newcommand{\bN}{\boldsymbol{N}}

\newcommand{\bP}{\boldsymbol{P}}
\newcommand{\bQ}{\boldsymbol{Q}}

\newcommand{\bS}{\boldsymbol{S}}
\newcommand{\bT}{\boldsymbol{T}}
\newcommand{\bU}{\boldsymbol{U}}
\newcommand{\bV}{\boldsymbol{V}}
\newcommand{\bW}{\boldsymbol{W}} 
\newcommand{\bX}{\boldsymbol{X}}

\newcommand{\bSigma}{\boldsymbol{\Sigma}} 
\newcommand{\bPi}{\boldsymbol{\Pi}}
\newcommand{\balpha}{\boldsymbol{\alpha}} 
\newcommand{\bbeta}{\boldsymbol{\beta}} 
\newcommand{\bvarepsilon}{\boldsymbol{\varepsilon}}

\newcommand{\bnu}{\boldsymbol{\nu}} 
\newcommand{\bzero}{\boldsymbol{0}} 
 
\newcommand{\bgamma}{\boldsymbol{\gamma}} 
\newcommand{\bOmega}{\boldsymbol{\Omega}} 
 
\newcommand{\bDelta}{\boldsymbol{\Delta}} 
 
\newcommand{\Gram}{\bG}
\newcommand{\bbetaopt}{\bbeta^{\mathrm{opt}}}
\newcommand{\bnuopt}{\bnu^{\mathrm{opt}}}

\newcommand{\cI}{\mathcal{I}}
\newcommand{\cJ}{\mathcal{J}}
\newcommand{\cL}{\mathcal{L}} 
\newcommand{\cN}{\mathcal{N}} 

\newcommand{\cR}{\mathcal{R}}
\newcommand{\cS}{\mathcal{S}} 

\newcommand{\cX}{\mathcal{X}}

\newcommand{\halpha}{\widehat{\alpha}} 
\newcommand{\hbeta}{\widehat{\beta}} 
\newcommand{\hDelta}{\widehat{\Delta}}

\newcommand{\hbbeta}{\widehat{\bbeta}} 
\newcommand{\hq}{\widehat{q}} 
\newcommand{\hy}{\widehat{y}}

\newcommand{\bhV}{\widehat{\bV}} 
\newcommand{\hvarepsilon}{\widehat{\varepsilon}} 
\newcommand{\hbvarepsilon}{\widehat{\bvarepsilon}} 
 
\newcommand{\hsigma}{\widehat{\sigma}} 

\newcommand{\hbalpha}{\widehat{\balpha}} 
 
\newcommand{\hbDelta}{\widehat{\bDelta}}

\newcommand{\tbM}{\widetilde{\bM}}
\newcommand{\tbS}{\widetilde{\bS}}
\newcommand{\tbX}{\widetilde{\bX}}
\newcommand{\tbbeta}{\widetilde{\bbeta}}
\newcommand{\tbvarepsilon}{\widetilde{\bvarepsilon}} 
\newcommand{\tvarepsilon}{\widetilde{\varepsilon}}

\newcommand{\tDelta}{\widetilde{\Delta}} 
\newcommand{\tbDelta}{\widetilde{\bDelta}} 

\newcommand{\diag}{\mathrm{diag}} 
\newcommand{\rank}{\mathrm{rank}\,}

\newcommand{\tr}{\text{\normalfont tr}\,}

\newcommand{\Cov}{\text{\normalfont Cov}} 
\newcommand{\nCov}{\text{\normalfont Cov}} 
\newcommand{\olsmin}{{\normalfont\texttt{OLS}}}
\newcommand{\si}{\sim i}
\newcommand{\J}{\textrm{Jack}}
\newcommand{\F}{\textup{F}}
\newcommand{\uL}{\textup{L}}

\newcommand{\PRESS}{\textsf{PRESS}} 
\newcommand{\jackknife}{\textsf{jackknife}}
\newcommand{\jackknifep}{\textsf{jackknife+}}

\newcommand{\PI}{\widehat{\mathcal{PI}}}

\newcommand{\Normal}{\mathsf{N}}

\DeclareMathOperator*{\argmin}{\arg \min}

\usepackage[
  colorlinks=true,
  citecolor=magenta,
  linkcolor=blue,
  urlcolor=blue
]{hyperref}

\makeatletter
\@ifundefined{theorem}{\newtheorem{theorem}{Theorem}[section]}{}
\@ifundefined{proposition}{\newtheorem{proposition}[theorem]{Proposition}}{}
\@ifundefined{lemma}{\newtheorem{lemma}[theorem]{Lemma}}{}
\@ifundefined{corollary}{\newtheorem{corollary}[theorem]{Corollary}}{}
\theoremstyle{definition}
\@ifundefined{definition}{\newtheorem{definition}[theorem]{Definition}}{}
\@ifundefined{assumption}{\newtheorem{assumption}[theorem]{Assumption}}{}
\@ifundefined{example}{\newtheorem{example}[theorem]{Example}}{}
\theoremstyle{remark}
\@ifundefined{remark}{\newtheorem{remark}[theorem]{Remark}}{}
\makeatother

\providecommand{\keywords}[1]{%
  \par\medskip\noindent\textbf{Keywords: }#1%
}

\title{Benign Overfitting Beyond Prediction:\\
The Ordinary Least Squares Interpolator}

\author[1]{Dennis Shen}
\author[2]{Dogyoon Song}
\author[3]{Peng Ding}
\author[4]{Jasjeet Sekhon}

\affil[1]{Department of Data Sciences \& Operations, University of Southern California,\\
Los Angeles, CA 90089, USA; \texttt{dennis.shen@marshall.usc.edu}}
\affil[2]{Department of Statistics, University of California, Davis,\\
Davis, CA 95616, USA; \texttt{dgsong@ucdavis.edu}}
\affil[3]{Department of Statistics, University of California, Berkeley,\\
Berkeley, CA 94720, USA; \texttt{pengdingpku@berkeley.edu}}
\affil[4]{Google DeepMind, Mountain View, CA 94043, USA; \texttt{jas@jsekhon.com}}

\date{}

\begin{document}

\maketitle

\begin{abstract}

Recent advances in deep learning have highlighted the phenomenon of benign overfitting in overparameterized statistical models, sparking significant interest in understanding its foundations. Owing to its simplicity and practical relevance, the ordinary least squares (OLS) interpolator has become a key object of study for gaining theoretical insight into this phenomenon. While the properties of OLS are well understood in classical underparameterized settings, its behavior in the overparameterized regime---unlike that of ridge regression or the lasso---remains comparatively less explored. We contribute to this growing literature by deriving new algebraic and statistical results for the minimum $\ell_2$-norm OLS interpolator. In contrast to much of the existing work, which focuses on prediction risk, we center our analysis on parameter estimation and inference, which are fundamental for many statistics and causal inference applications. Specifically, we establish overparameterized analogues of (i) the leave-$k$-out formulas, (ii) the omitted variable bias formula, and (iii) the Frisch–Waugh–Lovell theorem. Under the Gauss–Markov model, we further extend the Gauss–Markov theorem and analyze variance estimation under homoskedasticity in the overparameterized setting. Collectively, these results provide a systematic framework for studying parameter estimation and inference in overparameterized linear models, offering a novel perspective on benign overfitting beyond its implications for prediction.
\end{abstract}

\keywords{Frisch--Waugh--Lovell theorem; Gauss--Markov theorem;
leave-one-out; omitted-variable bias.}

\section{Introduction} \label{sec:intro}
\subsection{Recent interest in the Ordinary Least Squares Interpolator}
The study of overparameterized statistical models in deep learning has unveiled a striking phenomenon known as benign overfitting \citep{zhang2017understanding, belkin_2021}, whereby models that perfectly fit noisy training data can nevertheless generalize well to unseen data.  
This discovery has ignited considerable theoretical interest in statistics and machine learning. 
At the center of this discussion lies the ordinary least squares (OLS) interpolator. 

The OLS estimator is a foundational tool whose simplicity and broad applicability have made it central to both theoretical analysis and practical implementation across a wide range of domains.
Given a set of covariates $\bX \in \Rb^{n \times p}$ and responses $\by \in \Rb^n$, the OLS estimator is obtained by regressing $\by$ on $\bX$. 
The minimum $\ell_2$-norm OLS estimator is defined as $\hbbeta = \bX^{\dagger} \by$, where $\bX^{\dagger}$ is the Moore-Penrose pseudoinverse of $\bX$. 
This estimator is a solution to the problem, $\hbbeta \in \argmin_{\bbeta \in \RR^p} \| \by - \bX \bbeta \|_2^2$, and is called the OLS interpolator when $\by = \bX \hbbeta$.  
Notably, this estimator arises both as the limit of ridge regression as the regularization parameter tends to zero and as the limit of gradient flow, initialized at zero, applied to the least squares loss \citep{hastie_22}. 

The properties of OLS are well understood in the classical underparameterized regime ($n > p$), where notions such as bias, consistency, and inference are well established. 
By contrast, its behavior in high-dimensional, overparameterized settings ($p > n$)---a hallmark of modern datasets---remains comparatively underexplored relative to its explicitly regularized counterparts such as the lasso and ridge regression. 
Recent works, however, have advanced our comprehension of the prediction performance of OLS in this regime under specific stochastic assumptions on the data-generating process \citep{bartlett2020benign, belkin20, hastie_22}. 
%

Yet, in many statistics applications, the central focus is not with prediction, but with parameter estimation and inference. 
This is particularly commonplace in causal inference, where OLS often plays a pivotal role: under suitable identification assumptions, components of $\hbbeta$ acquire causal interpretation. 
In such settings, the validity of estimation and inference rest on classical OLS properties that hold only when $n > p$. 
When $p > n$, several challenges emerge: 
(i) infinitely many solutions exist, complicating the generalization of standard estimators, and 
(ii) traditional inferential procedures that rely on in-sample residuals falter since residuals vanish under interpolation.
Motivated by these challenges, we study the minimum $\ell_2$-norm OLS interpolator from the perspective of parameter estimation and inference, with the broader goal of understanding how far benign overfitting extends beyond prediction.

\subsection{Contributions} 
In this work, we uncover high-dimensional counterparts of several key results in the classical setting.
Our analysis proceeds in two stages. 
We first develop algebraic results that are independent of any assumptions on the data-generating process. 
This assumption-free perspective makes the results broadly applicable and provides a flexible foundation for future study under specific modeling frameworks. 
Building on these results, we endow our observations with one particular generative model, allowing us to investigate how the structure of the covariates shapes parameter recovery and inference.
We summarize the main contributions of this article below and provide a comprehensive overview in Table \ref{tab:summary}.

\medskip
\noindent \textit{Algebraic results.}
%
Sections~\ref{sec:row_partition} and \ref{sec:subsampled_column} provide high-dimensional analogs of three fundamental OLS results in the classical regime.
\begin{enumerate}[label=(\alph*)]
    \item 
    \textit{Leave-$k$-out formulas.} 
    Section \ref{sec:results.rows} develops a high-dimensional analogue of the leave-$k$-out formula (Theorem \ref{thm:subsampled_rows}). Specifically, we show that the OLS interpolator based on a subsample can be expressed as a projection of the full-sample OLS interpolator. 
    We further derive high-dimensional leave-one-out formulas (Corollary~\ref{cor:loo}) and corresponding residuals (Corollary~\ref{cor:loo.shortcut}). 
    While related residual characterizations have appeared previously, we provide a new proof based on the intermediate leave-one-out representation in Corollary~\ref{cor:loo}, which directly links the full-sample interpolator to its leave-one-out counterparts. 
    %

    \item 
    \textit{Frisch-Waugh-Lovell (FWL) theorem.}
    Section \ref{sec:partial_regularization} presents a high-dimensional equivalent of the FWL theorem (Theorem \ref{thm:partially_regularized_OLS}), which decomposes the influence of a covariate variable on the response after adjusting for the other variables. 
    Specifically, we consider the OLS interpolator that is (implicitly) regularized only for a subset of covariates rather than fully regularized for the entire collection of covariates (as in the standard minimum $\ell_2$-norm OLS). 
    We obtain the same formula as in the classical FWL theorem for the regularized covariates and obtain a similar yet different expression for the ``un-regularized'' covariates.
    We then illustrate the practical relevance of these results in treatment effect estimation with high-dimensional covariates, focusing on two important settings. 

    \item 
    \textit{Omitted variable bias formula.}
    Section \ref{sec:subsampled_columns} provides a high-dimensional counterpart of the omitted variable bias formula (Theorem \ref{thm:sub_column_OLS}).
    Our result relates the OLS interpolator based on a subset of covariates (aka ``short regression'') to 
    the full set of covariates (aka ``long regression''). 
    We show that the omitted variable bias formula in the classical regime continues to hold in high-dimensions for the minimum $\ell_2$-norm OLS interpolators; a weaker form of the formula concerning prediction holds for any pair of OLS interpolators, not necessarily having minimum $\ell_2$-norm. 
    
\end{enumerate}

\noindent \textit{Stochastic results.}
In Section \ref{sec:stat_inf}, we extend our analysis to the statistical setting, studying the behavior of the OLS interpolator under the Gauss-Markov model. 
\begin{enumerate}[label=(\alph*)]
    \item 
    \textit{Gauss-Markov theorem.} 
    Section \ref{sec:gauss.markov} presents a high-dimensional extension of the Gauss-Markov theorem (Theorem \ref{thm:Gauss_Markov}), which establishes the optimality of the OLS interpolator among linear unbiased estimators, albeit with certain restrictions compared to the classical counterpart. 
    Since the true linear regression model is not identifiable in high-dimensions, it is reasonable to focus on estimating its projection onto the rowspace of $\bX$. 
    Our result proves that the OLS interpolator exhibits minimal covariance for this projected parameter. 
    
    \item 
    \textit{Homoskedastic variance estimator.} Building on our leave-one-out results,
    Section \ref{sec:var_estimation} proposes and analyzes a natural variance estimator for both the classical and high-dimensional regimes under homoskedasticity (Theorem \ref{thm:var.estimator}), providing a potential foundation for inference under interpolation.
    To the best of our knowledge, this variance estimator is novel even in the classical regime. 
    We show that this estimator is unbiased in classical settings but is conservative in high-dimensional settings. 
\end{enumerate}

\noindent \textit{Simulation studies.} In Section~\ref{sec:simulations.updated}, we present simulation studies that illustrate the role of $\bX$ in parameter estimation and inference. Our results suggest that spiked-covariance-type models tend to yield especially favorable performance. Interestingly, this pattern mirrors conclusions from the benign overfitting literature on prediction, e.g., \citep{bartlett2020benign, Tsigler2020BenignOI}, suggesting that the value of approximate low-rank structure extends beyond predictive accuracy to parameter recovery and inference as well.

\begin{table}[t]
    \centering
    \caption{Summary of main results in this paper.}
    \label{tab:summary}
    \begin{adjustbox}{max width=\textwidth}
    \begin{tabular}{l c c c}
        \toprule
        Result  &   Section     &   Classical counterpart    &   Implications/Applications   \\
        \midrule
        Theorem \ref{thm:subsampled_rows}   &   Section \ref{sec:results.rows}  &   Leave-$k$-out formula       &   Stability of the OLS interpolator \\
        Theorem \ref{thm:sub_column_OLS}    &   Section \ref{sec:subsampled_columns}    &   Omitted variable bias formula   &   Bias quantification for omitted variables \\
        Theorem \ref{thm:partially_regularized_OLS} & Section \ref{sec:partial_regularization}  &   Frisch-Waugh-Lovell (FWL) theorem   &   Interpretation of OLS coefficients\\
        Theorem \ref{thm:Gauss_Markov}  &   Section \ref{sec:gauss.markov}  &   Gauss-Markov theorem    &   Optimality of the OLS interpolator\\
        Theorem \ref{thm:var.estimator} &   Section \ref{sec:var_estimation}    &   Homoskedastic variance estimator    &   (Conservative) variance estimation\\
        \bottomrule
    \end{tabular}
    \end{adjustbox}
\end{table}

\begin{remark}[Statement of results] \label{remark:hd_classical} 
    To better compare our high-dimensional results with their known counterparts from the classical regime, we will often present both sets of results in the same statement. 
    Unless specified otherwise, all results in the classical regime are simply a restatement of previously known results from the literature, i.e., such classical results are \textit{not} novel results developed in this article. 
\end{remark} 

\subsection{Notation}\label{sec:notation}
For any $n \in \NN$, let $[n] \coloneq \{ 1, \dots, n \}$.  
We denote $\langle \bx, \by \rangle = \sum_{i=1}^n x_i y_i$ for $\bx, \by \in \RR^n$ and $\| \bv \|_2 \coloneqq \langle \bv, \bv \rangle^{1/2}$. 
Let $\be_i$ signify the $i$-th standard basis vector of $\RR^n$ for each $i \in [n]$.
Let $\bI_n$ denote the $n \times n$ identity matrix. 
For any matrix $\bM \in \Rb^{n \times p}$, let $\| \bM \|_{\F} \coloneqq \tr( \bM^{\top} \bM )^{1/2}$ denote its Frobenius norm.
Let $\bM^{\top}$, $\bM^{-1}$, and $\bM^{\dagger}$ denote the transpose, inverse, and Moore-Penrose pseudoinverse of $\bM$, respectively. 
If the singular value decomposition of $\bM$ is denoted as $\bM = \bU \bS \bV^{\top}$, then $\bM^{\dagger} = \bV \bS^{-1} \bU^{\top}$. 
The column space and row space of $\bM$ are denoted as $\colsp(\bM) \subseteq \RR^n$ and $\rowsp(\bM) \subseteq \RR^p$, respectively. 
Let $\bP_{\bM} \coloneqq \bM \bM^{\dagger}$ symbolize the projection matrix onto the column space of $\bM$. 
We define $\bP_{\bM}^{\perp} = \bI_n - \bP_{\bM}$ to represent the projection matrix onto the orthogonal complement of $\colsp(\bM)$.
Moreover, let $\Gram_{\bM} = (\bM \bM^\top)^{\dagger} \in \Rb^{n \times n}$.
For $\bv \in \Rb^n$, let $\diag(\bv) \coloneqq \sum_{i=1}^n v_i \be_i \be_i^\top$.
For $\bM \in \Rb^{n \times n}$, let $\diag(\bM) \coloneqq \sum_{i=1}^n M_{ii} \be_i \be_i^\top$. 
For nonempty sets $\cI \subseteq [n]$ and $\cJ \subseteq [p]$, let $\bM_{\cI, \cJ} \in \RR^{|\cI| \times |\cJ|}$ be the submatrix of $\bM$ containing elements  $M_{ij}$ with $(i,j) \in \cI \times \cJ$. 
As a shorthand, let $\bM_{\cI, \star} \in \Rb^{|\cI| \times p}$ be the submatrix obtained by retaining the rows of $\bM$ with indices in $\cI$; define $\bM_{\star, \cJ} \in \Rb^{n \times |\cJ|}$ analogously. 
For a vector $\bv \in \Rb^n$, let  $\bv_{\cI} \in \Rb^{|\cI |}$ be its subvector.

\section{The ordinary least squares (OLS) estimator} \label{sec:setup} 

\subsection{The minimum Euclidean-norm OLS estimator} \label{sec:ols.interpolator}
We assume access to in-sample, or training, data $\{(\bx_i, y_i): i \in [n]\}$, where $\bx_i \in \Rb^p$ and $y_i \in \Rb$ are the $i$-th covariate and response, respectively. 
Let $\bX \in \Rb^{n \times p}$ collect the covariates and $\by \in \Rb^n$ collect the responses. 
We are interested in the minimum Euclidean-norm ($\ell_2$-norm) OLS estimator based on the data $(\bX, \by)$, which is formally defined below.  

\begin{definition}\label{defn:min_ols} 
   Let $\cS \coloneqq \argmin_{\bbeta' \in \Rb^p} \| \by - \bX \bbeta' \|_2^2 $ denote the set of vectors that minimize the in-sample $\ell_2$-error. 
    The \emph{minimum $\ell_2$-norm OLS estimator}, denoted by $\hbbeta$, 
    is defined as the minimizer satisfying 
    $\hbbeta \in \argmin_{\bbeta \in \cS} \| \bbeta \|_2^2$.
\end{definition}

Although Definition~\ref{defn:min_ols} is implicit, it defines a unique estimator for every $(\bX, \by)$: even if $\cS$ contains multiple elements, it is an affine set, and the strict convexity of $\| \bbeta \|_2^2$ ensures a unique minimum-norm element.
In what follows, we study two regimes determined by the relative sizes of $n$ and $p$, under which we further impose natural rank conditions on $\bX$. 


\begin{proposition}\label{prop:beta_hat}
    For any $\bX$ and $\by$, 
    the minimum $\ell_2$-norm OLS estimator is $\hbbeta = \bX^{\dagger} \by$.
\end{proposition}

\begin{proof} 
In the case when $\rank(\bX) \geq p$ (i.e., when $\rank(\bX) = p$),  the set $\cS$ consists of only a single element, i.e., $\cS = \big\{ \bX^{\dagger} \by \big\}$ and thus, $\hbbeta = \bX^{\dagger} \by$. 
When $\rank(\bX) < p$, the set $\cS$ takes the form of an affine subspace in $\RR^p$ of codimension $\rank(\bX)$, i.e., $\cS = \big\{ \bX^{\dagger} \by + \bv \in \RR^p : \bv \in \cN(\bX) \big\}$; notably, $\cS$ is a convex set. 
Given that the quadratic objective function $\|\bbeta \|_2^2$ is both strictly convex and coercive (meaning $\|\bbeta \|_2^2 \to \infty$ as $\| \bbeta \|_2 \to \infty$), there exists a unique minimum solution $\hbbeta \in \argmin_{\bbeta \in \cS} \| \bbeta \|_2^2$. 
Now, observe that $\bX^{\dagger} \by \in \rowsp(\bX) = \cN(\bX)^{\perp}$ and thus, $\langle \bX^{\dagger} \by, \bv \rangle = 0$ for all $\bv \in \cN(\bX)$. 
Therefore, $\| \bX^{\dagger} \by + \bv \|_2^2 = \| \bX^{\dagger} \by \|_2^2 + \| \bv \|_2^2 \geq \| \bX^{\dagger} \by \|_2^2$, with equality holding if and only if $\bv = 0$. 
Consequently, $\hbbeta = \bX^{\dagger} \by $, regardless of the rank of $\bX$. 
\end{proof}

For ease of exposition, we define a map, $\olsmin: (\bX, \by) \mapsto \bX^{\dagger} \by$, and write $\olsmin(\bX, \by) \coloneqq \bX^\dagger \by = \hbbeta$. 
In Sections \ref{sec:classical_regime} and \ref{sec:hd_regime} to follow, we consider two parameterization regimes, introduce relevant assumptions, and provide corresponding interpretations for $\hbbeta$. 

\subsection{OLS in the classical regime} \label{sec:classical_regime}
Recall that the classical, underparameterized regime is characterized by $n > p$. 
We state a canonical assumption placed on $\bX$ within this context. 

\begin{assumption} \label{assump:column_rank}
	$\bX$ has full column rank, i.e., $\rank(\bX) = p$.
\end{assumption}

Assumption~\ref{assump:column_rank} can hold only if $n \geq p$. 
As aforementioned, this implies that $\cS$ is a singleton and hence, the OLS solution is unique. 
Moreover, under Assumption~\ref{assump:column_rank}, we note that $\bX^{\top} \bX$ is invertible and thus, $\bX^{\dagger} = ( \bX^{\top} \bX )^{-1} \bX^{\top}$.
This restores the traditional OLS formula: 
$\hbbeta = ( \bX^{\top} \bX )^{-1} \bX^{\top} \by$.

\subsection{OLS in the high-dimensional regime} \label{sec:hd_regime}
The primary focus in this paper lies within the high-dimensional, overparameterized regime with $n \le p$. 
As needed, we impose the following assumption on $\bX$. 

\begin{assumption} \label{assump:row_rank}
	$\bX$ has full row rank, i.e., $\rank(\bX) = n$.
\end{assumption}

Assumption \ref{assump:row_rank} holds only if $n \leq p$. 
It implies that $\colsp(\bX) = \RR^n$ and $\min_{\bbeta' \in \RR^p} \| \by - \bX \bbeta' \|_2^2 = 0$. 
Consequently, for any $\bbeta \in \cS$, the relationship $y_i = \langle \bx_i, \bbeta \rangle$ holds for all $i \in [n]$, i.e., every solution in $\cS$ ``interpolates'' the in-sample data.  
By contrast, the classical regime with Assumption~\ref{assump:column_rank} yields a solution that can have nontrivial in-sample residuals. 
Under Assumption \ref{assump:row_rank}, every $\bbeta \in \cS$ is an interpolating solution. 
To remove any ambiguities, this work focuses on the minimum $\ell_2$-norm solution $\hbbeta$ (Definition~\ref{defn:min_ols}), which we will refer as the ``OLS interpolator'' for convenience. 

\begin{remark}[Geometric view for high-dimensional OLS]
    Under Assumption~\ref{assump:row_rank}, we have $\hbbeta = \bX^{\top} ( \bX \bX^{\top} )^{-1} \by$. 
    For each $j \in [p]$, the $j$-th coordinate of $\hbbeta$ is given as $\widehat{\beta}_j 
            = {\bx_j'}^{\top} ( \bX \bX^{\top} )^{-1} \by$, 
    where $\bx'_j \in \RR^n$ denotes the vector whose $i$-th coordinate is the value of variable $j$ in the $i$-th sample. 
    Since $\bX \bX^{\top}$ is positive definite, we can interpret this formula as the general inner product between $\bx'_j$ and $\by$ weighted by $( \bX \bX^{\top} )^{-1}$.
\end{remark}

\section{Row-partitioned regression} \label{sec:row_partition}
\subsection{Setting}
The leave-$k$-out setting is a classical topic in statistics with broad applicability and has long served as an important theoretical device \citep{Karoui_loo_2, maleki_2}. In this section, we contribute to this literature by developing high-dimensional analogues of the leave-$k$-out formulas for the OLS interpolator. To do so, we partition the rows of $\bX$ into two disjoint subsets. Since the rows correspond to samples, this partition induces two subsamples of the data.

Formally, let $\cI \subseteq [n]$ be a nonempty index set, and denote its complement by $\cI^c = [n] \setminus \cI$. 
Our objective is to investigate the minimum $\ell_2$-norm OLS estimator based on the subsampled data $\{(\bx_i, y_i): i \in \cI \}$: 
\begin{align} \label{eq:loo.model.beta.general}
    \hbbeta^{(\cI,\star)} \coloneqq \olsmin(\bX_{\cI,\star}, \by_{\cI}) = \bX_{\cI,\star}^{\dagger} \by_{\cI}. 
\end{align}
Note that the \textit{superscript} $(\cI, \star)$ indicates \eqref{eq:loo.model.beta.general} is determined from the submatrix $\bX_{\cI, \star}$; this is to avoid any confusion from \textit{subscripts}, which indicate subvectors. 

\subsection{Row-subsampled OLS}\label{sec:results.rows} 
Our first primary result relates the OLS estimates based on the subsampled data $(\bX_{\cI, \star}, \by_{\cI})$ to the OLS estimate based on the full data $(\bX, \by)$. 

\begin{theorem}\label{thm:subsampled_rows} \
    Let $\cI \subseteq [n]$ be a nonempty set.
    \begin{itemize} 
    	
    	\item[(a)] \emph{Classical} $(n > p)$: If Assumption~\ref{assump:column_rank} holds, then  
        \begin{align} 
            \hbbeta^{(\cI,\star)} 
                = \hbbeta - \big( \bX^{\dagger} \big)_{\star, \cI^c} \cdot \left\{ \big( \bP_{\bX}^{\perp} \big)_{\cI^c, \cI^c} \right\}^{-1} \cdot \hbvarepsilon_{\cI^c}, 
                \label{eq:beta.classical}
        \end{align} 
        where $\hbvarepsilon = \by - \bX \hbbeta$ denotes the in-sample residual vector.
    
    	\item[(b)] \emph{High-dimensional} $(n \le p)$: If Assumption \ref{assump:row_rank} holds, then 
        \begin{align}  
            \hbbeta^{(\cI,\star)} 
                &= (\bX_{\cI,\star})^{\dagger} \bX_{\cI,\star} \cdot \hbbeta \label{eq:beta.hd}
                \\
                &= \left\{ \bI_p - \bX^{\dagger}_{\star,\cI^c} (\bX^{\dagger}_{\star,\cI^c} )^{\dagger} \right\} \cdot \hbbeta. \label{eq:beta.hd.2}
        \end{align}
    \end{itemize}
\end{theorem} 
Theorem~\ref{thm:subsampled_rows} establishes a relationship between the minimum $\ell_2$-norm OLS estimator derived from the subsample $\cI$ and the full-sample. 
For the classical regime, \eqref{eq:beta.classical} is well-known but statisticians may find the following equivalent form more familiar:
\begin{equation} \label{eq:beta.classical.2}
    \hbbeta^{(\cI,\star)} 
        = \hbbeta - ( \bX^{\top} \bX )^{-1} (\bX_{\cI^c, \star})^{\top} \cdot \left\{ \bI_{|\cI^c|} - \bX_{\cI^c, \star} ( \bX^{\top} \bX )^{-1} (\bX_{\cI^c,\star})^{\top} \right\}^{-1} \cdot \hbvarepsilon_{\cI^c},
\end{equation} 
where $\bX_{\cI^c, \star} ( \bX^{\top} \bX )^{-1} (\bX_{\cI^c,\star})^{\top} = \bH_{\cI^c, \cI^c}$ is the principal submatrix of the hat matrix $\bH \coloneqq \bP_{\bX} = \bX ( \bX^{\top}\bX )^{-1} \bX^{\top}$ corresponding to $\cI^c$. 
The expression in \eqref{eq:beta.classical.2} is equivalent to \eqref{eq:beta.classical} since 
\[
    \big( \bX^{\top} \bX \big)^{-1} (\bX_{\cI^c, \star})^{\top} = \big[ ( \bX^{\top} \bX )^{-1} \bX \big]_{\star, \cI^c},
    ~~
    \bI_{|\cI^c|} - \bH_{\cI^c, \cI^c} = ( \bI_n - \bP_{\bX} )_{\cI^c, \cI^c} = \big( \bP_{\bX}^{\perp} \big)_{\cI^c, \cI^c}.
\]

For the high-dimensional regime, \eqref{eq:beta.hd} reveals that the row-subsampled OLS interpolator derived from $\cI$ can be computed by simply taking an orthogonal projection of that from the full-sample onto the subspace spanned by the subsampled rows. 
Equation \eqref{eq:beta.hd.2} offers a complementary representation of the OLS interpolator in terms of $\cI^c$. This characterization is particularly relevant for settings such as the leave-$k$-out configuration, which we analyze next.

\begin{remark}[Hat matrix]\label{rem:hat_matrix}
    The projection matrix $\bH = \bP_{\bX}$ is called the ``hat matrix'' because $\widehat{\by} = \bX \hbbeta = \bH \by$. 
    In the classical regime under Assumption \ref{assump:column_rank}, $\bH = \bX ( \bX^{\top}\bX )^{-1} \bX^{\top}$ and the diagonal entries $H_{ii}$ satisfy $0 \leq H_{ii} \leq 1$ and $\sum_{i=1}^n H_{ii} = \tr (\bH) = p < n$ because $\bH$ is an idempotent matrix with rank $p$. 
    In the high-dimensional regime under Assumption \ref{assump:row_rank}, $\bH = \bI_n$ and $\widehat{\by} = \by$.  
\end{remark}

\subsection{Leave-one-out configuration} \label{sec:loo.results} 
We focus on an important, specialized case where $\cI = \{i\}^c = [n] \setminus \{i\}$ is the entire population except one element $i \in [n]$.
We refer to this case as the leave-one-out configuration.
To avoid cluttered notation, let $\bX_{\sim i} \coloneqq \bX_{\{i\}^c,\star} \in \RR^{(n-1) \times p}$ and $\by_{\sim i} = \by_{\{i\}^c} \in \Rb^{n-1}$ denote the leave-$i$-out data. 
We denote the corresponding OLS estimator after leaving out the $i$-th datapoint as 
\begin{align} \label{eq:loo.model.beta.i}
    \hbbeta^{(\sim i)} \coloneqq \olsmin(\bX_{\sim i}, \by_{\sim i}) = \bX_{\sim i}^\dagger \by_{\sim i}.
\end{align}
We reemphasize that our use of the superscript $(\sim i)$ indicates that \eqref{eq:loo.model.beta.i} is derived from $\bX_{\sim i}$ and not a subvector of $\hbbeta$ linked to coordinates other than $i$.

\medskip 
\noindent 
\textit{Leave-one-out OLS formula.}
In Corollary \ref{cor:loo}, we quantify the gap between the leave-$i$-out and full-sample minimum $\ell_2$-norm OLS estimators in the classical and high-dimensional regimes. 
Recall $\bG_{\bX} = (\bX \bX^{\top})^{-1}$ is the ``inverse Gram matrix'' for the rows of $\bX$.

\begin{corollary} \label{cor:loo} 
Let $\cI = \{i\}^c$ for any $i \in [n]$. 
\begin{itemize} 
	
	\item[(a)] \emph{Classical} $(n > p)$: If Assumption~\ref{assump:column_rank} holds, then  
	\begin{align} 
		\hbbeta^{(\sim i)} 
            &= \hbbeta - \hvarepsilon_i \cdot \frac{(\bX^\top \bX)^{-1} \bx_i}{1 - \bx_i^\top (\bX^\top \bX)^{-1} \bx_i}, \label{eq:beta.loo.classical} 
	\end{align} 
    where $\hvarepsilon_i = y_i - \bx_i^\top \hbbeta$ denotes the $i$-th in-sample residual.

	\item[(b)] \emph{High-dimensional} $(n \le p)$: If Assumption~\ref{assump:row_rank} holds, then  
	\begin{align}
        \hbbeta^{(\sim i)}
            &= \left\{ \bI_p - \frac{\bX^{\dagger} \cdot \be_i \be_i^\top \cdot (\bX^{\dagger})^\top}{\be_i^\top \cdot \bG_{\bX} \cdot \be_i} \right\} \cdot \hbbeta  
            = \left\{ \bI_p - \frac{(\bX^\top \bX)^{\dagger} \bx_i \bx_i^\top (\bX^\top \bX)^\dagger}{\bx_i^\top [(\bX^\top \bX)^\dagger]^2 \bx_i} \right\} \cdot \hbbeta.  \label{eq:beta.loo.hd}
    \end{align}  
\end{itemize}
\end{corollary} 

We make two observations about Corollary \ref{cor:loo}.
First, under Assumption~\ref{assump:row_rank}, both the in-sample residuals $\hvarepsilon_i$ and denominator $1 - \bx_i^\top (\bX^\top \bX)^{-1} \bx_i$ in \eqref{eq:beta.loo.classical} take value zero. 
Thus, even if we define 0/0 = 0, we cannot hope to obtain an expression analogous to \eqref{eq:beta.loo.classical} in the high-dimensional $p>n$ regime.  
Instead, we obtain \eqref{eq:beta.loo.hd}, which expresses the leave-one-out solution as an orthogonal projection of the full-sample solution 
onto $\vspan( \bX^{\dagger}  \be_i )^{\perp}$, which is a subspace of co-dimension $1$ in $\RR^p$. 
The second formula of \eqref{eq:beta.loo.hd} follows by observing $\bX^{\dagger}  \be_i = (\bX^{\top} \bX)^{\dagger} \cdot \bx_i$.

Corollary~\ref{cor:loo} also fits into a broader literature relating leave-one-out estimators to their full-sample counterparts. Closely related work includes \cite{Karoui_loo_1} for ridge regression, \cite{Karoui_loo_2} for robust regression, \cite{maleki_2} for a broad class of regularized estimators through approximate leave-one-out risk analysis, \cite{maleki_1} for nonsmooth losses and regularizers, and \cite{sur_loo_1} for asymptotic analyses of maximum likelihood estimators. 
In this context, Corollary~\ref{cor:loo} contributes a formal leave-one-out characterization for the OLS interpolator. 
Notably, while its structure is closely related in spirit to analogous formulas for ridge regression, it necessarily differs because of the leverage-score behavior noted above.

\medskip \noindent \textit{Leave-one-out prediction residual formula.} 
%
Let $\tvarepsilon_i\coloneqq y_i - \bx_i^{\top} \hbbeta^{(\sim i)}$ represent the leave-$i$-out prediction residual, calculated as the difference between $y_i$ and the prediction made by the leave-$i$-out OLS model applied to $\bx_i$. 
%
This prediction residual is useful for understanding the generalization capabilities of the OLS interpolator and provides a potential foundation for variance estimation (Section~\ref{sec:var_estimation}). 
However, one immediate drawback of calculating the leave-one-out residuals is the necessity of estimating $n$ distinct leave-one-out models. 
To address this challenge, we leverage Corollary~\ref{cor:loo} to derive a ``shortcut'' formula that efficiently computes the $n$ leave-one-out prediction residuals, represented by the vector $\tbvarepsilon \in \Rb^n$.
%

\begin{corollary} \label{cor:loo.shortcut} 
Let $\bX \in \Rb^{n \times p}$ and $\by \in \Rb^n$. 
\begin{itemize} 
		
	\item[(a)] \emph{Classical} $(n > p)$: If Assumption~\ref{assump:column_rank} holds, then
	\begin{align}
		\tbvarepsilon &= \big[\diag(\bP_{\bX}^\perp)\big]^{-1} \cdot \bP_{\bX}^\perp  \by,  \label{eq:loo.shortcut.classical}
	\end{align} 
    	where we recall $\bP_{\bX}^\perp = \bI_n - \bX (\bX^\top \bX)^{-1} \bX^\top$. 
	
	\item[(b)] \emph{High-dimensional} $(n \le p)$: If Assumption~\ref{assump:row_rank} holds, then 
	\begin{align}
		\tbvarepsilon &= \big[ \diag(\Gram_{\bX}) \big]^{-1} \cdot \Gram_{\bX}  \by,  \label{eq:loo.shortcut.hd}
	\end{align}
    	where we recall $\Gram_{\bX} = (\bX \bX^\top)^{-1}$. 
\end{itemize}
\end{corollary} 

While Corollary~\ref{cor:loo} and our derivation of \eqref{eq:loo.shortcut.hd} based on Corollary~\ref{cor:loo} are
novel to the best of our knowledge, the formula in \eqref{eq:loo.shortcut.hd} itself has been previously discovered, e.g., \cite[Section 7.2]{hastie_22} derives \eqref{eq:loo.shortcut.hd} from the well-known ridge leave-one-out cross-validation formula. 

By \eqref{eq:loo.shortcut.classical} and \eqref{eq:loo.shortcut.hd}, the leave-one-out residuals in both the classical and high-dimensional regimes can be directly computed from $(\bX, \by)$ in a single shot without ever having to construct and assess the performance of $n$ distinct models. 
Moreover, \eqref{eq:loo.shortcut.classical} and \eqref{eq:loo.shortcut.hd} reveal that the expressions for the leave-one-out residuals take the same form in both regimes, albeit with different constructions of $\bP_{\bX}^\perp$ and $\Gram_{\bX}$. 
On this note, observe that in the classical regime, $\bP_{\bX}^\perp \by = (\bI_n - \bH) \by = \by - \widehat{\by}$ denotes the OLS estimator's in-sample residuals. 
Thus, there is a simple transformation between the OLS estimator's in-sample and leave-one-out residuals as in \eqref{eq:loo.shortcut.classical}. 
Such a relationship does not hold in high-dimensions as the in-sample residuals of the OLS interpolator are zero. 

\begin{remark}[Resemblance between two matrices]
    In the classical regime under Assumption~\ref{assump:column_rank}, if $H_{ii} < 1$ for all $i \in [n]$, then $\diag(\bP_{\bX}^\perp)$ is invertible and $( [\diag(\bP_{\bX}^\perp)]^{-1} \cdot \bP_{\bX}^\perp )_{ii} = 1$ for all $i \in [n]$. 
    Similarly, in the high-dimensional regime under Assumption~\ref{assump:row_rank}, then $\Gram_{\bX}$ is positive definite and $\diag(\Gram_{\bX})$ is invertible, which yields $( [ \diag(\Gram_{\bX}) ]^{-1} \cdot \Gram_{\bX} )_{ii} = 1$ for all $i \in [n]$. 
\end{remark} 

\begin{remark}[Applications of leave-one-out formulas]
For brevity, we summarize several concrete applications of the leave-one-out formulas in Table~\ref{tab:applications.rows} and relegate detailed descriptions of each application to the Supplementary Material. 
\end{remark} 

\begin{table} [t]
    \small
    \centering
    \caption{Summary of applications of leave-one-out residuals.}
    \label{tab:applications.rows}
    \begin{adjustbox}{max width=\textwidth}
    \begin{tabular}{l c c }
        \toprule
        Result  &   Section     &   Application/Utility   \\
        \midrule
        Corollary~\ref{cor:press}       &   Section~\ref{sec:app.loo}      &   Shortcut formula for the predicted residual error sum of squares statistic\\
        Corollary~\ref{cor:online}      &   Section~\ref{sec:app.loo}      &   Shortcut formula to update the OLS estimator $\hbbeta$ in online settings   \\
        Corollary~\ref{cor:loo.jackknife.1} & Section~\ref{sec:app.jackknife}      &   Connection between the leave-one-out residuals and the \jackknife~(in high-dim.)   \\
        Corollary~\ref{cor:loo.jackknife.2} &   Section~\ref{sec:app.jackknife}    &   \jackknife~variance estimator expressed in terms of leave-one-out residuals   \\
        Corollary~\ref{cor:loo.pred}    & Section~\ref{sec:app.jackknife}  &   Shortcut formula for prediction intervals with the \jackknifep~method of \cite{jackknife+}    \\
        \bottomrule
    \end{tabular}
    \end{adjustbox}
\end{table}

\section{Column-partitioned regression} \label{sec:subsampled_column}
\subsection{Setting}
This section complements the results in Section~\ref{sec:row_partition} by considering the leave-$k$-covariates-out configuration, which has also received considerable attention as both an important theoretical and practical tool \citep{Karoui_loo_2}. In particular, we develop high-dimensional analogues of the Frisch-Waugh-Lovell theorem and the omitted variable bias formula for the OLS interpolator. To this end, we partition the columns of $\bX$. Since each column of $\bX$ corresponds to a distinct covariate, this column partition naturally induces different subsets of covariates.

Formally, let $\cJ \subseteq [p]$ be a nonempty set and $\cJ^c$ its complement; we will imbue these sets with context-specific meanings below. 
In the classical regime, Assumption~\ref{assump:column_rank} implies that both $\bX_{\star,\cJ}$ and $\bX_{\star,\cJ^c}$ have full column rank for any $\cJ \subset [p]$ since the columns of $\bX$ are linearly independent. 
For the high-dimensional regime, we introduce an additional regularity assumption. 
\begin{assumption} \label{assump:partial} 
    For a nonempty set $\cJ \subseteq [p]$, $\bX_{\star,\cJ}$ has full row rank and $\bX_{\star,\cJ^c}$ has full column rank, i.e., $\rank(\bX_{\star,\cJ}) = n$ and $\rank(\bX_{\star,\cJ^c}) = |\cJ^c|$.
\end{assumption} 

Assumption~\ref{assump:partial} implies Assumption~\ref{assump:row_rank} and is therefore stronger; in particular, $\rank(\bX) \geq \rank(\bX_{\star,\cJ})$.
Next, we consider two closely related contexts of column partitioning. 
In Section \ref{sec:partial_regularization}, we study the OLS estimator that minimizes the norm of $\hbbeta$ confined to $\cJ$. 
In Section \ref{sec:subsampled_columns}, we analyze the OLS estimator based on the subset of covariates in $\cJ$ relative to the full covariate set. 
Before proceeding, we recall Remark~\ref{remark:hd_classical} concerning the accompaniment by classical results, which remains pertinent within this section. 

\subsection{Partially regularized OLS} \label{sec:partial_regularization}
We define the $\cJ$-partially regularized OLS estimator as 
\begin{align}\label{eqn:j_partial}
	\hbbeta^{[\cJ]} &\in \argmin_{\bbeta \in \cS} \| \bbeta_{\cJ} \|_2^2,
	~\text{ where }~ 
	\cS \coloneqq \argmin_{\bbeta \in \Rb^p} \| \by - \bX \bbeta \|_2^2. 
\end{align} 
When comparing the $\cJ$-partially regularized OLS estimator with the minimum $\ell_2$-norm OLS estimator in Definition \ref{defn:min_ols}, we see that the former solely minimizes the coefficients corresponding to those columns indexed by $\cJ$ whereas the latter minimizes all coefficients.
Clearly, if $\cJ = [p]$, the two estimators coincide. 
At the same time, if $\rank(\bX) \geq p$, then $\cS$ is a singleton set and the two estimators coincide regardless of $\cJ$. 
To motivate \eqref{eqn:j_partial}, we consider two canonical examples from causal inference. 

\begin{example}[Randomized experiments] \label{ex:rand}
Consider estimating the average treatment effect (ATE) in a randomized experiment. 
While the difference-in-means estimator is unbiased for the ATE, leveraging pre-treatment covariates can improve efficiency.  
The OLS regression adjustment method of \cite{lin_ols}, which resolves critiques by \cite{freedman2} on the classical analysis of covariance method of \cite{Fisher25}, is widely regarded as a best-practice \citep{wager2024causal}. 
Formally, let $\bX_{\star, \cJ}$ denote the pre-treatment covariates and $\bX_{\star, \cJ^c} = \bone$ the intercept vector. 
Lin's estimator (i) fits two separate OLS models for the treatment and control groups, and (ii) estimates the ATE as the difference of the intercepts, $\hbbeta_{\cJ^c}$.
Under the finite-population model (which abstains from assuming a hypothetical outcome generating process) with $n > p$, \cite{lin_ols} establishes that this approach is consistent, asymptotically normal, and more efficient than the difference-in-means estimator. 
Despite refinements and extensions \citep{lei2021regression, chang_ecma}, the behavior of Lin's estimator when $p > n$ remains unclear. 
In fact, Lin's approach is not even well-defined when $p > n$ as the OLS solution is no longer unique. 
A natural resolution is to minimize the $\ell_2$-norm of the covariate coefficients while leaving the intercept unregularized \citep{elements_of_stat_learning}, thereby avoiding dependence on the arbitrary choice of origin for the outcome vector. 
The ATE estimate then corresponds to the difference in the (unregularized) intercepts, $\hbbeta^{[\cJ]}_{\cJ^c}$, between the treatment and control groups. 
We explore this extension in a simulation study of Section~\ref{sec:sim.treatment}.
\end{example}

\begin{example}[Observational studies] \label{ex:obs}
Now consider ATE estimation in an observational study. 
Let $\bX_{\star, \cJ}$ again represent pre-treatment covariates, and define $\bX_{\star, \cJ^c} = [\bD, \bone]$, where $\bD$ is the treatment vector and $\bone$ the intercept. 
When $n > p$, the coefficient on the treatment indicator in a single OLS regression is a common ATE estimator \citep{Ding2024FirstCourse}. 
In the super-population framework, this coefficient provides a consistent, unbiased, and asymptotically normal estimate, provided unconfoundedness and exogeneity hold. 
When $p > n$, however, this estimator becomes ill-defined. 
Mirroring the logic of Example~\ref{ex:rand}, a natural remedy is to minimize the $\ell_2$-norm of the covariate coefficients while leaving the treatment indicator and intercept unregularized. 
The ATE is then estimated by the first element of $\hbbeta^{[\cJ]}_{\cJ^c}$. 
We study this extension through a simulation analysis in Section~\ref{sec:sim.treatment}. 
More generally, in the presence of endogeneity, an instrumental variables approach replaces the treatment indicator $\bD$ with its first-stage predicted values from a regression on instruments, covariates, and the intercept, yielding a two-stage least squares estimator that retains the same OLS-like structure.
\end{example}

Towards analyzing the settings above and other related applications, we provide the decompositional form for the $\cJ$-partially regularized OLS estimator into expressions for those coefficients of $\cJ$ and $\cJ^c$ separately.

\begin{theorem} [Frisch-Waugh-Lovell theorem] \label{thm:partially_regularized_OLS}
Let $\cJ \subseteq [p]$ be a nonempty set. 
Denote $\bW = \bX_{\star,\cJ}$ and $\bT = \bX_{\star,\cJ^c}$ as the \underline{w}ide and \underline{t}all sub-matrices of $\bX$, respectively.
\begin{itemize} 

	\item[(a)] \emph{Classical} $(n > p)$: If Assumption~\ref{assump:column_rank} holds, then 
	\begin{align}
		\hbbeta_{\cJ}^{[\cJ]} &= \big( \bP_{\bT}^{\perp} \bW \big)^{\dagger} \bP_{\bT}^{\perp} \by,
            \label{eq:fwl.j} \\
        \hbbeta_{\cJ^c}^{[\cJ]} &= \big( \bP_{\bW}^{\perp} \bT \big)^{\dagger} \bP_{\bW}^{\perp} \by.
            \label{eq:fwl.jc} 
	\end{align} 

	\item[(b)] \emph{High-dimensional} $(n \le p)$: If Assumption~\ref{assump:row_rank} holds, then 
	\begin{align}
        \hbbeta^{[\cJ]}_{\cJ} &= ( \bP_{\bT}^{\perp} \bW )^{\dagger} \bP_{\bT}^{\perp} \by.
		\label{eq:fwl.j.hd.0} 
    \end{align}
    If Assumption~\ref{assump:partial} additionally holds, then
    \begin{align}
        \hbbeta^{[\cJ]}_{\cJ^c} 
            &= \big( \bW^{\dagger} \bT \big)^{\dagger} \bW^{\dagger} \by.  
            \label{eq:fwl.jc.hd}
    \end{align}
\end{itemize}
\end{theorem} 
The expression in \eqref{eq:fwl.j}, established in the classical regime, is commonly recognized as the {\em Frisch-Waugh-Lovell} (FWL) theorem \citep{frisch_waugh, lovell1963seasonal} within econometrics. 
It is pertinent to remark that $\cS$ becomes a singleton in the classical regime under Assumption \ref{assump:column_rank}. 
Consequently, the roles of $\cJ$ and $\cJ^c$ are symmetric in this setting.

Considering \eqref{eq:fwl.j}, we regard \eqref{eq:fwl.j.hd.0} as its corresponding counterpart in the high-dimensional regime. 
As such, we proceed to discuss the interpretation of this expression and elucidate the difference between \eqref{eq:fwl.jc} and \eqref{eq:fwl.jc.hd}.
\begin{itemize} 
	\item[(i)] \textit{The FWL formula \eqref{eq:fwl.j.hd.0}}: 
	Interestingly, \eqref{eq:fwl.j.hd.0} takes the same form as \eqref{eq:fwl.j} and thus, inherits the same interpretation. 
	Specifically, \eqref{eq:fwl.j.hd.0} is computed by regressing $\bP_{\bT}^\perp \by$ on $\bP_{\bT}^\perp \bW$, where $\bP_{\bT}^\perp \by$ is the residual vector from the OLS fit of $\by$ on $\bT$ and $\bP_{\bT}^\perp \bW$ is the residual matrix from the column-wise OLS fit of $\bW$ on $\bT$. 
	In words, \eqref{eq:fwl.j.hd.0} measures the ``impact'' of $\bW$ on $\by$ after ``adjusting'' for the impact of $\bT$. 
	Additionally, we can further rewrite \eqref{eq:fwl.j.hd.0} as
	\begin{align}
		\hbbeta^{[\cJ]}_{\cJ} 
    		&= \left[ (\bP_{\bT}^\perp \bW)^\top (\bP_{\bT}^\perp \bW) \right]^\dagger (\bP_{\bT}^\perp \bW)^\top \cdot \bP_{\bT}^\perp \by
    		= (\bP_{\bT}^\perp \bW)^\dagger \by,
                \label{eq:fwl.j.hd.1} 
	\end{align}
	which is equivalent to $\olsmin(\bP_{\bT}^\perp \bW, \by)$.
    Recall $\hbbeta^{[\cJ]}_{\cJ}$ is obtained by regressing $\bP_{\bT}^\perp \by$ on $\bP_{\bT}^\perp \bW$ after ``residualization.'' 
    The expression in \eqref{eq:fwl.j.hd.1} suggests that it is not crucial to residualize $\by$ separately because it will be automatically accompanied by residualizing $\bW$.
	
	\item[(ii)] \textit{Extending the FWL formula \eqref{eq:fwl.jc.hd}}: 
	While \eqref{eq:fwl.j} and \eqref{eq:fwl.j.hd.0} admit the same formulation, it may no longer be possible to express $\hbbeta^{[\cJ]}_{\cJ^c}$ in the same form as in \eqref{eq:fwl.jc} in the high-dimensional setting. 
    This is due to asymmetry between $\cJ$ and $\cJ^c$ inducted by the partial regularization, as  evident from \eqref{eqn:j_partial}. 
    To address this, we present supplementary results by imposing a stronger assumption in Assumption \ref{assump:partial}, which implies Assumption \ref{assump:row_rank}. 
	Observe that the expressions for $\hbbeta^{[\cJ]}_{\cJ^c}$ in \eqref{eq:fwl.jc} and \eqref{eq:fwl.jc.hd} are different, which is to be expected.
	To see this, observe that Assumption~\ref{assump:partial} implies $\bP_{\bW} = \bI_n$, resulting in $\bP_{\bW}^\perp = \bzero$. 
	At the same time, $\colsp((\bW \bW^\top)^{-1} \bT)  \subseteq \colsp(\bW)$ whereas $\colsp(\bP_{\bW}^{\perp} \bT)  \subseteq \colsp( \bW )^{\perp}$.
    Instead, \eqref{eq:fwl.jc.hd} demonstrates that $\hbbeta^{[\cJ]}_{\cJ^c} = \olsmin(\bW^{\dagger} \bT, \bW^{\dagger} \by)$, i.e., $\hbbeta^{[\cJ]}_{\cJ^c}$ can be obtained by regressing $\bW^{\dagger} \by$ onto $\bW^{\dagger} \bT$. 
    Alternatively, using \(\bM^{\dagger} = (\bM^{\top} \bM)^{\dagger} \bM^{\top}\), we can rewrite \eqref{eq:fwl.jc.hd} as
    \begin{align}
        \hbbeta^{[\cJ]}_{\cJ^c}
            &= \left[  \big( \bW^{\dagger} \bT \big)^{\top}  \big( \bW^{\dagger} \bT \big) \right]^{\dagger} \big( \bW^{\dagger} \bT \big)^{\top} \bW^{\dagger} \by\\
            &=  ( \bT^\top \Gram_{\bW} \bT )^\dagger \bT^\top \Gram_{\bW} \by  \label{eq:fwl.jc.hd.gls}
     	    \\
            &= \left[\left( \Gram_{\bW} \bT \right)^{\top} \cdot \Gram_{\bW}^{-1} \cdot \left( \Gram_{\bW} \bT \right) \right]^\dagger \left( \Gram_{\bW} \bT \right)^{\top} \cdot \Gram_{\bW}^{-1} \cdot \Gram_{\bW} \by \label{eq:fwl.jc.hd.2}
            \\
            &= \left[ \left( \Gram_{\bW} \bT \right)^{\top} \cdot \Gram_{\bW}^{-1} \cdot \left( \Gram_{\bW} \bT \right) \right]^\dagger \left( \Gram_{\bW} \bT \right)^{\top} \cdot \by,
            \label{eq:fwl.jc.hd.3}
    \end{align}
    where $\Gram_{\bW} = (\bW \bW^{\top} )^{-1} = {\bW^{\dagger}}^{\top} \bW^{\dagger}$. 
    In this view, \eqref{eq:fwl.jc.hd.gls} can be interpreted as the generalized least squares (GLS) \citep{aitken1936iv} under a conditional noise variance of $\Gram_{\bW}^{-1}$. 
    Equivalently, \eqref{eq:fwl.jc.hd.2} also follows a GLS, where we instead regress the ``response'' $\Gram_{\bW} \by$ on the ``covariates'' $\Gram_{\bW} \bT$, assuming a conditional noise variance of $\Gram_{\bW}$. 
    As seen from Corollary~\ref{cor:loo.shortcut}, $\Gram_{\bW} \by$ and $\Gram_{\bW} \bT$ also represent the (scaled) leave-one-out residuals between $(\bW, \by)$ and $(\bW, \bT)$, respectively, and thus, play an analogous role to the in-sample residuals.  
    In turn, similar to \eqref{eq:fwl.j.hd.0}, we can view \eqref{eq:fwl.jc.hd.2} as measuring the impact of $\bT$ on $\by$ after adjusting for the impact of $\bW$. 
    Further, similar to \eqref{eq:fwl.j.hd.1}, \eqref{eq:fwl.jc.hd.3} reveals that it is crucial to residualize $\bT$ but not $\by$. 
\end{itemize} 

\begin{remark} [Alternative expression for \eqref{eq:fwl.j.hd.0}] \label{remark:fwl.extension}
Under Assumption \ref{assump:partial}, we can equivalently write $\hbbeta^{[\cJ]}_{\cJ} = \bP_{\bW^{\top}} \bP_{\bW^{\dagger} \bT}^{\perp} \bW^{\dagger} \by$, the proof of which can be found in the Supplementary Material. 
    Given that Assumption \ref{assump:partial} implies Assumption \ref{assump:row_rank}, this expression must be identical to that in \eqref{eq:fwl.j.hd.0} under Assumption \ref{assump:partial}, despite their visible dissimilarity. 
    That is,
    \begin{equation}\label{eqn:matrix_equivalence}
        \text{Assumption \ref{assump:partial}}
        \qquad\implies\qquad
        ( \bP_{\bT}^{\perp} \bW )^{\dagger} \bP_{\bT}^{\perp} = \bP_{\bW^{\top}} \bP_{\bW^{\dagger} \bT}^{\perp} \bW^{\dagger}.
    \end{equation}
    While \eqref{eqn:matrix_equivalence} may offer an intuitive geometric interpretation based on the subspaces associated to $\bW$ and $\bT$, we currently do not have a clear interpretation and record this fact for potential reference.
\end{remark}

\subsection{Column-subsampled OLS} \label{sec:subsampled_columns}
We now study the relationship between the OLS solution based on the full covariate set and the OLS solution based on a column-subset of $\bX$. 
Consider the three solution sets below: 
\begin{align}
	&\cS_1 \coloneqq \argmin_{\bbeta \in \Rb^p} \| \by - \bX \bbeta \|_2^2, 
	\\
	&\cS_2 \coloneqq \argmin_{\balpha \in \Rb^{|\cJ|}} \| \by - \bX_{\star, \cJ} \balpha \|_2^2, 
	\\
	&\cS_3 \coloneqq \argmin_{\bDelta \in \Rb^{|\cJ| \times |\cJ^c|}} \| \bX_{\star, \cJ^c} - \bX_{\star, \cJ} \bDelta \|_F^2.  \label{eq:column.s1} 
\end{align} 
Note that $\cS_1$ in \eqref{eq:column.s1} precisely corresponds to the set of OLS solutions $\cS$ in Definition~\ref{defn:min_ols}. 
To contextualize the solution sets in \eqref{eq:column.s1}, consider the setting where $\bX$ is only partially observed; namely, $\bX_{\star, \cJ}$ is observed while $\bX_{\star, \cJ^c}$ is unobserved.  
As a result, $\cS_1$ represents the ideal OLS fit that controls for all covariates; by contrast, $\cS_2$ represents the restricted model that a researcher is forced to fit using only the covariates $\bX_{\star, \cJ}$ that are accessible. 

The following result quantifies the gap between the ideal and restricted models by establishing an algebraic connection between the solution sets $\cS_1$, $\cS_2$, and $\cS_3$. 
Recall that for a vector $\bv \in \mathbb{R}^p$ and a set $\cJ \subseteq [p]$, we denote $\bv_{\cJ}$ as the subvector of $\bv$ containing coordinates with indices exclusively within $\cJ$, and $\bv_{\cJ^c}$ as the subvector containing coordinates with indices within $\cJ^c$.

\begin{theorem}[Omitted variable bias formula] \label{thm:sub_column_OLS} 
Let $\cJ \subseteq [p]$ be a nonempty set. 
\begin{itemize} 

	\item[(a)] \emph{Classical} $(n > p)$: If Assumption~\ref{assump:column_rank} holds, then for any $\hbbeta \in \cS_1$, $\hbalpha \in \cS_2$, $\hbDelta \in \cS_3$, 
	\begin{align}
		\hbalpha &= \hbbeta_{\cJ} + \hbDelta \hbbeta_{\cJ^c}. \label{eq:cochran.classical}
	\end{align} 
	
	\item[(b)] \emph{High-dimensional} $(n \le p)$: If Assumption~\ref{assump:partial} holds, then for any $\hbbeta \in \cS_1$, $\hbalpha \in \cS_2$, $\hbDelta \in \cS_3$, 
	\begin{align}
    		\bX_{\star, \cJ} \hbalpha = \bX_{\star, \cJ} ( \hbbeta_{\cJ} + \hbDelta \hbbeta_{\cJ^c} ). \label{eq:cochran.hd}
        \end{align} 
    	If $\hbbeta \in \cS_1$, $\hbalpha \in \cS_2$, $\hbDelta \in \cS_3$ are each the unique minimum $\ell_2$-norm solutions, then  
          \begin{align}
		\hbalpha &= \hbbeta_{\cJ} + \hbDelta \hbbeta_{\cJ^c}. \label{eq:cochran.hd.2}
	\end{align} 	
		
\end{itemize}
\end{theorem} 
The result presented in \eqref{eq:cochran.classical}, stated in the classical regime, is commonly referred 
as {\em Cochran's formula} \citep{cox_cochran} in statistics
and the {\em omitted variable bias formula} \citep{angrist2008mostly} in econometrics since it quantifies the bias in the OLS coefficient of $\bX_{\star, \cJ}$ that is incurred by omitting variables in $\bX_{\star,\cJ^c}$. 
Accordingly, \eqref{eq:cochran.hd} and \eqref{eq:cochran.hd.2} can be viewed as extensions of the omitted variable bias formula to the high-dimensional setting. 
We compare the two results. 

In the classical regime, the sets $\cS_1$ to $\cS_3$ are singletons, which implies that the solutions $(\hbalpha, \hbbeta, \hbDelta)$ are unique. 
By contrast, in the high-dimensional regime, each set contains infinitely many solutions, rendering \eqref{eq:cochran.hd} valid for infinitely many tuples. 
Nevertheless, these tuples differ solely in the nullspace of the subdesign matrix $\bX_{\star, \cJ}$, and thus, the expressions on both sides of \eqref{eq:cochran.hd} yield identical prediction values for all $(\hbalpha, \hbbeta, \hbDelta) \in \cS_2 \times \cS_1 \times \cS_3$. 
In this sense, \eqref{eq:cochran.hd} can be construed as the predictive counterpart of the omitted variable bias formula.
Even in the high-dimensional setting, the classical omitted variable bias relationship can be restored by imposing minimum $\ell_2$-norm constraints on the solutions in $\cS_1$, $\cS_2$, and $\cS_3$. 
Restricting attention to these minimum-norm solutions yields \eqref{eq:cochran.hd.2}, which matches the classical formula \eqref{eq:cochran.classical}. 

Next, consider a leave-one-covariate-out configuration (a special case of \eqref{eq:column.s1} with $\cJ^c$ corresponding to a single column), mirroring the leave-one-sample-out configuration of Section~\ref{sec:loo.results}. 
In this setting, \eqref{eq:cochran.hd.2} provides an explicit formula for the leave-one-covariate-out regression coefficient $\hbalpha$ in terms of the full regression model $\hbbeta$. 
The following result offers a complementary, implicit characterization of leave-one-covariate-out models. 
\begin{corollary} [Leave-one-covariate-out] \label{cor:loco}
Consider the high-dimensional setting $(n \le p)$. 
For any $j \in [p]$, let $\cJ = \{j\}^c = [p] \setminus \{j\}$ and $\cJ^c = \{j \}$. 
Denote $\hbalpha^{(\sim j)} \in \cS_2$ as the minimum $\ell_2$-norm OLS solution and  $\hbbeta^{(\sim j)} \in \Rb^p$ as its zero-padded version: for each entry $\ell \in [p]$, 
\begin{equation} \label{eq:ovb.special.append}
	\hbeta^{(\sim j)}_\ell 
        = \begin{cases}
		  \halpha^{(\sim j)}_\ell, & \text{if } \ell < j,\\
            0, & \text{if } \ell = j,\\
    		\halpha^{(\sim j)}_{\ell-1}, & \text{if } \ell > j. 
	\end{cases} 
\end{equation} 
If $\hbbeta \in \cS_1$ is the minimum $\ell_2$-norm OLS solution, then 
\begin{align}
	\hbbeta = \sum_{j = 1}^p \lambda_j \cdot \hbbeta^{(\sim j)}, 
\end{align}
where $\lambda_j = (p-n)^{-1} (1-h_{jj})$ with $h_{j \ell} = \big[\bX^\top (\bX \bX^\top)^{-1} \bX \big]_{j \ell}$ for any $j, \ell \in [p]$. 
\end{corollary}
Corollary~\ref{cor:loco} states that the full regression model can be written as a convex combination of leave-one-covariate-out coefficients since $\lambda_j \ge 0$ and $\sum_{j=1}^p \lambda_j = 1$. 
In the Supplementary Material, we demonstrate that Corollary~\ref{cor:loco}, attributed to \citet[Proposition 1]{spiess2023double}, can be established via Theorem~\ref{thm:sub_column_OLS}. 


\section{Statistical results} \label{sec:stat_inf} 

\subsection{Setting}
Thus far, we have not imposed any stochastic assumptions, and all preceding results have been purely algebraic in nature. 
In this section, we investigate statistical properties of the minimum $\ell_2$-norm OLS estimator under a particular stochastic model. 
Within this framework, we extend the Gauss-Markov theorem from the classical regime to the high-dimensional regime, and then leverage the leave-one-out results from Section~\ref{sec:row_partition} to furnish a homoskedastic variance estimator for the interpolating regime, where the in-sample residuals vanish. 

\medskip
\noindent \textit{Stochastic assumptions.} 
We postulate the Gauss-Markov model \citep{aitken1936iv}.

\begin{assumption} \label{assump:lm} 
    The response vector is generated by $\by = \bX \bbeta + \bvarepsilon$, 
	where $\bX$ and $\bbeta$ are fixed, and $\bvarepsilon$ is a random vector with $\bbE[\bvarepsilon] = \bzero$ and $\Cov(\bvarepsilon) = \bSigma$. 
\end{assumption} 

\noindent \textit{Identification.} 
Throughout this section, we define 
\begin{align}
    \bbeta^* = \bP_{\bX^\top} \bbeta = \bX^\dagger \bX \bbeta \label{eq:beta.start}
\end{align}
as the orthogonal projection of $\bbeta$ onto $\rowsp(\bX)$, which corresponds to the minimum $\ell_2$-norm solution to the equation $\bbE[\by] = \bX \bbeta$ under Assumption \ref{assump:lm}. 
In the classical regime under Assumption \ref{assump:column_rank}, $\bX^{\dagger} \bX = \bI_p$, yielding $\bbeta^* = \bbeta$. 
However, in the high-dimensional regime, the $p$-dimensional vector $\bbeta$ cannot be uniquely determined from $n$ measurements. 
Instead, we can identify a subspace $\cV_{\bbeta} \subseteq \RR^p$ of least squares solutions with codimension $p-n$ that contains $\bbeta$ and is orthogonal to $\rowsp(\bX)$. 
Accordingly, rather than attempting to recover $\bbeta$ itself, we focus on the more realistic target
$\bbeta^* = \argmin_{\bbeta' \in \cV_{\bbeta}} \| \bbeta' \|_2$ 
as defined in \eqref{eq:beta.start}. 
This perspective is similar in spirit to the treatment of ridge regression in \cite{shao2012estimation} and \cite{zhang_ridge_1, zhang_ridge_2}.

\subsection{Some general results} \label{sec:stat.results}
The next proposition presents an elementary result on the first two moments of $\hbbeta$.

\begin{proposition} \label{prop:linear_model} 
If Assumption~\ref{assump:lm} holds, then the following statements are true.
\begin{itemize}

	\item[(a)] \emph{Classical} $(n > p)$: $\bbE[\hbbeta] = \bbeta$ and $\nCov(\hbbeta) = \bX^{\dagger} \cdot \bSigma \cdot (\bX^{\dagger})^\top$. 
	\item[(b)] \emph{High-dimensional} $(n \le p)$: $\bbE[\hbbeta] = \bbeta^*$ and $\nCov(\hbbeta) = \bX^{\dagger} \cdot \bSigma \cdot (\bX^{\dagger})^\top$. 
\end{itemize}
\end{proposition} 

\begin{proof} 
Recall that $\hbbeta = \bX^{\dagger} \by$ and thus, $\hbbeta = \bX^{\dagger} \big( \bX \bbeta + \bvarepsilon \big) = \bbeta^* + \bX^{\dagger} \bvarepsilon $. 
 Then, it follows from Assumption~\ref{assump:lm} that $\bbE\big[ \hbbeta \big] = \bbeta^*$ and $\nCov(\hbbeta) = \bX^{\dagger} \cdot \bbE[ \bvarepsilon \bvarepsilon^\top ] \cdot(\bX^{\dagger})^\top = \bX^{\dagger} \cdot \bSigma \cdot (\bX^{\dagger})^\top$.
\end{proof} 

Given that the OLS minimum $\ell_2$-norm estimator admits the same expression in both regimes, as discussed in Section~\ref{sec:ols.interpolator}, it is not surprising that the first and second moments of $\hbbeta$ admit the same expression as well. 
Furthermore, we remark that $\tbbeta \in \RR^p$ is an unbiased estimator of $\bbeta^*$ if and only if $\tbbeta$ matches the predictive capabilities of $\bbeta$ with respect to $\bX$, as expressed by $\bbE[ \bX \tbbeta] = \bX \bbeta$. We formalize this observation in the following proposition.
\begin{proposition}\label{prop:identif_pred}
    Let $\bX \in \Rb^{n \times p}$ and $\bbeta \in \RR^p$. 
    Suppose either Assumption~\ref{assump:column_rank} or Assumption~\ref{assump:row_rank} holds. 
    For any estimator $\tbbeta$ of $\bbeta^* = \bX^{\dagger} \bX \bbeta$, the following are true:
    \begin{itemize}
        \item[(a)] 
        If $\bbE[ \tbbeta ] = \bbeta^*$, then $\bbE[ \bX \tbbeta ] = \bX \bbeta$.
     
        \item[(b)]
        Conversely, if $\bbE[ \bX \tbbeta ] = \bX \bbeta$, then $\bX^{\dagger} \bX \cdot \bbE[ \tbbeta ] = \bbeta^*$. 
    \end{itemize}
\end{proposition}

\begin{proof} 
    To prove part (a), note that $\bX^{\dagger} \bX = \bI_p$ under Assumption~\ref{assump:column_rank}, while $\bX \bX^{\dagger} = \bI_n$ under Assumption~\ref{assump:row_rank}. 
    By linearity of expectation, $\bbE[ \bX \tbbeta ] = \bX \bbE[ \tbbeta ] = \bX \bbeta^* = \bX \bX^{\dagger} \bX \bbeta = \bX \bbeta$. 
    To prove part (b), suppose that $\bbE[ \bX \tbbeta ] = \bX \bbeta$. 
    Then, $\bbE[\tbbeta] = \bX^{\dagger}\bX \bbE[\tbbeta] = \bX^{\dagger} \bbE[ \bX\tbbeta ] = \bX^{\dagger} \bX \bbeta = \bbeta^*$. 
\end{proof}

\subsection{The Gauss-Markov theorem} \label{sec:gauss.markov}

We now specialize to the homoskedastic setting, where $\bSigma = \sigma^2 \bI_n$. 
Below, we provide an analogous result of the Gauss-Markov theorem in the high-dimensional regime. 
Recall for symmetric matrices, $\bA \preceq \bB$ denotes $\bB - \bA$ is positive semidefinite.

\begin{theorem}[Gauss-Markov theorem]\label{thm:Gauss_Markov}
Let Assumption~\ref{assump:lm} hold with $\bSigma = \sigma^2 \bI_n$. 
    \begin{itemize}
        \item[(a)] \emph{Classical} $(n > p)$: 
    	Let Assumption~\ref{assump:column_rank} hold.
        If $\hbbeta = \olsmin(\bX, \by)$ and $\tbbeta$ is any unbiased estimator of $\bbeta$ that is linear in $\by$, then $\Cov( \hbbeta ) \preceq \Cov( \tbbeta )$.
     
        \item[(b)] \emph{High-dimensional} $(n \le p)$:
        Let Assumption~\ref{assump:row_rank} hold. 
        If $\hbbeta = \olsmin(\bX, \by)$ and $\tbbeta = \bL \by$ is any estimator of $\bbeta$ that is linear in $\by$ such that $\bbE[\bX \tbbeta] = \bX \bbeta$, then: 
        \begin{itemize}
            \item[(i)] 
            For every $\bv = \bX^{\top} \bc \in \rowsp(\bX)$, we have 
                $\sigma^2 \|\bc\|_2^2 = \bv^{\top} \Cov( \hbbeta ) \bv = \bv^{\top} \Cov( \tbbeta ) \bv$.
            
            \item[(ii)] 
            For every $\bw \in \rowsp(\bX)^{\perp}$, we have 
                $0 = \bw^{\top} \Cov( \hbbeta ) \bw \leq \bw^{\top} \Cov( \tbbeta ) \bw$,
            with equality for all such $\bw$ if and only if $\bL = \bX^\dagger$ (i.e., $\tbbeta = \hbbeta$).
        \end{itemize}
    \end{itemize}
\end{theorem}

Theorem \ref{thm:Gauss_Markov}(a) is the classical {\em Gauss-Markov theorem}, which asserts that $\bv^{\top} \Cov( \hbbeta ) \bv \leq \bv^{\top} \Cov( \tbbeta ) \bv$ for all unbiased linear estimators $\tbbeta$ and every direction $\bv \in \Rb^p$. 
That is, under homoskedasticity, OLS is the {\em best linear unbiased estimator}.

When $n \leq p$, $\bbeta$ is not point-identifiable from $\bX \bbeta$ due to rank deficiency, rendering unbiased estimation of $\bbeta$ infeasible without additional assumptions. 
Instead, we require estimators $\tbbeta$ that are ``unbiased through the lens of $\bX$,'' in the sense that $\bbE[ \bX \tbbeta ] = \bX \bbeta$. 
For linear estimators $\tbbeta = \bL \by$, this condition is equivalent to $\bX \bL = \bI_n$, since $\bbE[ \bX \tbbeta ] = \bX\bL \bbE[ \by ] = \bX\bL\bX \bbeta = \bX \bbeta$. 

Within this class of ``unbiased-through-$\bX$'' linear estimators, Theorem \ref{thm:Gauss_Markov}(b) yields a slightly weaker, ``directional''-analog of the classical Gauss-Markov conclusion: 
(i) On $\rowsp(\bX)$, the directional variance is invariant across all right-inverses: for any $\bv \in \rowsp(\bX)$, the directional variance $\bv^{\top} \Cov(\tbbeta) \bv$ is the same for all $\bL$ satisfying $\bX \bL = \bI_n$. 
(ii) On $\rowsp(\bX)^{\perp}$, $\hbbeta = \bX^{\dagger} \by$ uniquely minimizes directional variance, provided $\sigma > 0$. 
The variance-minimization on $\rowsp(\bX)^{\perp}$ connects to generalization and singles out the OLS interpolator $\hbbeta$ among all estimators of the form $\tbbeta = \bL \by$ with identical in-sample variances on $\rowsp(\bX)$. 
Thus, Theorem \ref{thm:Gauss_Markov}(b) states that $\hbbeta$ is ``best'' among all linear estimators $\tbbeta$ satisfying $\bbE[\bX \tbbeta] = \bX \bbeta$.

\begin{remark} [Limitation of the Gauss-Markov property when $n \le p$] \label{remark:counter}
    In the high-dimensional regime, the strong dominance result $\Cov(\hbbeta) \preceq \Cov(\tbbeta)$ does not hold in general. 
    As a simple counter-example, consider the case $n=1$ and $p=2$, with $\bX = [1, ~ 0 ]$ 
    such that $\bX^{\dagger} = \bX^{\top}$.
    Letting $\tbbeta_{\alpha} = (\bX^{\top} + \alpha \be_2) \by$ for $\alpha \in \RR$ and $\hbbeta = \bX^{\top} \by$, we can easily verify that 
    \begin{align} 
    	\Cov\big( \tbbeta_{\alpha} \big) = \begin{bmatrix} 1 & \alpha \\ \alpha & \alpha^2 \end{bmatrix}, \quad \Cov\big( \hbbeta \big) = \begin{bmatrix} 1 & 0 \\ 0 & 0 \end{bmatrix}.
    \end{align}  
    Observe that 
    \begin{align}
    	\Cov\big( \tbbeta_{\alpha} \big) - \Cov\big( \hbbeta \big) = \begin{bmatrix} 0 & \alpha \\ \alpha & \alpha^2 \end{bmatrix}
    \end{align}
    is not positive semidefinite unless $\alpha = 0$.
\end{remark}

\begin{remark} [On the role of homoskedasticity] \label{remark:gauss.marko.homo}
Strictly speaking, Theorem~\ref{thm:Gauss_Markov}(b) does not require homoskedasticity, while Theorem~\ref{thm:Gauss_Markov}(a) in the classical regime does. 
We clarify this distinction and provide a more general formulation of Theorem~\ref{thm:Gauss_Markov}(b) in the Supplementary Material. 
\end{remark}

\subsection{Homoskedastic variance estimation}\label{sec:var_estimation}
Continuing in the homoskedastic framework, we now turn to the problem of estimating $\sigma^2$, a fundamental quantity underlying inference in regression problems and many tasks in statistics and machine learning more generally. 
In the classical regime, the in-sample residuals are often central to constructing an unbiased estimator of $\sigma^2$:
$\hsigma^2_{\text{in}} = (n-p)^{-1} \cdot \| \bP_{\bX}^\perp \by \|_2^2$. 
However, this approach is no longer viable in high-dimensions since the OLS interpolator perfectly fits the training data, yielding $\bP_{\bX}^\perp \by = \bzero$. 
To overcome this limitation, we leverage leave-one-out residuals, which remain informative in both classical and overparameterized regimes, to construct a variance estimator for $\sigma^2$. 
See Section~\ref{sec:loo.results} for a refresher on the leave-one-out configuration. 

\begin{theorem} \label{thm:var.estimator} 
Let Assumption~\ref{assump:lm} hold with $\bSigma = \sigma^2 \bI_n$. 
\begin{itemize}
		
	\item[(a)] \emph{Classical} $(n > p)$: 
	Let Assumption~\ref{assump:column_rank} hold. 
	Define 
	\begin{align}
		\hsigma^2 &= \frac{ \left\| \tbvarepsilon \right\|_2^2 }{ \big\| \big[ \diag(\bP_{\bX}^\perp)\big]^{-1} \cdot \bP_{\bX}^\perp \big\|_{\F}^2 }, \label{eq:var.estimator.classical} 
	\end{align}
	where $\bP_{\bX}^\perp$ and $\tbvarepsilon$ are defined as in \eqref{eq:loo.shortcut.classical}.	
	Then,
	\begin{align}
    		\bbE[\hsigma^2] = \sigma^2. \label{eq:bias.classical} 
    	\end{align} 
	
	\item[(b)] \emph{High-dimensional} $(n \le p)$: 
	Let Assumption~\ref{assump:row_rank} hold. 
	Define 
	\begin{align}
		\hsigma^2 &= \frac{ \left\| \tbvarepsilon \right\|_2^2 }{ \big\| \big[ \diag(\Gram_{\bX}) \big]^{-1} \cdot \Gram_{\bX} \big\|_{\F}^2 }, \label{eq:var.estimator.hd} 
	\end{align}
	where $\Gram_{\bX}$ and $\tbvarepsilon$ are defined as in \eqref{eq:loo.shortcut.hd}. 
	Then,
	\begin{align}
        \bbE[\hsigma^2] = \sigma^2 + \frac{ \big\| \bbE[\tbvarepsilon] \big\|_2^2}{ \big\| \big[ \diag(\Gram_{\bX}) \big]^{-1} \cdot \Gram_{\bX} \big\|_{\F}^2 }. \label{eq:bias.hd} 
    \end{align} 
\end{itemize}
\end{theorem} 
Theorem~\ref{thm:var.estimator} shows that the leave-one-out residuals yield an unbiased estimator of $\sigma^2$ in the classical regime. 
To our knowledge, the estimator in \eqref{eq:var.estimator.classical} has not been explicitly documented in the literature; here, it arises naturally by re‑expressing the leave-one-out shortcut in a single quadratic form whose expectation can be evaluated exactly. 
In the high-dimensional regime, the analogous estimator in \eqref{eq:var.estimator.hd} remains well-defined but is conservative, exhibiting an upward bias: 
\begin{align}
    \frac{ \left\| \bbE[\tbvarepsilon] \right\|_2^2 }{ \left\| [ \diag(\Gram_{\bX}) ]^{-1} \cdot \Gram_{\bX} \right\|_{\F}^{2} },
    \qquad\text{where}\qquad
    \bbE[\tbvarepsilon] = \big[ \diag(\Gram_{\bX}) \big]^{-1} \cdot \Gram_{\bX} \cdot \bX \bbeta^*.  \label{eq:homo.var.bias}
\end{align}
The magnitude of this bias depends on the interplay between the design $\bX$ and the signal $\bbeta$, and vanishes under suitable structural assumptions. 
Consequently, \eqref{eq:var.estimator.hd} may provide a principled and robust means of estimating $\sigma^2$ even under perfect interpolation, potentially making it a natural candidate for valid inference in overparameterized settings such as those described in Example~\ref{ex:obs}. 
For completeness, we note that \eqref{eq:bias.classical} can also be expressed in the same form. 
However, in the classical regime, we have 
$\bbE[\tbvarepsilon] 
        = \big[\diag(\bP_{\bX}^\perp)\big]^{-1} \cdot (\bP_{\bX}^\perp) \cdot \bbE[\by]
        = \big[\diag^{-1}(\bP_{\bX}^\perp)\big]^{-1} \cdot \bP_{\bX}^\perp \cdot \bX \bbeta^*
        = \bzero,$
since $\bP_{\bX}^\perp \cdot \bX = \bzero$. 
Hence, the estimator in \eqref{eq:var.estimator.classical} is unbiased when $n > p$.

\begin{remark} [On some advantages of the leave-one-out approach] \label{rem:compare-naive} 
	Many conservative variance estimators are indeed possible. 
	For instance, the naive estimator $\tilde{\sigma}^2 = (1/n) \sum_{i=1}^n y_i^2$ also provides a conservative estimate of $\sigma^2$. 
    Under homoskedasticity with fixed $\bX$, 
    \[
        \bbE \left[\|\by\|_2^2\right]
            =\|\bX\bbeta\|_2^2 + n\sigma^2,
    \]
    so the upward bias equals $\|\bX\bbeta\|_2^2/n$. 
    This bias can be arbitrarily large and is insensitive to the design. 
    By contrast, $\hsigma^2$ in \eqref{eq:var.estimator.classical} and \eqref{eq:var.estimator.hd} has several desirable properties: 
    (i) an \emph{exact} finite‑sample expectation, unbiased in the classical regime and conservative in high‑dimensions with an explicit bias; 
    (ii) \emph{design adaptivity} through the leave-one-out geometry, since it removes the in-rowspace component via the operators $\bP_{\bX}^\perp$ or $\Gram_{\bX}$; 
    (iii) \emph{continuity across the interpolation threshold} as it reduces to an unbiased estimator when $n>p$ and remains well‑defined when $p\ge n$; 
    and (iv) \emph{closed‑form computation} via the leave-one-out shortcuts in \eqref{eq:loo.shortcut.classical}-\eqref{eq:loo.shortcut.hd}, without requiring refitting $n$ separate models. 
\end{remark}





\section{Simulation Studies} \label{sec:simulations.updated}

\subsection{Objectives} \label{sec:sim.objectives} 
In this section, we conduct simulation studies on (i) treatment effect estimation and (ii) homoskedastic variance estimation. 
For both set of studies, we consider different generative models for the covariates to study its role in our high-dimensional estimation tasks. 

\subsection{Covariate models} \label{sec:sim.covariates} 
Consider two generating processes for the covariate matrix $\bX$. 
\begin{enumerate} 

\item \emph{Standard normal model.} 
Let $\bX$ be a random matrix whose entries are i.i.d. draws from a standard normal distribution. 

\item \emph{Spiked model} \citep{spiked_model}. 
Let $\bX = \bU \bSigma^{1/2}$, where $\bU \in \Rb^{n \times p}$ is a random matrix with orthonormal rows and $\bSigma = \sigma_x^2 \cdot ( \bI_p + \sum_{\ell=1}^k \lambda_\ell \bv_\ell \bv_\ell^\top ) \in \Rb^{p \times p}$ with $\sigma_x \in \Rb$, $\lambda_\ell \gg 1$, and $v_\ell$ sampled uniformly at random from the unit sphere of dimension $p-1$. 

\end{enumerate} 

\subsection{Treatment Effect Estimation} \label{sec:sim.treatment}

{\em Randomized experiment.}  
We first consider the randomized experiment of Example~\ref{ex:rand}. In the classical regime, \cite{lin_ols} fits separate OLS regressions in the treatment and control groups and estimates the average treatment effect (ATE) as the difference between the intercepts. To extend this idea to $p > n$, we penalize the coefficients on the covariates while leaving the intercept unregularized. We refer to this estimator as \texttt{OLS(partial)}. A key advantage of our framework is that Theorem~\ref{thm:partially_regularized_OLS} yields a closed-form expression for this estimator through \eqref{eq:fwl.j.hd.0} and \eqref{eq:fwl.jc.hd}.

We compare \texttt{OLS(partial)} with three alternatives: the \texttt{difference-in-means} estimator, which remains the standard benchmark in randomized experiments; \texttt{OLS(full)}, which penalizes all coefficients including the intercept; and \texttt{ridge} regression, which regularizes only the covariate coefficients. Figure~\ref{fig:rct} reports the resulting estimation bias over 1000 simulation replications. Across both covariate models, \texttt{OLS(partial)} and \texttt{ridge} perform best. Both attain essentially the same bias as \texttt{difference-in-means} while exhibiting substantially smaller empirical variance. These gains are especially pronounced under the spiked design. By contrast, \texttt{OLS(full)} performs poorly, with particularly severe bias under the standard normal model. Table~\ref{table:ridge} further shows that the data-driven ridge tuning parameters are consistently close to zero, indicating that the selected \texttt{ridge} estimator is effectively indistinguishable from \texttt{OLS(partial)}.

{\em Observational study.} We next consider the observational study of Example~\ref{ex:obs}. The objective remains ATE estimation in the $p > n$ regime, but now under a different stochastic design. As in the randomized setting, we extend the canonical regression-adjustment by penalizing the coefficients on the covariates while leaving the treatment indicator and intercept unregularized. We again refer to this estimator as \texttt{OLS(partial)}, and obtain closed-form expressions from Theorem~\ref{thm:partially_regularized_OLS}.

We benchmark this estimator against \texttt{OLS(full)} and \texttt{ridge} regression. The results are reported in Figure~\ref{fig:obs}. The overall pattern closely mirrors that of the randomized experiment. \texttt{OLS(full)} exhibits substantial bias under both covariate models, whereas \texttt{OLS(partial)} and \texttt{ridge} remain nearly unbiased. Moreover, their efficiency gains are again most visible under the spiked design. As in the randomized setting, Table~\ref{table:ridge} shows that the \texttt{ridge} tuning parameters are selected as zero, so that ridge reduces to \texttt{OLS(partial)}.

{\em Takeaways.} Taken together, these two studies suggest that partial regularization provides a promising extension of classical regression adjustment to the high-dimensional regime. They also indicate that the performance of these procedures depends strongly on the structure of the covariates, an issue to which we return after studying variance estimation. Moreover, the fact that \texttt{ridge} consistently selects a near-zero penalty underscores its close connection with the OLS interpolator. Understanding this relationship more formally, for example through consistency and Gaussian approximation results in the spirit of \cite{zhang_ridge_1, zhang_ridge_2}, may help clarify when vanishing regularization is optimal. 

\begin{table} [t!]
    \centering
    \caption{Median \texttt{ridge} tuning parameter chosen via $5$-fold cross-validation.} 
    \begin{adjustbox}{max width=\textwidth}
    \begin{tabular} {@{}lcc@{}}
        \toprule
        						& Randomized experiment	& Observational study		\\
        \toprule
        Standard normal	& $8.50 \cdot 10^{-4}$ & 0			\\
        \hdashline
        Spiked		 	& $1.19 \cdot 10^{-5}$ & 0	\\ 
        \bottomrule
    \end{tabular}
    \end{adjustbox}
    \label{table:ridge}
\end{table} 

\begin{figure*}[t!]
	\centering  
        \subfloat[Standard normal model.]{\includegraphics[width=0.4\textwidth]{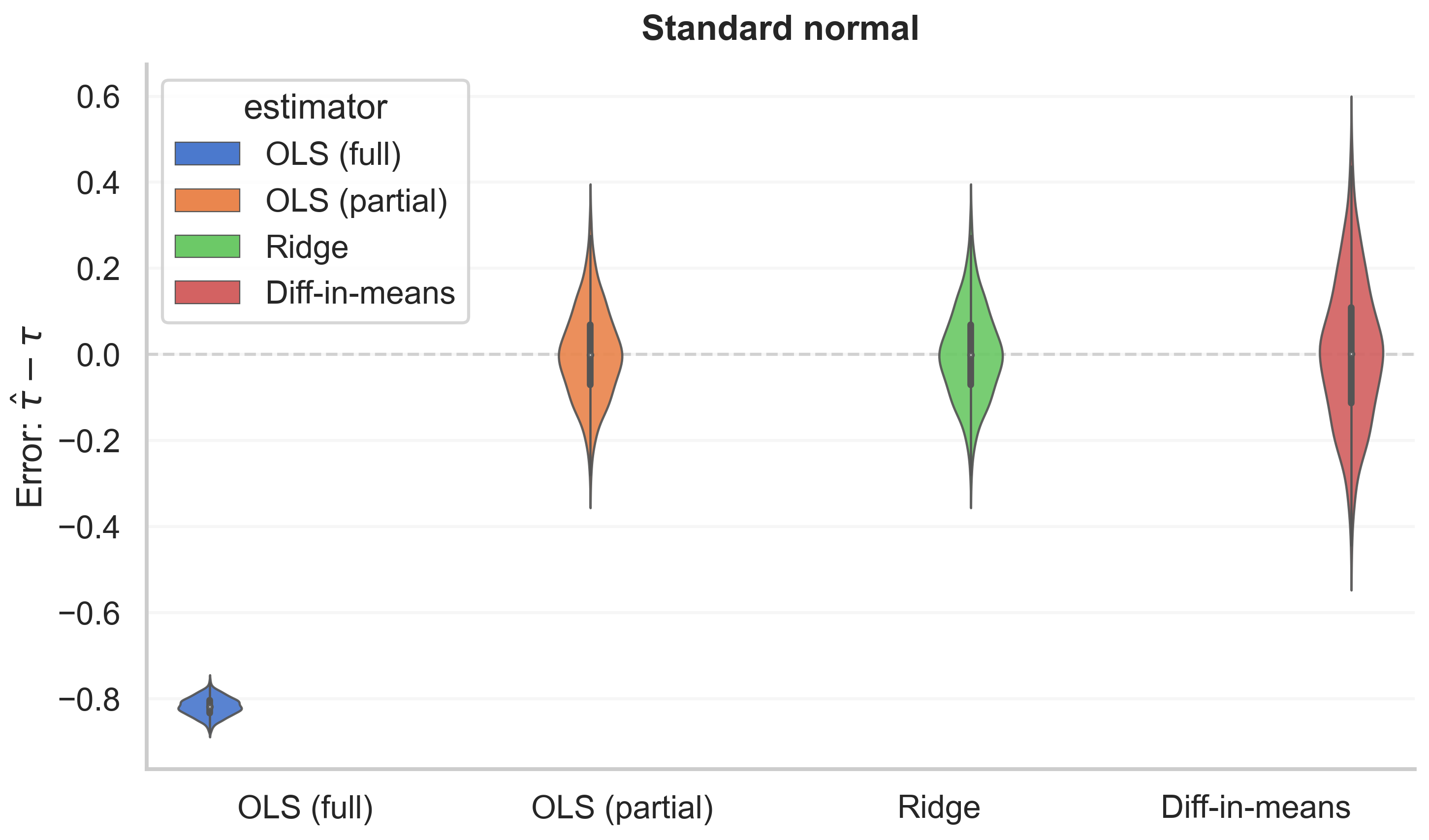}\label{fig:rct.normal}}
        \qquad \quad
        \subfloat[Spiked model.]{\includegraphics[width=0.4\textwidth]{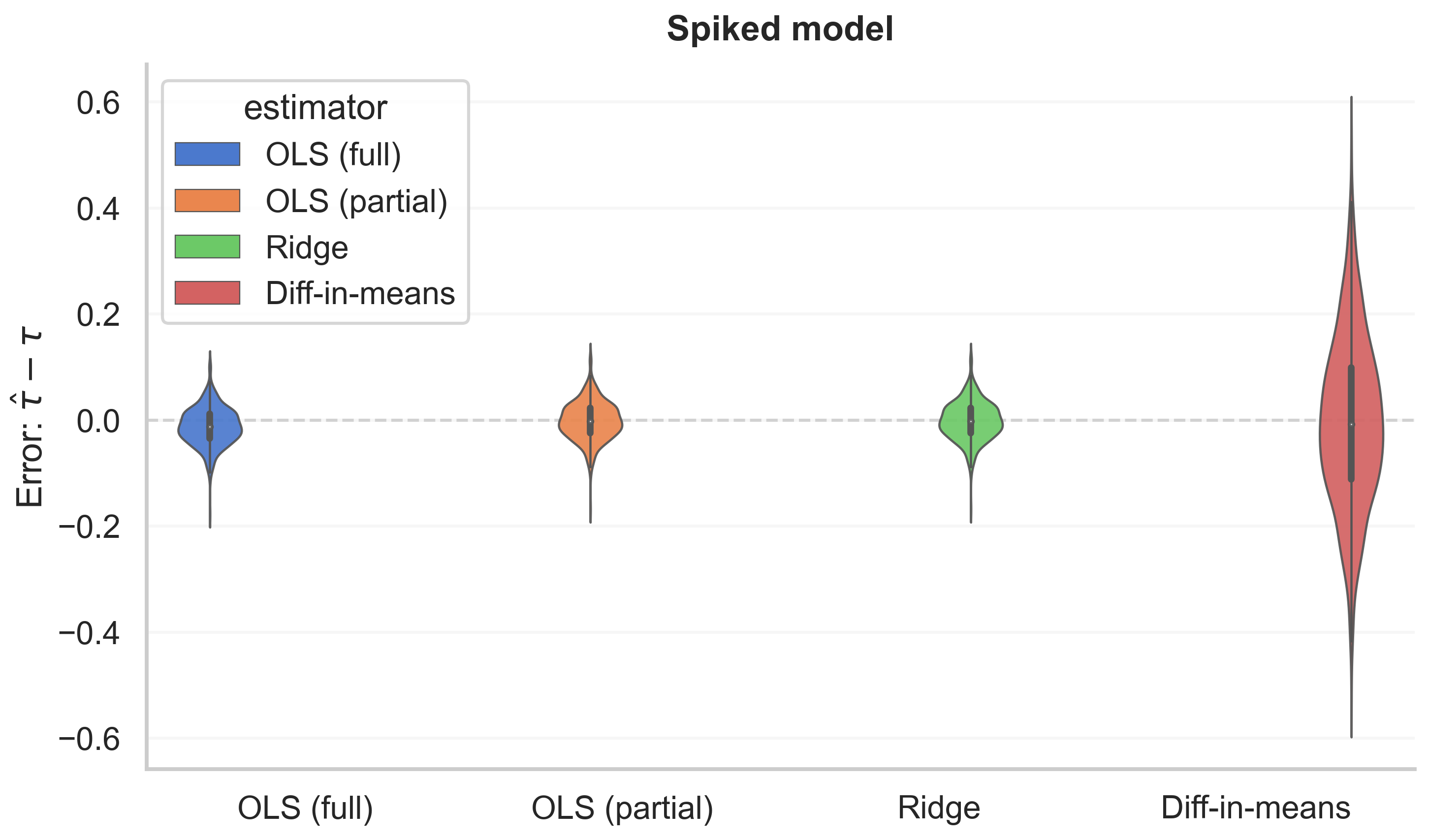}\label{fig:rct.spiked}} 
	\caption{Bias in ATE estimation under the randomized design (Example~\ref{ex:rand}): \texttt{OLS(partial)} and \texttt{ridge} remain nearly unbiased across both covariate models with clear efficiency gains over \texttt{difference-in-means}, particularly under the spiked design; by contrast, \texttt{OLS(full)} exhibits substantial bias, particularly under the standard normal design.}
	\label{fig:rct} 
\end{figure*}

\begin{figure*}[t!]
	\centering  
        \subfloat[Standard normal model.]{\includegraphics[width=0.4\textwidth]{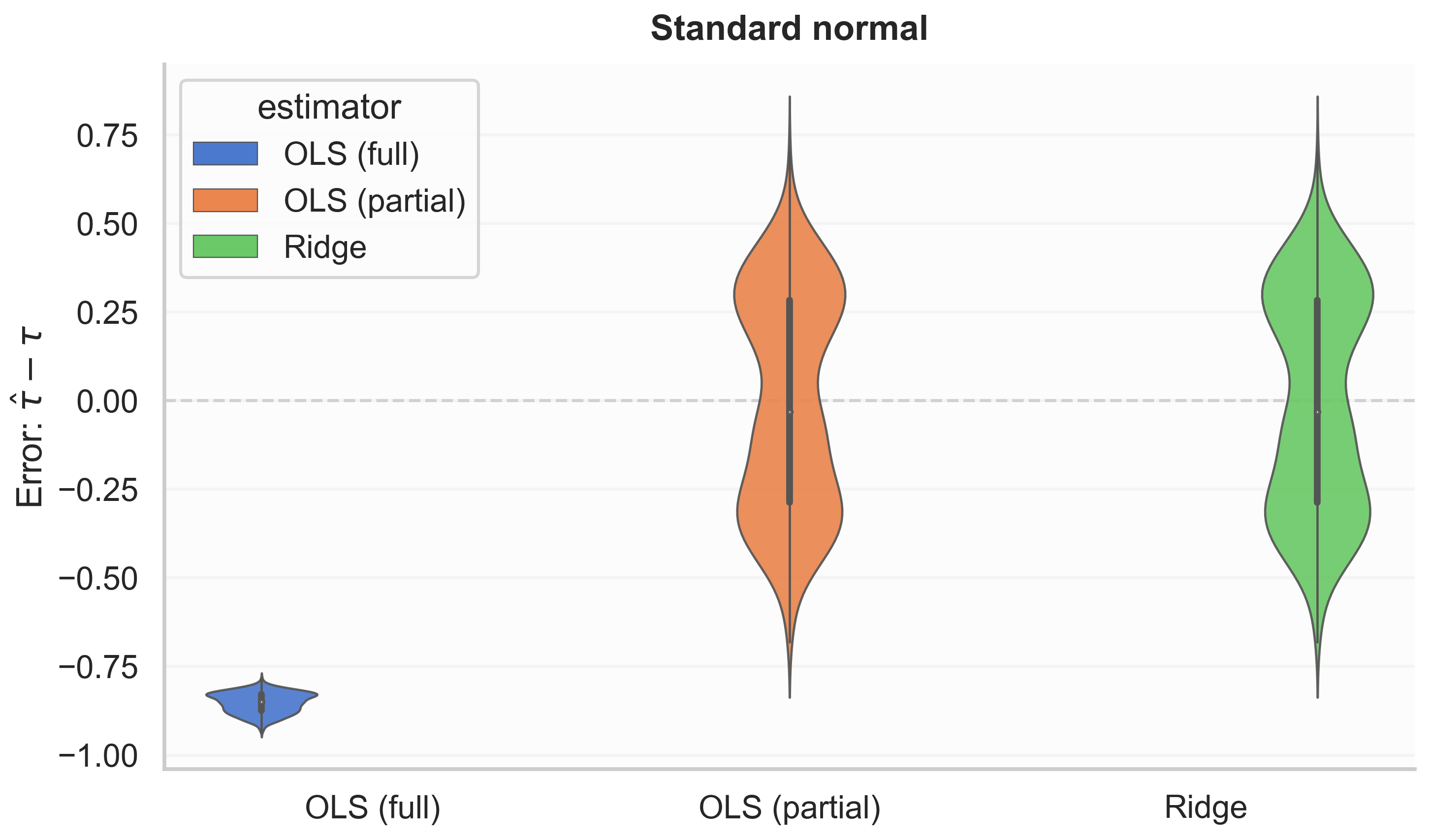}\label{fig:rct.normal}}
        \qquad \quad
        \subfloat[Spiked model.]{\includegraphics[width=0.4\textwidth]{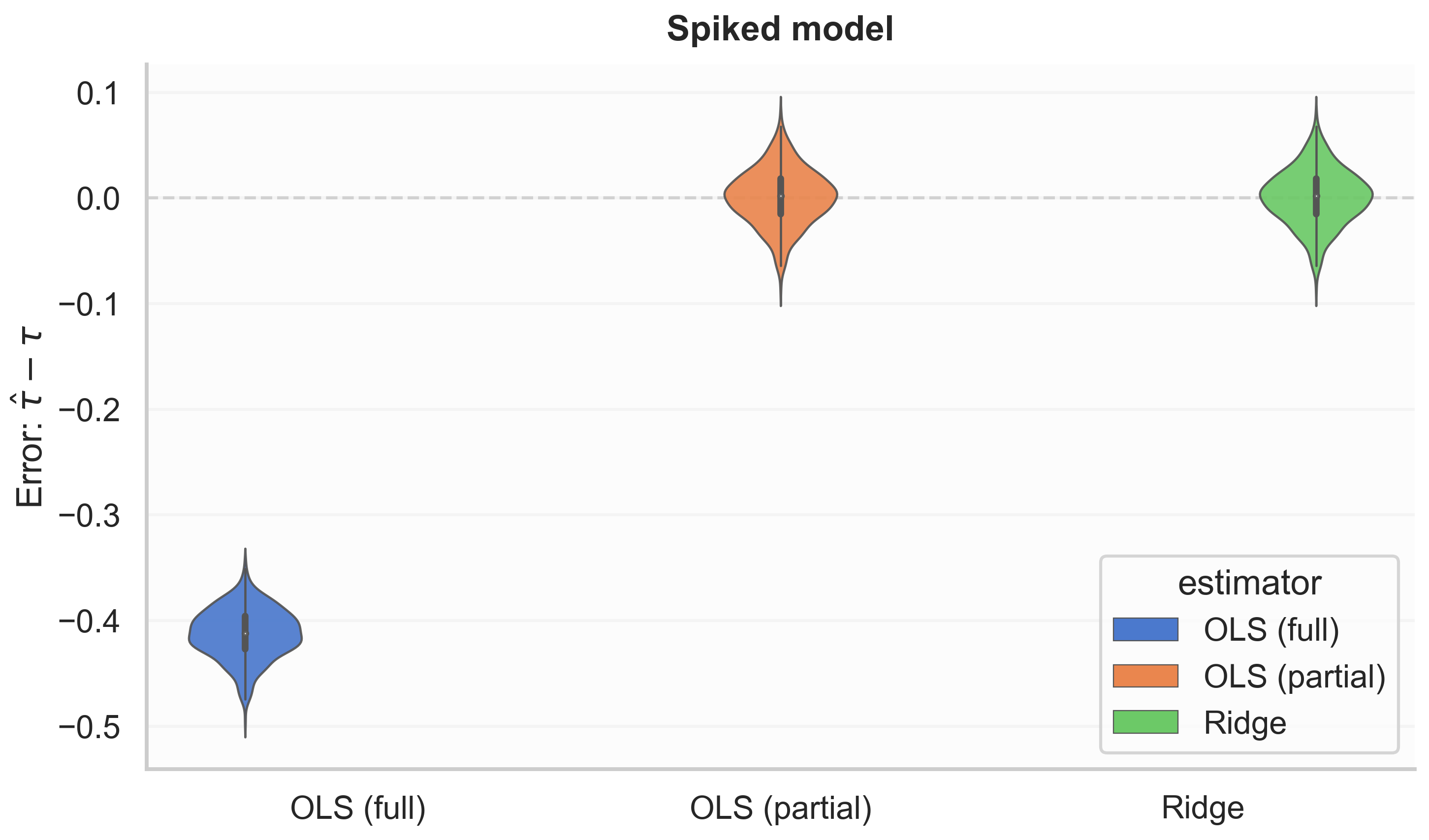}\label{fig:rct.spiked}} 
	\caption{Bias in ATE estimation under the observational design (Example~\ref{ex:obs}): \texttt{OLS(partial)} and \texttt{ridge} remain nearly unbiased with substantively smaller variances under the spiked design. Meanwhile, \texttt{OLS(full)} exhibits substantial bias across both covariate models.}
	\label{fig:obs} 
\end{figure*}

\subsection{Homoskedastic Variance Estimation} \label{sec:sim.variance}
We now study the finite-sample performance of the homoskedastic variance estimator developed in \eqref{eq:var.estimator.hd} of Theorem~\ref{thm:var.estimator}.
See Supplementary Material for details of the simulation setup.  

{\em Estimation bias.} 
We first examine the bias $\hsigma^2 - \sigma^2$ in a setting with fixed $p > n$ and increasing $n$. Figure~\ref{fig:bias} displays the resulting bias. Under the spiked model, the average bias remains close to zero throughout. Under the standard normal design, by contrast, the bias is consistently positive, indicating conservative behavior, which is in line with Theorem~\ref{thm:var.estimator}. The Supplementary Material reports analogous experiments under two additional regimes: fixed aspect ratio $p/n$ with (i) increasing $n$ and (ii) increasing $\sigma^2$. Across these settings, the same qualitative pattern persists. The bias appears to depend primarily on the design matrix $\bX$ and the parameter $\bbeta$, as suggested by \eqref{eq:bias.hd} and \eqref{eq:homo.var.bias}, with substantially more favorable performance under the spiked model.

\begin{figure*}[t!]
	\centering  
        \subfloat[Standard normal model.]{\includegraphics[width=0.35\textwidth]{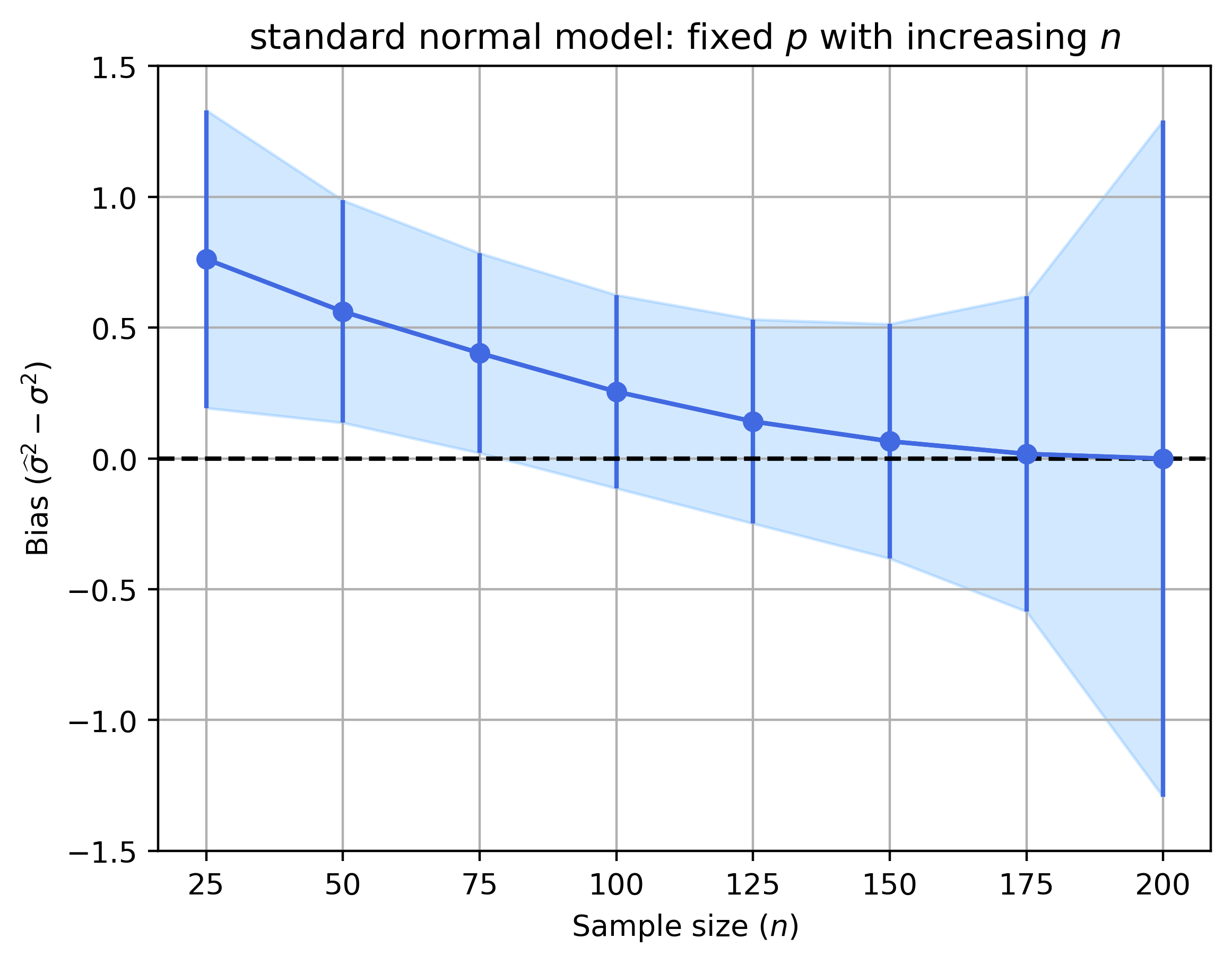}\label{fig:standard_normal}}
        \qquad \quad
        \subfloat[Spiked model.]{\includegraphics[width=0.35\textwidth]{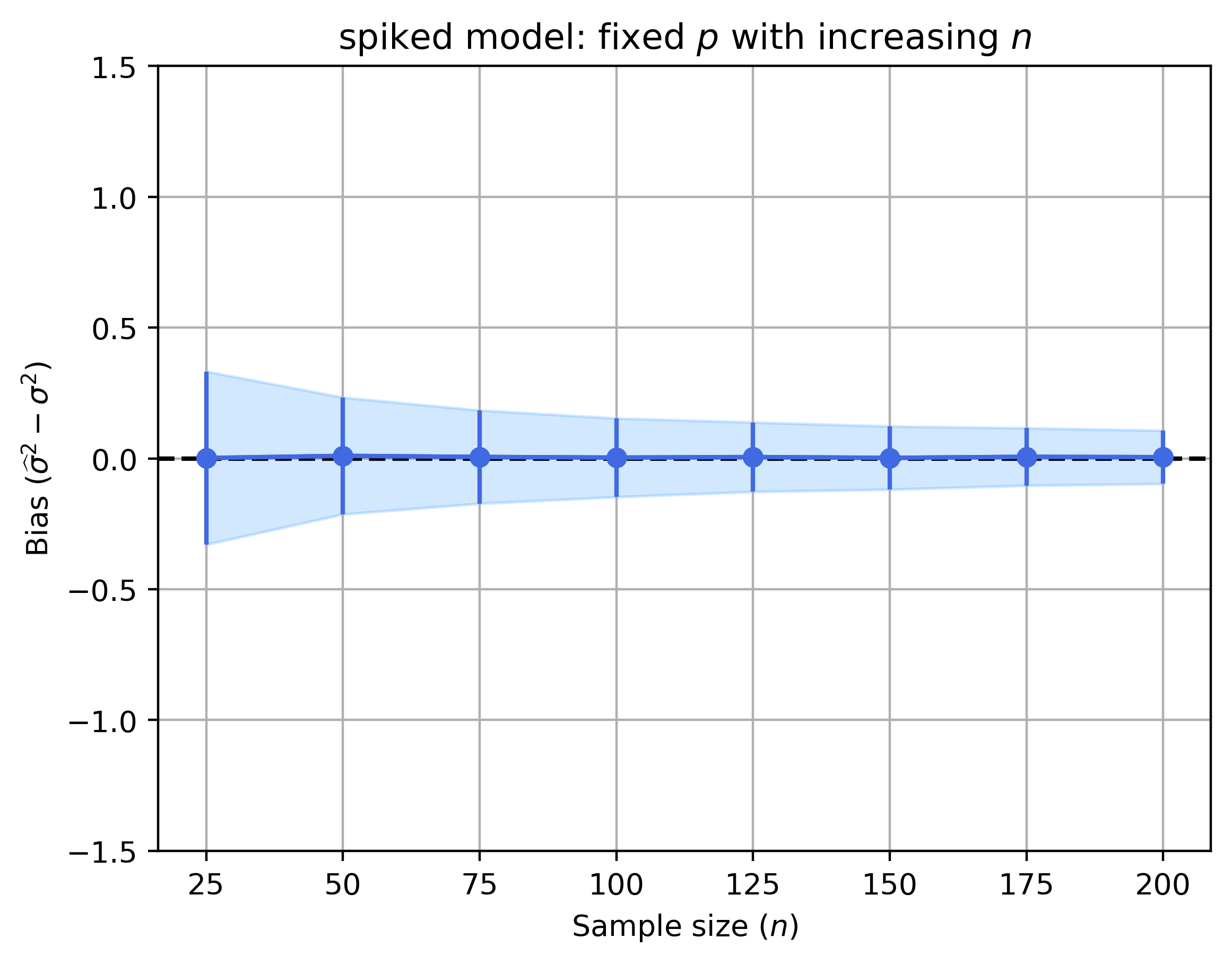}\label{fig:spiked}} 
	\caption{Simulation results of the biases of our variance estimator in \eqref{eq:var.estimator.hd} with fixed $p = 200$ and varying $n \in \{25, 50, 75, \dots, 175 \}$. The horizontal solid lines show the mean over $100$ trials, with shading and vertical bars to show $\pm$ one standard error.}
	\label{fig:bias} 
\end{figure*}

{\em Coverage.} We next study the coverage properties of the resulting confidence intervals under different covariate and noise models for a fixed $p > n$ with increasing $n$. Specifically, we consider both standard normal and uniform noise. 
For $\alpha \in (0,1)$, we construct the interval 
\begin{align}
	\texttt{CI}_{\alpha} (\bx_{n+1}) = \left[ \bx^\top_{n+1} \hbbeta ~\pm~ z_{\alpha/2}  \cdot \sqrt{\hsigma^2 \cdot \bx_{n+1}^\top (\bX^\top \bX )^\dagger \bx_{n+1}}  \right], 
\end{align} 
where $z_{\alpha/2}$ is the upper $\alpha/2$ quantile of the standard normal distribution. 
%
Tables~\ref{table:coverage.normal} and \ref{table:coverage.spiked} report the empirical coverage probabilities and average interval lengths at the 90\% nominal level for the standard normal and spiked models with respect to $\bbE[y_{n+1}]$. 

Under the spiked model, the confidence intervals achieve coverage close to the nominal level across all sample sizes and both noise distributions. Under the standard normal model, however, the intervals tend to undercover, and their lengths become impractically large when $n \approx p$. This behavior suggests that successful variance estimation, like successful treatment effect estimation, depends critically on the geometric structure of the design.

\subsection{Connections with Benign Overfitting Prediction Literature}
Collectively, Sections~\ref{sec:sim.treatment} and \ref{sec:sim.variance} reveal a common pattern: performance is substantially better under the spiked model than under the standard normal design. In the treatment effect experiments, the spiked model yields more stable estimators and greater efficiency gains. In the variance estimation experiments, it leads to smaller bias and confidence intervals with coverage closer to the nominal level. These findings echo a central theme in the benign overfitting literature on prediction, suggesting that low-dimensional structure in $\bX$ is important not only for prediction, but also for parameter recovery and inference, and potentially for broader statistical objectives.


\begin{table} [t]
\footnotesize
\centering 
\caption{Standard normal coverage for 90\% confidence intervals across 10000 replications.} 
\begin{tabular}{@{}lcccc@{}}
\toprule
\midrule 
\multirow{2}{*}{Sample size}
& \multicolumn{2}{c}{Gaussian noise $\Normal(0,1)$}   
& \multicolumn{2}{c}{Uniform noise $\textsf{U}[0,1]$}   
\\ 
\cmidrule(l){2-3} 
\cmidrule(l){4-5} 
& Coverage probability  & Interval length 
& Coverage probability  & Interval length 
\\
\midrule 
$n = 25$   	 
& 0.644	& 1.580
& 0.458	& 1.337  	 
\\
\hdashline
$n = 50$    	 
& 0.722	& 2.346
& 0.575	& 1.799  
\\
\hdashline
$n = 75$    	 
& 0.824	& 3.088
& 0.679	& 2.174  
\\
\hdashline
$n = 100$   	 
& 0.849	& 3.645
& 0.719	& 2.406 
\\
\hdashline
$n = 125$   	 
& 0.863	& 4.513
& 0.884	& 2.922 
\\
\hdashline
$n = 150$   	 
& 0.880	& 5.715
& 0.838	& 3.496 
\\
\hdashline
$n = 175$   	 
& 0.869	& 8.389
& 0.803	& 4.728  
\\
\hdashline
$n = 200$   	 
& 0.859	& 109.644
& 0.894	& 292.171  
\\
\bottomrule
\end{tabular}
\label{table:coverage.normal}
\end{table}

\begin{table} [t]
\footnotesize
\centering 
\caption{Spiked coverage for 90\% confidence intervals across 10000 replications.} 
\begin{tabular}{@{}lcccc@{}}
\toprule
\midrule 
\multirow{2}{*}{Sample size}
& \multicolumn{2}{c}{Gaussian noise $\Normal(0,1)$}   
& \multicolumn{2}{c}{Uniform noise $\textsf{U}[0,1]$}   
\\ 
\cmidrule(l){2-3} 
\cmidrule(l){4-5} 
& Coverage probability  & Interval length 
& Coverage probability  & Interval length 
\\
\midrule 
$n = 25$   	 
& 0.886	& 1.302
& 0.895	& 0.753  	 
\\
\hdashline
$n = 50$    	 
& 0.900	& 2.057
& 0.918	& 1.167  
\\
\hdashline
$n = 75$    	 
& 0.888	& 2.513
& 0.902	& 1.459  
\\
\hdashline
$n = 100$   	 
& 0.895	& 2.911
& 0.905	& 1.728 
\\
\hdashline
$n = 125$   	 
& 0.900	& 3.361
& 0.904	& 1.947 
\\
\hdashline
$n = 150$   	 
& 0.897	& 3.715
& 0.918	& 2.100 
\\
\hdashline
$n = 175$   	 
& 0.906	& 3.919
& 0.898	& 2.257  
\\
\hdashline
$n = 200$   	 
& 0.899	& 4.264
& 0.916	& 2.460  
\\
\bottomrule
\end{tabular}
\label{table:coverage.spiked}
\end{table}

\section{Conclusion} \label{sec:conclusion} 

Taken together, our findings shift attention from prediction risk---the focus of much of the benign overfitting literature---to parameter estimation and inference, which are central to causal inference and many other problems in statistics and the social sciences. At the same time, our simulations reinforce a recurring insight from the prediction literature: performance is especially favorable when the covariates exhibit spiked, approximately low-dimensional structure. More generally, our work contributes to the growing understanding of overparameterized models. 
In particular, it complements a broader recent movement in the literature to study benign overfitting beyond its original predictive framing. For example, \cite{ols_benign_bayesian} examined high-dimensional linear regression within a Bayesian setup and \cite{montanari_aos} showed that related benign overfitting phenomena can persist in classification. 
It also connects to a growing econometrics literature on overparameterized methods in panel-data settings, as in \cite{same_root} and \cite{spiess2023double}. Our contribution is different in focus but similar in spirit: within regression, we ask whether the favorable behavior of overparameterized models can extend beyond prediction to parameter recovery and uncertainty quantification. We hope the results developed here stimulate further research on when and why overparameterized methods remain statistically effective across a broader range of objectives, and help guide the development of methods that exploit these phenomena in practice.

\section*{Acknowledgements}
Dennis Shen and Dogyoon Song contributed equally.
We are sincerely grateful to the editor, three reviewers, Avi Feller,
Ben Recht, and Alexander Tsigler for their thoughtful feedback and helpful
suggestions. We note that an early version of the paper appeared on arXiv
under the title ``Algebraic and Statistical Properties of the Ordinary Least
Squares Interpolator.''

\newpage

\bibliographystyle{plainnat}
\bibliography{causal}

\newpage
\appendix

\begin{center}
  {\Large\textsc{Supplementary Material}}\\[4pt]
  {\large for ``Benign Overfitting Beyond Prediction:\
  The Ordinary Least Squares Interpolator''}
\end{center}
\vspace{12pt}

The supplementary material is structured as follows.
Section~\ref{apps.rows} details applications of the leave-one-out formula
(Corollary~\ref{cor:loo} in Section~\ref{sec:loo.results}).
Section~\ref{sec:gauss.markov.exp} provides additional commentary on
Remark~\ref{remark:gauss.marko.homo}.
Section~\ref{app:sims} details the simulation setup from
Section~\ref{sec:sim.variance} used to study the homoskedastic variance
estimator proposed in \eqref{eq:var.estimator.hd} of
Theorem~\ref{thm:var.estimator}, and provides complementary results.
To provide technical preliminaries,
Section~\ref{sec:pseudo_properties} overviews some properties of the
Moore--Penrose pseudoinverse.
Sections~\ref{sec:proofs.rows}, \ref{sec:proofs.columns}, and
\ref{sec:proofs.stats} contain the proofs of theorems, propositions, and
corollaries deferred from Sections~\ref{sec:row_partition},
\ref{sec:subsampled_column}, and \ref{sec:stat_inf}, respectively.

\section{Applications of the leave-one-out OLS formula} \label{apps.rows} 
%

\subsection{Revisiting the leave-one-out OLS formula using leave-one-out residuals}\label{sec:loo_residual}
Before we formally state our example applications, we first obtain a formula for the leave-one-out OLS estimates expressed in terms of the leave-one-out residuals by combining Corollaries~\ref{cor:loo} and \ref{cor:loo.shortcut}. 
\begin{corollary} \label{cor:loo.beta.shortcut}
Let $\bX \in \Rb^{n \times p}$ and $\by \in \Rb^n$. 
\begin{itemize} 
		
	\item[(a)] \emph{Classical regime} $(n > p)$: If Assumption~\ref{assump:column_rank} holds, then for any $i \in [n]$,
	\begin{align}
		\hbbeta^{(\sim i)} - \hbbeta 
	   = - \tvarepsilon_i\cdot (\bX^\top \bX)^{-1} \bx_i, \label{eq:loo.redundant.classical}
	\end{align}
where $\tvarepsilon_i$ is the $i$-th coordinate of the leave-one-out residual vector $\tbvarepsilon = \big[\diag(\bP_{\bX}^\perp)\big]^{-1} \cdot \bP_{\bX}^\perp  \by$, as defined in \eqref{eq:loo.shortcut.classical}. 

	\item[(b)] \emph{High-dimensional regime} $(n \le p)$: If Assumption~\ref{assump:row_rank} holds, then for any $i \in [n]$,
	\begin{align}
		\hbbeta^{(\sim i)} - \hbbeta 
	   = - \tvarepsilon_i\cdot (\bX^\top \bX)^\dagger \bx_i, \label{eq:loo.redundant.hd}
	\end{align}
where $\tvarepsilon_i$ is the $i$-th coordinate of the leave-one-out residual vector $\tbvarepsilon = \big[ \diag(\Gram_{\bX}) \big]^{-1} \cdot \Gram_{\bX}  \by$, as defined in \eqref{eq:loo.shortcut.hd}. 
\end{itemize}
\end{corollary} 

\begin{proof} 
The formula for the classical regime is well known and can be derived by plugging \eqref{eq:loo.shortcut.classical} into \eqref{eq:beta.loo.classical} with $\tvarepsilon_i = (1 - H_{ii})^{-1} \hvarepsilon_i$, where $H_{ii} = \bx_i^\top (\bX^\top \bX)^{-1} \bx_i$. 
Similarly in high-dimensions, we combine \eqref{eq:beta.loo.hd} with \eqref{eq:loo.shortcut.hd} to obtain that for any $i \in [n]$,
\begin{align}
	\hbbeta^{(\si)} - \hbbeta 
        &= - \frac{\bX^{\dagger} \cdot \be_i \be_i^\top \cdot \bX^{\dagger,\top}}{\be_i^\top \cdot \bG_{\bX} \cdot \be_i} \hbbeta\\
        &= -\bX^{\dagger} \be_i \cdot \frac{\be_i^\top \cdot \bG_{\bX} \cdot \by}{\be_i^\top \cdot \bG_{\bX} \cdot \be_i}
            &&\because \hbbeta = \bX^{\dagger} \by ~\text{and}~ \bG_{\bX} = (\bX^{\dagger})^\top \bX^{\dagger}\\
        &= - (\bX^\top \bX)^\dagger \bx_i \cdot \frac{\be_i^\top \cdot \bG_{\bX} \cdot \by}{ \diag( \bG_{\bX} )_{ii}}
            &&\because \bX^{\dagger} \be_i = (\bX^\top \bX)^\dagger \bx_i\\
        &= - (\bX^\top \bX)^\dagger \bx_i \cdot \tvarepsilon_i. 
\end{align}
\end{proof} 

Corollary~\ref{cor:loo.beta.shortcut} establishes that the difference between the leave-$i$-out and full-sample minimum $\ell_2$-norm OLS estimators not only can be expressed via the leave-$i$-out residuals, but also maintains the same formulation in both data regimes.  
While \eqref{eq:loo.redundant.classical} and \eqref{eq:loo.redundant.hd} might seem circular since the leave-one-out residuals are computed from the leave-one-out OLS estimates by definition, they are useful in deriving subsequent results in Sections~\ref{sec:app.loo} and \ref{sec:app.jackknife}. 



\subsection{Point prediction} \label{sec:app.loo} 
\noindent \textit{\PRESS~statistic for cross-validation.} 
In many settings, the researcher would like to estimate the performance of a model on out-of-sample data. 
For regression analysis, the predicted residual error sum of squares (\PRESS) statistic, which is a type of cross-validation metric that is popularly used to judge the predictive capability of the model by providing a summary measure of the model's fit to observations that were not themselves used during training. 
To compute \PRESS, a researcher would systematically hold out each data pair $(\bx_i, y_i)$ and learn a model on the remaining datapoints; 
then, the leave-$i$-th out residual is evaluated between $y_i$ and the prediction of the  model applied to $\bx_i$. 

Recall that $\tbvarepsilon \in \RR^n$ is the leave-one-out residual vector such that $\tvarepsilon_i\coloneqq y_i - \bx_i^{\top} \hbbeta^{(\sim i)}$ denote the leave-$i$-out prediction residuals. 
Then, \PRESS~is calculated as follows: 
\begin{align} \label{eq:press} 
	\PRESS &\coloneqq \sum_{i=1}^n \tvarepsilon_i^2 = \| \tbvarepsilon \|_2^2. 
\end{align} 
Based on Corollary \ref{cor:loo.shortcut}, we obtain a convenient computational formula for \PRESS.

\begin{corollary} \label{cor:press} 
    Let $\bX \in \Rb^{n \times p}$ and $\by \in \Rb^n$. 
    \begin{itemize} 
		
	\item[(a)] \emph{Classical regime} $(n > p)$: If Assumption~\ref{assump:column_rank} holds, then  
	\begin{align} \label{eq:loo.classical} 
		{\normalfont\PRESS} = \by^{\top} \cdot \left( \bP_{\bX}^\perp \cdot \big[\diag(\bP_{\bX}^\perp)\big]^{-1} \cdot \big[\diag(\bP_{\bX}^\perp)\big]^{-1} \cdot \bP_{\bX}^\perp \right) \cdot \by. 
	\end{align} 
	
	\item[(b)] \emph{High-dimensional regime} $(n \le p)$: If Assumption~\ref{assump:row_rank} holds, then  
	\begin{align} \label{eq:loo.hd} 
		{\normalfont\PRESS} = \by^{\top} \cdot \left( \Gram_{\bX} \cdot \big[ \diag(\Gram_{\bX}) \big]^{-1} \cdot \big[ \diag(\Gram_{\bX}) \big]^{-1} \cdot \Gram_{\bX} \right) \cdot \by
	\end{align} 
	\end{itemize}
\end{corollary}

\textit{Gauss updating formula for online regression.}
Consider the online setting in which data points arrive sequentially streaming. 
In particular, suppose we construct $\hbbeta^{(n)} \in \Rb^p$ from the first $n$ data points, which are recorded in $\bX^{(n)} \in \RR^{n \times p}$ and $\by^{(n)} \in \RR^n$. 
Then our task is to update $\hbbeta^{(n)}$ with the novel datapoint $(\bx_{n+1}, y_{n+1})$ to obtain $\hbbeta^{(n+1)}$. 
Below, we provide a formula to perform this update.

\begin{corollary} \label{cor:online}
Let $\bX^{(n)} \in \Rb^{n \times p}$ and $\by^{(n)} \in \Rb^n$.
Let $(\bx_{n+1}, y_{n+1})$ denote a novel datapoint. 

\begin{itemize} 
		
	\item[(a)] \emph{Classical regime} $(n > p)$: If Assumption~\ref{assump:column_rank} holds, then  
\begin{align} \label{eq:online.classical} 
	\hbbeta^{(n+1)} &= \hbbeta^{(n)} + \tvarepsilon_{n+1} \cdot \bgamma^{(n+1)} , 
\end{align} 
where $\tvarepsilon_{n+1} = y_{n+1} - \bx_{n+1}^\top \hbbeta^{(n)}$ and $\bgamma^{(n+1)} = \left({\bX^{(n+1)}}^\top \bX^{(n+1)}\right)^{-1} \bx_{n+1}$. 

	\item[(b)] \emph{High-dimensional regime} $(n \le p)$: If Assumption~\ref{assump:row_rank} holds, then  
\begin{align} \label{eq:online.hd} 
	\hbbeta^{(n+1)} &= \hbbeta^{(n)} + \tvarepsilon_{n+1} \cdot \bgamma^{(n+1)} , 
\end{align} 
where $\tvarepsilon_{n+1} = y_{n+1} - \bx_{n+1}^\top \hbbeta^{(n)}$ and $\bgamma^{(n+1)} = \left( {\bX^{(n+1)}}^\top \bX^{(n+1)} \right)^\dagger \bx_{n+1}$. 
\end{itemize}
\end{corollary} 

\begin{proof} 
If we view the first $n+1$ data points as the full data, i.e., $\bX = \bX^{(n+1)}$ and $\by = \by^{(n+1)}$, then $\hbbeta^{(n)}$ is the leave-$(n+1)$-out solution. 
Applying Corollaries~\ref{cor:loo} and \ref{cor:loo.beta.shortcut} gives the desired result. 
\end{proof} 

In words, Corollary~\ref{cor:online} states that the adjustment from $\hbbeta^{(n)}$ to $\hbbeta^{(n+1)}$ is related to the predicted residual $\tvarepsilon_{n+1}$.  
If the predicted residual for the $(n+1)$-th observation is large, then the adjustment is consequentially large; at the same time, if the predicted residual is zero, then no adjustment is necessary. 
Notably, the adjustment takes the same form in both data regimes. 

Finally, we note that \eqref{eq:online.classical} and \eqref{eq:online.hd} can be efficiently computed by first viewing ${\bX^{(n+1)}}^\top \bX^{(n+1)}$ as the rank-one update of  ${\bX^{(n)}}^\top \bX^{(n)}$, i.e.,
${\bX^{(n+1)}}^\top \bX^{(n+1)} =  {\bX^{(n)}}^\top \bX^{(n)}+ \bx_{n+1} \bx_{n+1}^\top$, and then applying the Sherman-Morrison formula to \eqref{eq:online.classical} and its generalization for pseudoinverses \cite{gen_inverse_meyers} to \eqref{eq:online.hd}; 
see Lemma~\ref{lemma:meyer} in Section~\ref{sec:proof.alt.beta.loo} for the latter result.


\subsection{Inference under the \jackknife~and \jackknife+} \label{sec:app.jackknife} 
Whereas Section~\ref{sec:app.loo} provided an approach to evaluate the predictive capability of the OLS interpolator and update it through its leave-one-out residuals, this subsection focuses on variance estimation and predictive inference. 
In particular, we consider the \jackknife, which is closely connected to the leave-one-out configuration. 
%

\medskip \noindent
\textit{The \jackknife~for $\hbbeta$.} 
We begin by discussing the \jackknife~procedure for $\hbbeta$. 

\medskip
\noindent \textit{I: Point estimator}.  
The \jackknife~estimate of $\hbbeta$ is defined as 
\begin{align}\label{eqn:jackknife}
	\hbbeta^{\J} = n \hbbeta - (n-1) \tbbeta, 
\end{align} 
where $\tbbeta = n^{-1} \sum_{i=1}^n \hbbeta^{(\sim i)}$. 
We can also construct the $i$-th leave-one-out pseudo-value as $\tbbeta^{(i)} = n \hbbeta - (n-1) \hbbeta^{(\sim i)}$. 
In turn, we can rewrite \eqref{eqn:jackknife} as the empirical mean of the $n$ leave-one-out pseudo-values, i.e., 
\begin{align}\label{eqn:jackknife.pseudo}
    \hbbeta^{\J} &= \frac{1}{n} \sum_{i=1}^n \tbbeta^{(i)}. 
\end{align} 
Given the connection between the \jackknife~and the leave-one-out configuration, we investigate the leave-one-out residuals through the lens of the \jackknife~estimate of the OLS minimum $\ell_2$-norm estimator in the corollary below. 
A proof of Corollary \ref{cor:loo.jackknife.1} is provided in Section \ref{sec:proof_jackknife.1}. 

\begin{corollary} \label{cor:loo.jackknife.1} 
Let $\bX \in \Rb^{n \times p}$ and $\by \in \Rb^n$. 
\begin{itemize} 
	\item[(a)] \emph{Classical regime} $(n > p)$: If Assumption~\ref{assump:column_rank} holds, then
\begin{align} \label{eq:loo.jack.classical}
	\left\{ \bP_{\bX} - \frac{n}{n-1} \cdot \diag\left(\bP_{\bX}^\perp \right) \right\} \cdot \tbvarepsilon &= \frac{n}{n-1} \cdot \left( \bX \hbbeta^{\text{\normalfont Jack}} - \by \right), 
\end{align}
where the leave-one-out residual vector $\tbvarepsilon = \big[\diag(\bP_{\bX}^\perp)\big]^{-1} \cdot \bP_{\bX}^\perp  \by$, as defined in \eqref{eq:loo.shortcut.classical}.

\item[(b)] \emph{High-dimensional regime} $(n \le p)$: If Assumption~\ref{assump:row_rank} holds, then 
\begin{align} \label{eq:loo.jack.hd} 
	\tbvarepsilon =  \frac{n}{n-1} \cdot \left(\bX \hbbeta^{\text{\normalfont Jack}} - \by \right),
\end{align}  
where the leave-one-out residual vector $\tbvarepsilon = \big[ \diag(\Gram_{\bX}) \big]^{-1} \cdot \Gram_{\bX}  \by$, as defined in \eqref{eq:loo.shortcut.hd}. 
\end{itemize}
\end{corollary} 
%
Corollary \ref{cor:loo.jackknife.1} reveals a close link between the estimation of the leave-one-out residuals and the \jackknife{}. 
In the classical regime, the relationship between the \jackknife{} and the leave-one-out residuals is clear and has been well known, albeit in a slightly different form from \eqref{eq:loo.jack.classical}.
%
However, we highlight that Corollary \ref{cor:loo.jackknife.1}, in combination with the closed-form expressions for the leave-one-out residuals $\bvarepsilon$, provides a handy computational mean to compute $\hbbeta^{\J}$ without needing to learn $n$ distinct models:
\begin{align*}
    \hbbeta^{\J} 
        &= \bX^{\dagger} \by + \bX^{\dagger} \cdot\left\{ \frac{n-1}{n} \cdot \bP_{\bX} - \diag\left(\bP_{\bX}^\perp \right) \right\} \cdot \tbvarepsilon\\
        &= \hbbeta + \bX^{\dagger} \cdot\left\{ \frac{n-1}{n} \cdot \bP_{\bX} - \diag\left(\bP_{\bX}^\perp \right) \right\} \cdot \big[\diag(\bP_{\bX}^\perp)\big]^{-1} \cdot \bP_{\bX}^\perp  \by.
\end{align*}

Furthermore, \eqref{eq:loo.jack.hd} demonstrates that the OLS interpolator's leave-one-out residuals, which can be computed as per \eqref{eq:loo.shortcut.hd}, can equivalently be calculated by simply rescaling the in-sample residuals based on the \jackknife~estimate. 
Note that $\hbbeta^{\J} \in \rowsp(\bX)$ because $\hbbeta, \hbbeta^{(\sim i)} \in \rowsp(\bX)$ by definition of the minimum $\ell_2$-norm OLS interpolator. 
Therefore, one can similarly express $\hbbeta^{\J}$ in terms of $\tbvarepsilon$ in the high-dimensional regime as
\begin{align*}
    \hbbeta^{\J}
        &= \bX^{\dagger} \by + \frac{n-1}{n} \tbvarepsilon
        = \hbbeta + \frac{n-1}{n} \big[ \diag(\Gram_{\bX}) \big]^{-1} \cdot \Gram_{\bX}  \by.
\end{align*}

Although there exists a rescaling factor that left multiples $\tbvarepsilon$ on the left-hand side of \eqref{eq:loo.jack.classical} in the classical regime, this matrix factor reduces the identity matrix when Assumption~\ref{assump:row_rank} holds since $\bP_{\bX} = \bI_n$. 
Thus, in principle, the high-dimensional formula \eqref{eq:loo.jack.hd} can be written in the same formulation as \eqref{eq:loo.jack.classical}. 

\medskip
\noindent \textit{II: Variance estimator}. 
The \jackknife~estimation procedure is often used in the classical regime to construct a variance estimator of $\hbbeta$. 
In fact, the HC3 correction of the well-known Eicker-Hubert-White (EHW) variance estimator for $\hbbeta$ was motivated by the \jackknife. 
Amongst the many EHW variants, Long and Ervin recommended the HC3 correction based on extensive simulation studies. 

Formally, the HC3 correction is defined as 
\begin{equation}\label{eqn:jackknife_var}
	    \bhV^{\J} 
            = \frac{1}{(n-1)^2} \sum_{i=1}^n \left(\tbbeta^{(i)} - \hbbeta \right) \left(\tbbeta^{(i)} - \hbbeta \right)^\top. 
\end{equation}
Here, we remark that centering at the unbiased estimator $\hbbeta$ in \eqref{eqn:jackknife_var} simplifies the formula, although the original definition of the Jackknife variance estimator is centered at $\hbbeta^{\J}$. 

\begin{corollary} \label{cor:loo.jackknife.2} 
Let $\bX \in \Rb^{n \times p}$ and $\by \in \Rb^n$. 
\begin{itemize} 
	\item[(a)] \emph{Classical regime} $(n > p)$: If Assumption~\ref{assump:column_rank} holds, then
\begin{align} \label{eq:loo.varjack.classical}
	\bhV^{\text{\normalfont Jack}}
            =  \bX^\dagger \cdot \bOmega  \cdot (\bX^\dagger)^\top,  
\end{align} 
where $\bOmega =  \sum_{i=1}^n \tvarepsilon_i^2 \cdot \be_i \be_i^\top$ with $\tvarepsilon_i$ denoting the $i$-th entry of the leave-one-out residual vector $\tbvarepsilon$ as in \eqref{eq:loo.shortcut.classical}.
 
\item[(b)] \emph{High-dimensional regime} $(n \le p)$: If Assumption~\ref{assump:row_rank} holds, then 
\begin{align} \label{eq:loo.varjack.hd}
    	\bhV^{\text{\normalfont Jack}}
            =  \bX^\dagger \cdot \bOmega \cdot (\bX^\dagger)^\top,  
    \end{align} 
    where $\bOmega = \sum_{i=1}^n \tvarepsilon_i^2 \cdot \be_i \be_i^\top$  with $\tvarepsilon_i$ denoting the $i$-th entry of the leave-one-out residual vector $\tbvarepsilon$ as in \eqref{eq:loo.shortcut.hd}. 
\end{itemize}
\end{corollary} 
\begin{proof} 
In both the classical and high-dimensional regimes,
\begin{align}
	\tbbeta^{(i)} 
        &= n \hbbeta - (n-1) \hbbeta^{(\si)}
        \\
        &= n \hbbeta - (n-1) \cdot \left\{ \hbbeta - \tvarepsilon_i (\bX^\top \bX)^\dagger \bx_i \right\}
            &&\because \text{Corollary~\ref{cor:loo.beta.shortcut}}
        \\
        &= \hbbeta + (n-1) \cdot \tvarepsilon_i \cdot \bX^{\dagger} \cdot \be_i,
            &&\because (\bX^\top \bX)^\dagger \bx_i = \bX^{\dagger} \cdot \be_i
            \label{eq:jack.0}
\end{align} 
where the leave-one-out residuals $\tvarepsilon_i$ are appropriately defined for each of the classical and high-dimensional settings. 
We used the fact $(\bX^\top \bX)^{-1} = (\bX^\top \bX)^\dagger$ in the classical regime. 
Thus, the HC3 correction of the EHW variance estimator \eqref{eqn:jackknife_var} admits the expression
\begin{align}
	\bhV^{\text{Jack}} 
	    &= \bX^{\dagger} \cdot \left( \sum_{i=1}^n \tvarepsilon_i^2 \cdot \be_i \be_i^\top \right)  \cdot \big( \bX^\dagger \big)^\top \label{eq:jack.var.1} 
\end{align} 
in both data regimes. 
\end{proof}

Comparing the expressions in \eqref{eq:loo.varjack.classical} and \eqref{eq:loo.varjack.hd}, we observe that $\bhV^{\J}$ maintains the same form in both data regimes. 
Since the leave-one-out residual $\tbvarepsilon$ can be computed via ``shortcut'' formulas in \eqref{eq:loo.shortcut.classical} and \eqref{eq:loo.shortcut.hd}, we can also compute $\bhV^{\J}$ using simple matrix multiplications without needing to learn $n$ distinct models. 
However, the actual computation naturally differs due to the distinct calculation of the leave-one-out residuals $\tbvarepsilon$. 

\medskip \noindent 
\textit{Predictive inference.} 
%
%
Suppose that we have random training data $(\bx_i, y_i) \in \RR^d \times \RR$ for $i \in [n]$ and a new test point $(\bx_{n+1}, y_{n+1})$. 
A reliable prediction interval for a new test point is often desirable. 
That is, for a number $\alpha \in [0, 1]$, we want to construct a prediction interval $\PI_{\alpha}$ with target coverage level $1-\alpha$ as a function of the $n$ training data points. 
Formally, given $\left\{ (\bx_i, y_i): i \in [n] \right\}$ and $\alpha$, we want the map $\PI_{\alpha}: \bx_{n+1} \mapsto \PI_{\alpha} ( \bx_{n+1} ) \subseteq \RR$ to satisfy
\[
    \mathbb{P}\left\{ y_{n+1} \in \PI_{\alpha}(\bx_{n+1}) \right\} \geq 1 - \alpha,
\]
where the probability is taken with respect to the randomness in both the training data and the new test point. 

A straightforward approach is to use the in-sample residuals to estimate the prediction error at a new test point $\bx_{n+1}$, for example, by considering
\begin{align*}
    \PI_{\alpha}(\bx_{n+1}) = \left[ \hy_{n+1} - r_{\alpha}, \hy_{n+1} + r_{\alpha} \right],
\end{align*}
where $\hy_{n+1} = \bx_{n+1}^{\top} \hbbeta$ and $r_{\alpha} \coloneqq \text{the }(1-\alpha)\text{ quantile of }\{ | y_i- \bx_i^{\top} \hbbeta|: i \in [n] \}$. 
%
However, this approach can lead to undercoverage, meaning that $\mathbb{P}\{ y_{n+1} \in \PI_{\alpha}(\bx_{n+1}) \}$ tends to be lower than the target level $1-\alpha$.

\medskip
\noindent \textit{I: Standard \jackknife}. 
To address this problem, the standard \jackknife~prediction method uses the leave-one-out residuals instead of the in-sample residuals to construct a prediction interval.

To be precise, for any $\alpha \in [0,1]$, the \jackknife~prediction interval is defined as follows:
\begin{equation}
    \begin{aligned}
        \PI^{\J}_{\alpha}(\bx_{n+1}) = \left[ \hy_{n+1} - r^{\J}_{\alpha}, \hy_{n+1} + r^{\J}_{\alpha} \right],
        \label{eq:pred.interval.jack} 
    \end{aligned}
\end{equation}
where $r^{\J}_{\alpha} \coloneqq \text{the }(1-\alpha)\text{ quantile of }\left\{ | \tvarepsilon_i| : i \in [n] \right\}$.
Recall that $\tvarepsilon_i\coloneqq y_i - \bx_i^{\top} \hbbeta^{(\sim i)}$ denotes the leave-$i$-out prediction residual, and these leave-one-out residuals can be efficiently computed via the computational shortcut formulas in \eqref{eq:loo.shortcut.classical} (classical regime) or \eqref{eq:loo.shortcut.hd} (high-dimensional regime) of Corollary~\ref{cor:loo.shortcut}. 
Intuitively, the \jackknife~approach should mitigate the overfitting problem by using leave-one-out residuals and achieve the desired coverage on average.

\medskip
\noindent \textit{II: The \jackknife+}. 
Despite its widespread use, the \jackknife~approach is criticized for two drawbacks: (1) the lack of universal theoretical guarantees, and (2) the tendency of undercoverage in the case where the regression algorithm is unstable \citep{jackknife+}. 
As a remedy, \cite{jackknife+} introduced the \jackknifep. 
While both the \jackknife~and \jackknifep~use the leave-one-out residuals, the \jackknifep~also uses the leave-one-out predictions for the test point. 

Formally, for $\alpha \in [0,1]$ and a finite set $\cR$, we define the quantile function as 
\[
    \hq_{\alpha}(\cR) \coloneqq 
        \begin{cases}
            \text{the }\lceil (1-\alpha)(|\cR|+1) \rceil\text{-th smallest value in } \cR &
                \text{if }\alpha \geq \frac{1}{|\cR|+1},\\
            \infty  & \text{if }\alpha < \frac{1}{|\cR|+1},
        \end{cases}
\]
where $|\cR|$ is the cardinality of the set $\cR$. 
For any $\alpha \in [0,1]$, the \jackknifep~prediction interval is defined as
\begin{equation}
    \begin{aligned}
        \PI^{\J+}_{\alpha}(\bx_{n+1}) = \left[ - \hq_{\alpha} \Big\{ -\bx_{n+1}^{\top} \hbbeta^{(\sim i)} + |\tvarepsilon_i|: i \in [n] \Big\},~~ \hq_{\alpha} \Big\{ \bx_{n+1}^{\top} \hbbeta^{(\sim i)} + |\tvarepsilon_i|: i \in [n] \Big\} \right].
        \label{eq:pred.interval.jack+} 
    \end{aligned}
\end{equation}
As discussed in \cite{jackknife+}, the \jackknifep~prediction interval can be interpreted as an interval around the \textit{median} prediction of the leave-one-out predictions,
\[
    \textup{Median} \left( \bx_{n+1}^{\top} \hbbeta^{(\sim 1)}, \dots, \bx_{n+1}^{\top} \hbbeta^{(\sim n)} \right),
\]
which is guaranteed to be in $\PI^{\J+}_{\alpha}(\bx_{n+1})$ for all $\alpha \leq 1/2$; although, this interval is generally not symmetric around this median.

Given the efficacy of the \jackknifep~method, we provide a shortcut to compute the leave-one-out predictions for the OLS minimum $\ell_2$-norm estimator. 
To state our result, recall $\hy_{n+1} = \bx_{n+1}^\top \hbbeta$ denotes the predicted value at $\bx_{n+1}$ based on the full-sample OLS.

\begin{corollary} \label{cor:loo.pred}
Let $\bX \in \Rb^{n \times p}$, $\by \in \Rb^n$, and $\bx_{n+1} \in \Rb^p$. 
\begin{itemize} 
		
	\item[(a)] \emph{Classical regime} $(n > p)$: If Assumption~\ref{assump:column_rank} holds, then 
\begin{align}\label{eqn:jackknifep.classical}
    \bx_{n+1}^{\top} \begin{bmatrix} \hbbeta^{(\sim 1)} & \cdots & \hbbeta^{(\sim n)}\end{bmatrix}
        = \hy_{n+1} \cdot \bone_{n}^{\top} - \bx_{n+1}^{\top} \cdot\bX^{\dagger} \cdot \diag(\tbvarepsilon), 
\end{align} 
where the leave-one-out residual vector $\tbvarepsilon = \big[\diag(\bP_{\bX}^\perp)\big]^{-1} \cdot \bP_{\bX}^\perp  \by$, as defined in \eqref{eq:loo.shortcut.classical}. 

	\item[(b)] \emph{High-dimensional regime} $(n \le p)$: If Assumption~\ref{assump:row_rank} holds, then 
\begin{align}\label{eqn:jackknifep.hd}
    \bx_{n+1}^{\top} \begin{bmatrix} \hbbeta^{(\sim 1)} & \cdots & \hbbeta^{(\sim n)}\end{bmatrix}
        = \hy_{n+1} \cdot \bone_{n}^{\top} - \bx_{n+1}^{\top} \cdot\bX^{\dagger} \cdot \diag(\tbvarepsilon),
\end{align} 
where the leave-one-out residual vector $\tbvarepsilon = \big[ \diag(\Gram_{\bX}) \big]^{-1} \cdot \Gram_{\bX}  \by$, as defined in \eqref{eq:loo.shortcut.hd}. 
\end{itemize}
\end{corollary} 
\begin{proof} 
For each $i \in [n]$, we have 
\begin{align}
    \bx_{n+1}^\top \hbbeta^{(\si)}
    	&= \bx_{n+1}^\top \cdot \left\{ \hbbeta - (\bX^\top \bX)^\dagger \bx_i  \tvarepsilon_i \right\}
            &&\because \text{Corollary~\ref{cor:loo.beta.shortcut}}\\
    	&= \hy_{n+1} - \bx_{n+1}^\top (\bX^\top \bX)^\dagger \bx_i \tvarepsilon_i
            &&\because \hy_{n+1} = \bx_{n+1}^{\top} \hbbeta\\
        &= \hy_{n+1} - \bx_{n+1}^\top \bX^{\dagger} \be_i  \be_i^\top \tbvarepsilon
            &&\because \bx_i = \bX^\top \be_i \text{ and }\tvarepsilon_i = \be_i^\top \tbvarepsilon
\end{align}
in both data regimes. 
Collecting the expressions $\bx_{n+1}^\top \hbbeta^{(\si)}$ in a single matrix-vector equation, we obtain
\begin{equation}
    \bx_{n+1}^{\top} \begin{bmatrix} \hbbeta^{(\sim 1)} & \cdots & \hbbeta^{(\sim n)}\end{bmatrix}
        = \hy_{n+1} \cdot \bone_{1, n} - \bx_{n+1}^{\top} \bX^{\dagger} \cdot  \diag(\tbvarepsilon), 
\end{equation}
where we recall $\diag(\tbvarepsilon) \coloneqq \sum_{i=1}^n \tvarepsilon_i \be_i \be_i^{\top}$.
\end{proof} 
It is evident from \eqref{eqn:jackknifep.classical} and \eqref{eqn:jackknifep.hd} that the predicted value based on the leave-$i$-out samples can be directly computed from $(\bX, \by)$ in a single operation for all $i \in [n]$ without having to construct and evaluate the performance of $n$ individual models.

\subsection{Proof of Corollary~\ref{cor:loo.jackknife.1}} \label{sec:proof_jackknife.1}

\begin{proof} 
We present the proofs for both the classical and high-dimensional regimes. 
In this proof, we first prove the equations for the leave-one-out residuals and then the formulas for the variance estimators.

Observe that for both classical and high-dimensional regimes, the $i$-th leave-one-out pseudo-value is written as
\begin{align}
	\tbbeta^{(i)} 
        &= n \hbbeta - (n-1) \hbbeta^{(\si)}
        \\
        &= n \hbbeta - (n-1) \cdot \left\{ \hbbeta - \tvarepsilon_i (\bX^\top \bX)^\dagger \bx_i \right\}&&\because \text{Corollary~\ref{cor:loo.beta.shortcut}}
        \\
        &= \hbbeta + (n-1) \cdot \tvarepsilon_i (\bX^\top \bX)^\dagger \bx_i. \label{eq:jack.1} 
\end{align} 
Note that we used the fact $(\bX^\top \bX)^{-1} = (\bX^\top \bX)^\dagger$ when applying Corollary~\ref{cor:loo.beta.shortcut} for the classical regime. 
In turn, \eqref{eq:jack.1} implies that the jackknife point estimator, as expressed in \eqref{eqn:jackknife.pseudo}, can be written as 
\begin{align}
	\hbbeta^{\text{Jack}} 
	    &= \frac{1}{n} \sum_{i=1}^n \tbbeta^{(i)}
	    = \hbbeta + \left(\frac{n-1}{n} \right) \bX^\dagger \tbvarepsilon. \label{eq:jack.2} 
\end{align} 
Multiplying $\bX$ from left to both sides of \eqref{eq:jack.2} and subtracting $\by$, we obtain
\begin{align}
	\bX \hbbeta^{\text{Jack}} - \by
        &= \bX \left\{ \hbbeta + \left(\frac{n-1}{n} \right) \bX^{\dagger} \tbvarepsilon \right\} - \by\\
	    &= \left( \bX \hbbeta - \by \right) + \left(\frac{n-1}{n} \right) \bP_{\bX} \tbvarepsilon, \label{eq:jack.3} 
\end{align} 
where we recall $\bP_{\bX} = \bX \bX^\dagger$. 

\begin{itemize} 
	\item[(i)] \emph{Proof of \eqref{eq:loo.jack.classical} (classical regime)}: 
	We first note that $\by - \bX \hbbeta = \by - \bX\bX^{\dagger} \by = \bP_{\bX}^\perp \by$ precisely corresponds to the in-sample residuals. 
	Thus, \eqref{eq:jack.3} yields    
	\begin{align}
		\bX \hbbeta^{\text{Jack}} - \by
	    	&= - \bP_{\bX}^\perp \by + \left(\frac{n-1}{n} \right) \bP_{\bX} \tbvarepsilon
    		\\
    		&= -\diag\left( \bP_{\bX}^\perp \right) \cdot \tbvarepsilon + \left(\frac{n-1}{n} \right) \bP_{\bX} \tbvarepsilon
                &&\because \text{Corollary~\ref{cor:loo.shortcut}},~\eqref{eq:loo.shortcut.classical}
    		\\
    		&= \left\{  \left( \frac{n-1}{n} \right) \bP_{\bX} - \diag\left(\bP_{\bX}^\perp \right) \right\} \cdot \tbvarepsilon,
	\end{align} 
	where the leave-one-out residual vector $\tbvarepsilon$ is as defined as in \eqref{eq:loo.shortcut.classical}. 
	
	\item[(ii)] \emph{Proof of \eqref{eq:loo.jack.hd} (high-dimensional regime)}: 
    Contextualizing \eqref{eq:jack.3} in the high-dimensional regime, we observe that under Assumption~\ref{assump:row_rank}, (i) $\bX\hbbeta - \by = \bzero$ and (ii) $\bP_{\bX} = \bI_n$.
	Hence,
	\begin{align}
		\bX \hbbeta^{\text{Jack}} - \by
	    	&= \left(\frac{n-1}{n} \right) \bP_{\bX} \tbvarepsilon
		= \left( \frac{n-1}{n} \right) \tbvarepsilon,
	\end{align} 
	where the leave-one-out residual vector $\tbvarepsilon$ is as defined as in \eqref{eq:loo.shortcut.hd}. 
\end{itemize} 
\end{proof}


\section{Expansion on remark~\ref{remark:gauss.marko.homo}} \label{sec:gauss.markov.exp}
\subsection{Restatement of Theorem~\ref{thm:Gauss_Markov}}
We begin with a more general restatement of the high-dimensional result in Theorem~\ref{thm:Gauss_Markov}(b). 
\begin{theorem}[Gauss-Markov theorem in the high-dimensional regime $(n \le p)$]\label{thm:Gauss_Markov.general}
Suppose Assumption~\ref{assump:row_rank} and Assumption~\ref{assump:lm} hold with an \underline{arbitrary} $\bSigma \succeq \bzero$. 
        If $\hbbeta = \olsmin(\bX, \by)$ and $\tbbeta = \bL \by$ is any estimator of $\bbeta$ that is linear in $\by$ such that $\bbE[\bX \tbbeta] = \bX \bbeta$, then: 
        \begin{itemize}
            \item[(i)] 
            For every $\bv = \bX^{\top} \bc \in \rowsp(\bX)$, 
            \[
                \bc^{\top} \bSigma \bc = \bv^{\top} \Cov( \hbbeta ) \bv = \bv^{\top} \Cov( \tbbeta ) \bv. 
            \]
            
            \item[(ii)] 
            For every $\bw \in \rowsp(\bX)^{\perp}$, 
            \[
                0 = \bw^{\top} \Cov( \hbbeta ) \bw \leq \bw^{\top} \Cov( \tbbeta ) \bw,
            \]
            with equality for all such $\bw$ if and only if $\rowsp(\bL-\bX^\dagger)\subseteq \ker(\Sigma)$.
        \end{itemize}
\end{theorem}

Theorem~\ref{thm:Gauss_Markov.general} makes clear that the high-dimensional result does not require homoskedasticity. 
The key distinction is that Theorem~\ref{thm:Gauss_Markov.general} only compares variances along $\rowsp(\bX)$ and its orthogonal complement, rather than enforcing a global positive semidefinite ordering as in the classical Gauss-Markov relationship stated in Theorem~\ref{thm:Gauss_Markov}(a).  
For $\bv=\bX^\top\bc$, the identity 
\[
        \bv^\top\Cov(\bL\by)\,\bv=\bc^\top\bSigma\,\bc
\]
holds for every right inverse $\bL$ and every $\bSigma\succeq0$. 
Thus, Theorem~\ref{thm:Gauss_Markov.general}(i) does not rely on homoskedasticity. 
For $\bw\in\rowsp(\bX)^\perp$, writing $\bL=\bX^\dagger+\bN$ with $\bX\bN=0$ gives
\[
        \bw^\top \Bigl\{ \Cov(\bL\by)-\Cov(\bX^\dagger\by) \Big\} \bw
            = \| \bSigma^{1/2}\bN^\top\bw \|_2^2
            \geq 0. 
\]
Hence, Theorem~\ref{thm:Gauss_Markov.general}(ii) also requires only $\bSigma\succeq \bzero$; equality for all such $\bw$ occurs if and only if $\rowsp(\bL-\bX^\dagger)\subseteq\ker(\bSigma)$. 
By contrast, Theorem~\ref{thm:Gauss_Markov}(a) seeks a \emph{global} minimization in the positive semidefinite ordering across \emph{all} directions in $\RR^p$. 
Without homoskedasticity, OLS need not be the best linear unbiased estimator. 
For this reason, the classical result of Theorem~\ref{thm:Gauss_Markov}(a) assumes $\bSigma=\sigma^2\bI_n$, while Theorem~\ref{thm:Gauss_Markov.general} does not.  

To extend the analysis in the classical regime to the heteroskedastic setting, consider a general covariance matrix $\bSigma \succ \bzero$. 
Let $\bSigma^{1/2} = \bU \bS^{1/2} \bU^{\top}$ be its eigendecomposition, with $\bSigma^{-1/2} = \big( \bSigma^{1/2} \big)^{-1}$. 
Transform the data and rewrite the linear model in Assumption~\ref{assump:lm} as $\bSigma^{-1/2} \by = \bSigma^{-1/2} \bX \bbeta + \bSigma^{-1/2} \bvarepsilon$. 
Observe that the transformed error $\bvarepsilon' \coloneqq \bSigma^{-1/2} \bvarepsilon$ has mean $\bzero$ and covariance $\bI_n$. 
Therefore, 
\begin{align}
	\hbbeta^{\bSigma} \coloneqq \olsmin\big( \bSigma^{-1/2} \bX, \bSigma^{-1/2} \by \big)
             = \bigl( \bSigma^{-1/2} \bX \bigr)^{\dagger} \bSigma^{-1/2} \by\
\end{align} 
is the best linear unbiased estimator in the sense of Theorem \ref{thm:Gauss_Markov.general}.
In the classical regime, this estimator reduces to the well-known generalized least squares estimator \citep{aitken1936iv}:
     \begin{align*}
         \hbbeta^{\bSigma}
             &= \Bigl\{ \bigl( \bSigma^{-1/2} \bX \bigr)^{\top} \bigl( \bSigma^{-1/2} \bX \bigr) \Bigr\}^{-1} \cdot\bigl( \bSigma^{-1/2} \bX \bigr)^{\top} \bSigma^{-1/2} \by
             \\
             &= \bigl\{ \bX^{\top} \bSigma^{-1} \bX \bigr\}^{-1} \cdot \bX^{\top} \bSigma^{-1} \by.
     \end{align*}

\section{Simulation Study Details from Section~\ref{sec:sim.variance}} \label{app:sims}
This section details the simulation setup from Section~\ref{sec:sim.variance}, and offers several complementary results on the estimation bias. 
%
Notably, while Theorem~\ref{thm:var.estimator} holds under Assumption~\ref{assump:lm} with $\bX$ fixed, our simulations will introduce randomness in $\bX$ to assess the estimator's average performance. 
%
%
%
%
%

\subsection{Simulation I: fixed dimension and varying sample size} \label{sec:sim.bias.1} 
%

\textit{Data generating process.} \label{sec.sim.1.dgp} 
We fix $p = 200$ and set $\bbeta = p^{-1/2} \cdot \bone_p$. 
We consider $n \in \{25, 50, 75, \dots, 175 \}$. 
For each $n$, we generate two covariatates $\bX$ as per Section~\ref{sec:sim.covariates}. 
For the spiked model, we choose $\sigma^2_x = 1$ and sample $\lambda_\ell$ independently from a uniform distribution over $[10, 20]$. 
We construct $\by = \bX \bbeta + \bvarepsilon$ by sampling the entries of $\bvarepsilon$ from a standard normal. 

\textit{Simulation results.} \label{sec:sim.1.results} 
For each $(\bX, \by)$, we estimate $\sigma^2$ via the LOO residual-based homoskedastic variance estimator $\hsigma^2$, as presented in \eqref{eq:var.estimator.hd}. 
Figure~\ref{fig:bias} displays the biases $\hsigma^2 - \sigma^2$ for each covariate generative model across sample sizes and averaged over $100$ trials, where each trial consists of an independent draw of $\bX$ and observations of size $100$; this amounts to $10000$ total simulation repeats per sample size. 
%

\subsection{Simulation II: varying dimension with fixed sample-to-dimension ratio} \label{sec:sim.bias.2} 
%

\textit{Data generating process.} \label{sec:sim.2.dgp}
We fix the aspect ratio as $n/p = 0.8$. 
For each $p \in \{200, 400, 600, 800, 1000\}$, we generate the true model $\bbeta = p^{-1/2} \cdot \bone_p$. 
Then, we generate two covariates $\bX$ as per Section~\ref{sec:sim.covariates} with the same parameters as selected in Section~\ref{sec.sim.1.dgp}. 
We generate $\by = \bX \bbeta + \bvarepsilon$ by sampling the entries of $\bvarepsilon$ from a standard normal.

\medskip
\textit{Simulation results.} \label{sec:sim.2.results} 
For each $(\bX, \by)$, we compute $\hsigma^2$ in \eqref{eq:var.estimator.hd}. 
Figure~\ref{fig:bias.2} visualizes the biases $\hsigma^2 - \sigma^2$ for each covariate generative model across covariate sizes and averaged over $100$ trials, where each trial is defined as in Section~\ref{sec:sim.1.results}. 
For the spiked model, the average bias is nearly zero across covariate sizes. 
For the standard normal model, the average bias is positive, indicating conservatism; this is consistent with Theorem~\ref{thm:var.estimator}. 
However, the variances under both models decrease as $p$ increases. 


\subsection{Simulation III: varying noise variance}  \label{sec:sim.bias.3} 
%

\textit{Data generating process.} 
We maintain a fixed aspect ratio of $n/p = 0.8$ with $p = 200$ and $n=160$. 
We set $\bbeta = p^{-1/2} \cdot \bone_p$ and generate tw covariates $\bX$ as per Section~\ref{sec:sim.covariates} with the same parameters as selected in Section~\ref{sec.sim.1.dgp}.  
We consider increasing noise $\sigma \in \{1, 2, 3, \dots, 10\}$. 
For each $\sigma$, we generate $\by = \bX \bbeta + \bvarepsilon$ by sampling $\bvarepsilon \sim \Normal(\bzero, \sigma^2 \cdot \bI_n)$.

\medskip
\textit{Simulation results.} 
For each $(\bX, \by)$, we compute the variance estimator $\hsigma^2$ in \eqref{eq:var.estimator.hd}. 
Figure~\ref{fig:bias.3} visualizes the biases for each covariate model across increasing levels of noise and averaged over $100$ trials, where each trial is defined as in Section~\ref{sec:sim.1.results}. 
Across covariate models, our findings indicate that the average bias is nearly zero but the variance increases as the noise increases. 
Specifically, our observations suggest that the bias primarily depends on the design matrix $\bX$ and parameter $\bbeta$, as implied by \eqref{eq:bias.hd} and \eqref{eq:homo.var.bias}.




\begin{figure*}
	\centering  
        \subfloat[Standard normal model.]{\includegraphics[width=0.4\textwidth]{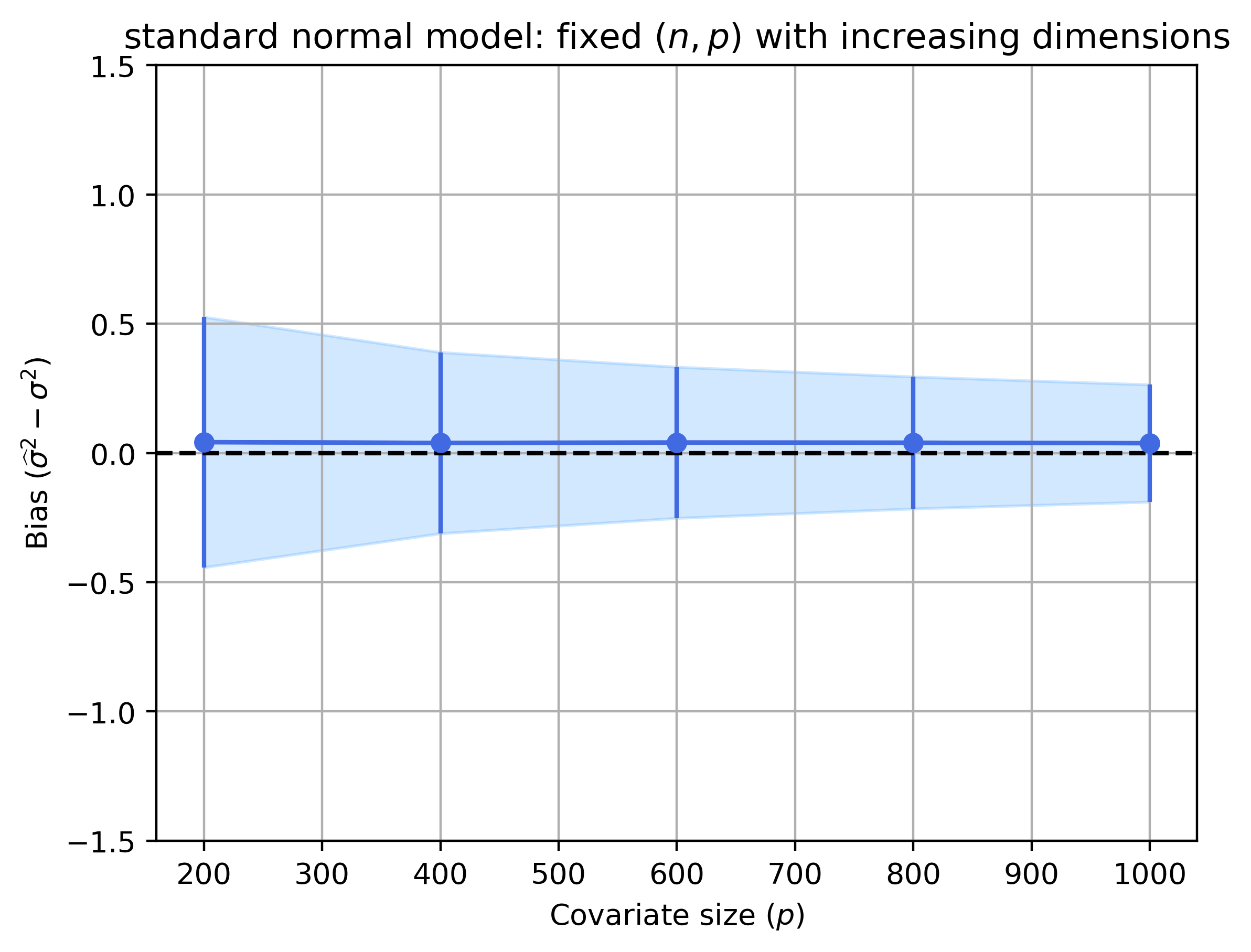}\label{fig:standard_normal.2}}
        \qquad \quad
        \subfloat[Spiked model.]{\includegraphics[width=0.4\textwidth]{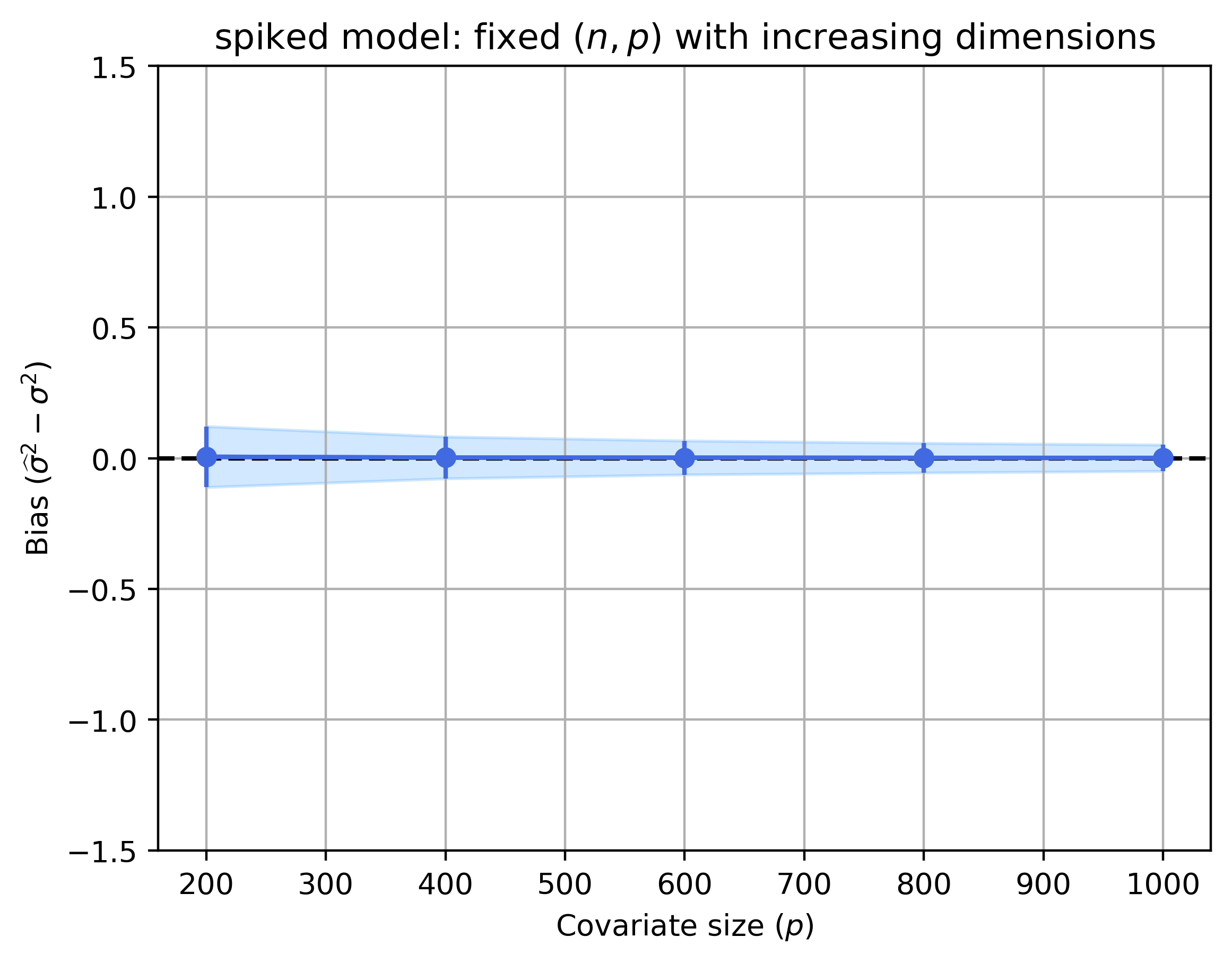}\label{fig:spiked.2}} 
	\caption{Simulation results of the biases of our variance estimator in \eqref{eq:var.estimator.hd} with fixed  $n/p = 0.8$ and varying $n \in \{200, 400, \dots, 1000 \}$. The horizontal solid lines show the mean over $100$ trials, with shading and vertical bars to show $\pm$ one standard error.}
	\label{fig:bias.2} 
\end{figure*}

\begin{figure*}
	\centering  
        \subfloat[Standard normal model.]{\includegraphics[width=0.4\textwidth]{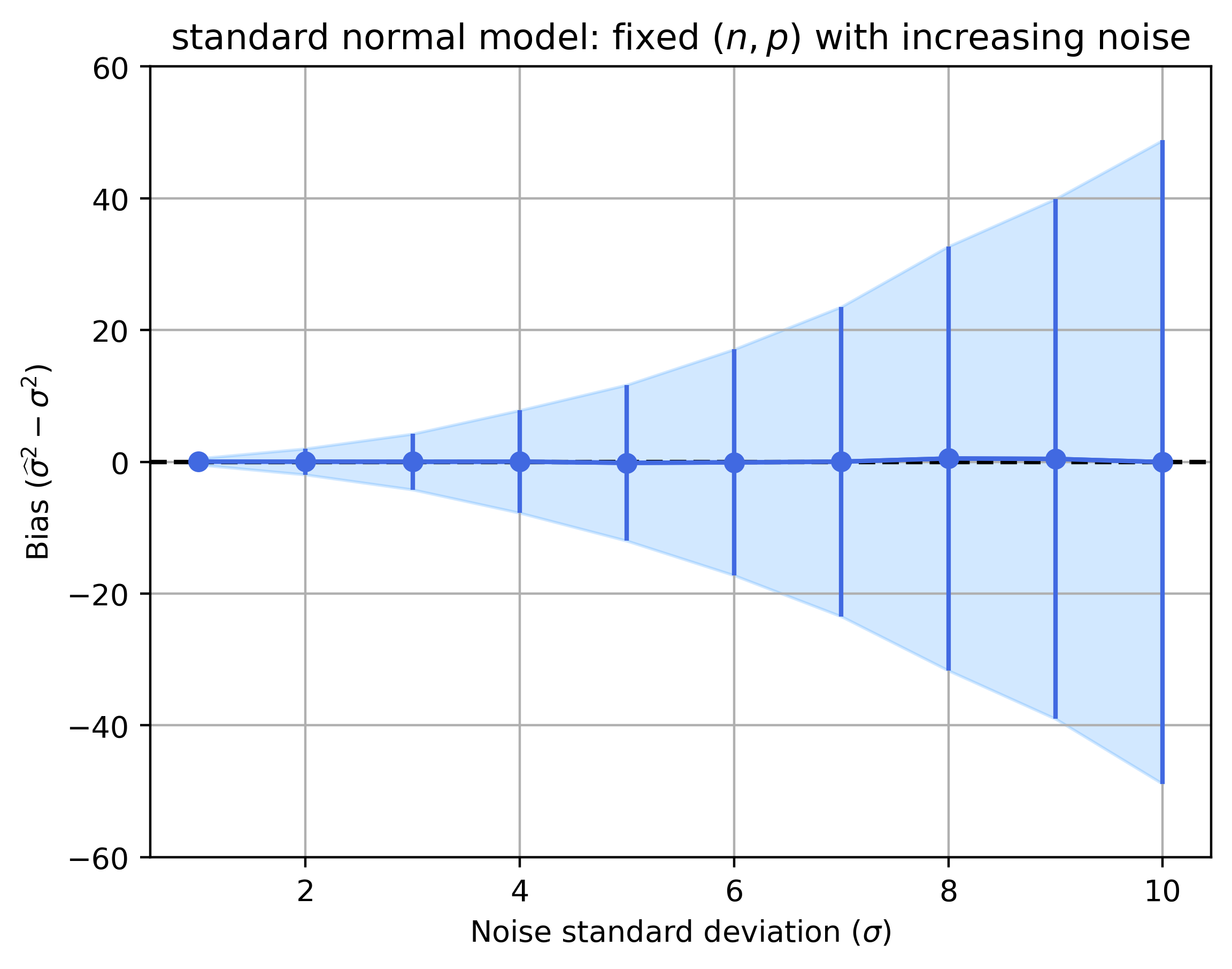}\label{fig:standard_normal.3}}
        \qquad \quad
        \subfloat[Spiked model.]{\includegraphics[width=0.4\textwidth]{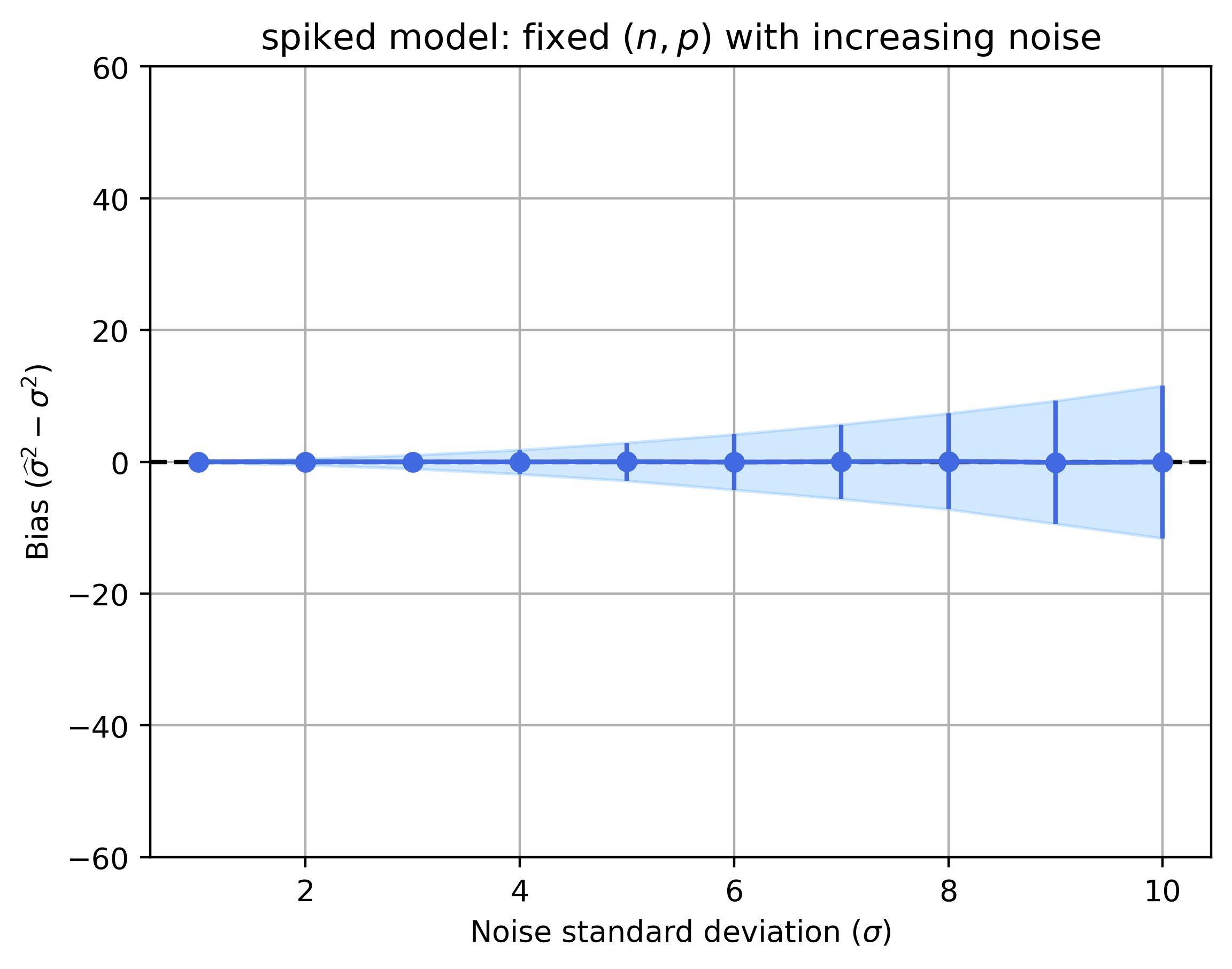}\label{fig:spiked.3}} 
	\caption{Simulation results of the biases of our variance estimator in \eqref{eq:var.estimator.hd} with fixed $n/p=160/200$ and varying $\sigma \in \{1, 2, 3, \dots, 10 \}$. The horizontal solid lines show the mean over $100$ trials, with shading and vertical bars to show $\pm$ one standard error.}
	\label{fig:bias.3} 
\end{figure*}

\subsection{Simulation IV: Coverage of the proposed variance estimator}\label{sec:simulations.extra} 
%

\textit{Data generating process.} 
%
%
Following Section~\ref{sec.sim.1.dgp}, we fix $p = 200$, set $\bbeta = p^{-1/2} \cdot \bone_p$, and vary the sample size $n \in \{25, 50, 75, \dots, 175 \}$. 
For each $n$, we generate our training covariates $\bX$ as per the two models outlined in Section~\ref{sec:sim.covariates} with the same parameters as chosen in Section~\ref{sec.sim.1.dgp}. 
We then construct the training responses $\by = \bX \bbeta + \bvarepsilon$ by either sampling 
(i) $\varepsilon_i \stackrel{\text{i.i.d.}}{\sim} \Normal(0, 1)$ 
or 
(ii) $\varepsilon_i \stackrel{\text{i.i.d.}}{\sim} \textsf{U}[0,1]$ for $i \in [n]$. 
For our test data, we use the same sampling procedure for $\bX$ to draw our test covariate $\bx_{n+1} \in \Rb^p$.  
Our parameter is the expected test response $\bbE[y_{n+1}] = \langle \bx_{n+1}, \bbeta \rangle$. 

\section{The Moore-Penrose Pseudoinverse} \label{sec:pseudo_properties} 
%
%
This section provides a brief technical background of the Moore-Penrose pseudoinverse, beginning with its formal definition below. 
For thorough and insightful discussions on the general theory and applications of the pseudoinverse, we refer the interested reader to \cite{albert1972regression}. 
\begin{definition}\label{defn:pseudo}
    For $\bM \in \RR^{m \times n}$, a pseudoinverse of $\bM$ is defined as a matrix $\bM^{\dagger} \in \RR^{n \times m}$ satisfying all of the following for criteria, known as the Moore-Penrose criteria:
    \begin{enumerate} [label=(\roman*)]
        \item 
        $\bM \bM^{\dagger} \bM = \bM$; 
        \item 
        $\bM^{\dagger} \bM \bM^{\dagger} = \bM^{\dagger}$; 
        \item
        $(\bM \bM^{\dagger})^{\top} = \bM \bM^{\dagger}$;
        \item 
        $(\bM^{\dagger} \bM)^{\top} = \bM^{\dagger} \bM$.
    \end{enumerate}
\end{definition}

For any matrix $\bM$, the pseudoinverse $\bM^{\dagger}$ exists and is unique \citep{golub2013matrix}; i.e., there is precisely one matrix $\bM^{\dagger}$ that satisfies the four properties of Definition \ref{defn:pseudo}. 
Moreover, it can be computed using the singular value decomposition; if $\bM = \bU \bS \bV^{\top}$ is a compact singular value decomposition, in which $\bS$ is a square diagonal matrix of size $r \times r$ where $r = \rank(\bM) \leq \min\{m,n\}$, the pseudoinverse takes the form $\bM^{\dagger} = \bV \bS^{-1} \bU^{\top}$.  

\medskip
\noindent \textit{Basic properties.} 
Here we review some useful properties of the pseudoinverse.
\begin{enumerate} [label=(\roman*)]
    \item 
    If $\bM$ is invertible, its pseudoinverse is its inverse, i.e., $\bM^{\dagger} = \bM^{-1}$.
    \item 
    The pseudoinverse of the pseudoinverse is the original matrix, i.e., $(\bM^{\dagger})^{\dagger} = \bM$.
    \item 
    Pseudoinverse commutes with transposition: $(\bM^{\top})^{\dagger} = (\bM^{\dagger})^{\top}$. Thus, we may use notations such as $\bM^{\dagger, \top}$ or $\bM^{\top, \dagger}$ exchangeably, omitting parentheses for notational brevity.
    \item 
    The pseudoinverse of a scalar multiple of $\bM$ is the reciprocal multiple of $\bM^{\dagger}$: $(\alpha \bM)^{\dagger} = \alpha^{-1} \bM^{\dagger}$ for $\alpha \neq 0$.
    \item 
    If a column-wise partitioned block matrix $\bM = \begin{bmatrix} \bA & \bB\end{bmatrix}$ is of full column rank (has linearly independent columns), then 
    \begin{align}
        \bM^{\dagger}
            &=\begin{bmatrix} \bA & \bB \end{bmatrix}^{\dagger}
            = \begin{bmatrix}   (\bP_{\bB}^{\perp}\bA)^{\dagger} \\ (\bP_{\bA}^{\perp}\bB)^{\dagger} \end{bmatrix},
                \nonumber\\
        \bM^{\top, \dagger}
            &= \begin{bmatrix} \bA^{\top} \\ \bB^{\top} \end{bmatrix}^{\dagger}
            = \begin{bmatrix}   (\bA^{\top}\bP_{\bB}^{\perp})^{\dagger} & (\bB^{\top}\bP_{\bA}^{\perp})^{\dagger} \end{bmatrix}.
                \label{eqn:block_pseudo}
    \end{align}
\end{enumerate} 

Recall the property $(\bA \bB )^{-1} = \bB^{-1} \bA^{-1}$, which holds for invertible matrices $\bA$ and $\bB$. A similar correspondence exists for the pseudoinverse under certain conditions, although it is not universally applicable. Here, we present several equivalent conditions for this relationship, outlined in \cite{greville1966note}, as a lemma.

\begin{lemma}
    Let $\bA, \bB$ be real matrices. The following are equivalent:
    \begin{enumerate} [label=(\roman*)]
        \item 
        $(\bA \bB)^{\dagger} = \bB^{\dagger} \bA^{\dagger}$.
        \item 
        $\bA^{\dagger}\bA\bB\bB^{\top}\bA^{\top} = \bB\bB^{\top}\bA^{\top}$ and $\bB\bB^{\dagger}\bA^{\top}\bA\bB = \bA^{\top}\bA\bB$.
        \item
        $\bA^{\dagger}\bA\bB\bB^{\top}$ and $\bA^{\top}\bA\bB\bB^{\dagger}$ are both symmetric.
        \item 
        $\bA^{\dagger}\bA\bB\bB^{\top}\bA^{\top}\bA\bB\bB^{\dagger} = \bB\bB^{\top}\bA^{\top}\bA$.
        \item 
        $\bA^{\dagger}\bA\bB = \bB(\bA\bB)^{\dagger}\bA\bB$ and $\bB\bB^{\dagger}\bA^{\top} = \bA^{\top} \bA\bB (\bA\bB)^{\dagger}$.
    \end{enumerate}
\end{lemma}

\noindent \textit{Perturbation theory.} 
We gather well-established results from matrix perturbation theory \cite[Theorem 3.2]{stewart1977perturbation} and present them as a lemma. 
\begin{lemma} \label{lem:pseudoinverse}
    Let $\bA, \bB \in \RR^{m \times n}$. Then
    \begin{align}
        \bB^{\dagger} - \bA^{\dagger}
            &= - \bB^{\dagger} \bP_{\bB} ( \bB - \bA ) \bP_{\bA^{\top}} \bA^{\dagger}
                + \bB^{\dagger} \bP_{\bB} \bP_{\bA}^{\perp} 
                - \bP_{\bB^{\top}}^{\perp} \bP_{\bA^{\top}} \bA^{\dagger},
                    \label{eqn:pinv_perturb.a}\\
        \bB^{\dagger} - \bA^{\dagger}
            &= - \bB^{\dagger} \bP_{\bB} ( \bB - \bA ) \bP_{\bA^{\top}} \bA^{\dagger}
                + ( \bB^\top \bB )^{\dagger} \bP_{\bB^{\top}} \left( \bB - \bA \right)^\top \bP_{\bA}^{\perp}
                \\ & \quad - \bP_{\bB^{\top}}^{\perp} ( \bB - \bA )^\top \bP_{\bA} ( \bA \bA^\top )^{\dagger}.
                    \label{eqn:pinv_perturb.b}
    \end{align}
\end{lemma}


\section{Deferred proofs from Section~\ref{sec:row_partition}} \label{sec:proofs.rows}
\subsection{Context}
In this section, we present all proofs postponed from Section~\ref{sec:row_partition}. 
Recall from Remark~\ref{remark:hd_classical} that we have presented results for both classical regime and high-dimensional regime in many of our theorem statements, but most of the classical results are already known in the literature. 
Thus, we primarily derive results for the high-dimensional regime. 
For those results within the classical regime that are novel, we will explicitly highlight its novelty and also provide its proof.

\subsection{Proof of Theorem \ref{thm:subsampled_rows}}\label{sec:proof_subsampled_ols}

\begin{proof}
    To prove \eqref{eq:beta.hd}, recall that $\tbX_{\cI,\star} \in \RR^{n \times p}$ is the matrix such that the $i$-th row of $\tbX_{\cI,\star}$ is equal to the $i$-th row of $\bX$ if and only if $i \in \cI$ ($\bzero_p$ otherwise). 
    Under Assumption \ref{assump:row_rank}, $\bX$ has independent rows. 
    Thus, we have 
    \begin{align*}
        \hbbeta^{(\cI, \star)} 
            &= (\bX_{\cI,\star})^{\dagger} \by_{\cI}\\
            &= \begin{bmatrix} (\bX_{\cI, \star})^{\dagger} & \bzero \end{bmatrix} \begin{bmatrix} \by_{\cI} \\ \by_{\cI^c} \end{bmatrix}\\
            &= \begin{bmatrix} \bX_{\cI, \star} \\ \bzero \end{bmatrix}^{\dagger} \begin{bmatrix} \by_{\cI} \\ \by_{\cI^c} \end{bmatrix}
                &&\because \eqref{eqn:block_pseudo}\\
            &= (\tbX_{\cI,\star})^{\dagger} \by.
    \end{align*}
    Therefore, $\hbbeta^{(\cI, \star)} - \hbbeta = \left( (\tbX_{\cI,\star})^{\dagger} - \bX^{\dagger} \right) \by$.

    Let $\cI^c \coloneqq [n] \setminus \cI$ denote the complement of $\cI$ in $[n]$. 
    Applying Lemma \ref{lem:pseudoinverse} with $\bA = \bX$ and $\bB = \tbX_{\cI,\star}$, we deduce from \eqref{eqn:pinv_perturb.a} that
    \begin{align}
        &\tbX_{\cI,\star}^{\dagger} - \bX^{\dagger}
            \\ &\qquad= - \underbrace{ (\tbX_{\cI,\star})^{\dagger} \bP_{\tbX_{\cI,\star}} }_{=(\tbX_{\cI,\star})^{\dagger} } \underbrace{ \big( \tbX_{\cI,\star} - \bX \big) }_{=-\tbX_{\cI^c,\star}}\underbrace{ \bP_{\bX^{\top}} \bX^{\dagger} }_{=\bX^{\dagger}}
                + (\tbX_{\cI,\star})^{\dagger} \underbrace{ \bP_{\tbX_{\cI,\star}} \bP_{\bX}^{\perp} }_{=\bzero}
                - \bP_{(\tbX_{\cI,\star})^{\top}}^{\perp} \underbrace{ \bP_{\bX^{\top}} \bX^{\dagger} }_{=\bX^{\dagger}}
                    \nonumber\\
            &\qquad= \left( (\tbX_{\cI,\star})^{\dagger} \tbX_{\cI^c,\star} - \bP_{(\tbX_{\cI,\star})^{\top}}^{\perp} \right) \bX^{\dagger}
                    \nonumber\\
            &\qquad= - \bP_{(\tbX_{\cI,\star})^{\top}}^{\perp} \bX^{\dagger}, 
                \label{eqn:proof_thm1_diff}
    \end{align}
    where the last equality follows from 
    \[
        (\tbX_{\cI,\star})^{\dagger} \tbX_{\cI^c,\star}
            = \left( (\tbX_{\cI,\star})^{\top} \tbX_{\cI,\star} \right)^{\dagger} (\tbX_{\cI,\star})^{\top} \tbX_{\cI^c,\star} 
            = \bzero.
    \]
    Therefore, we conclude that 
    \begin{align*}
        \hbbeta^{(\cI, \star)}
            &= \big( \hbbeta^{(\cI, \star)} - \hbbeta \big) + \hbbeta\\
            &= - \bP_{(\tbX_{\cI,\star})^{\top}}^{\perp} \bX^{\dagger} \by + \bX^{\dagger} \by  &&\because \eqref{eqn:proof_thm1_diff}\\
            &= \bP_{(\tbX_{\cI,\star})^{\top}} \bX^{\dagger} \by\\
            &= \bP_{(\tbX_{\cI,\star})^{\top}} \hbbeta              &&\because \hbbeta = \bX^{\dagger}\by\\
            &= \bP_{(\bX_{\cI,\star})^{\top}} \hbbeta               && \because \rowsp(\tbX_{\cI,\star}) = \rowsp(\bX_{\cI,\star})\\
            &= (\bX_{\cI,\star})^{\dagger} \bX_{\cI,\star} \hbbeta = \Pi_{\rowsp(\bX_{\cI,\star})}\big( \hbbeta \big).
    \end{align*}
    
    To prove \eqref{eq:beta.hd.2}, define two vector subspaces of $\RR^p$ as follows: 
    \begin{align*}
        \cV_1 &\coloneqq \vspan \left\{ \bX^{\top} \be_i: i \in \cI \right\} = \rowsp(\bX_{\cI,\star}) \subseteq \rowsp(\bX),\\
        \cV_2 &\coloneqq \vspan \left\{ \bX^{\dagger} \be_i: i \in \cI^c \right\} = \colsp\big( \bX^{\dagger}_{\star,\cI^c} \big) \subseteq \colsp(\bX^{\dagger}).
    \end{align*}
    Note that if Assumption \ref{assump:row_rank} holds, then the following three statements are true:
    \begin{itemize} 
        \item[(i)] \label{obs:full}
        $\rowsp(\bX) = \colsp(\bX^{\dagger})$ is an $n$-dimensional subspace of $\RR^p$.
        \item[(ii)] \label{obs:dimension}
        $\dim \cV_1 = |\cI|$ and $\dim \cV_2 = |\cI^c| = n - |\cI|$.
        \item[(iii)] \label{obs:orthogonal} 
        $\cV_1 \perp \cV_2$, i.e., $\langle v_1, v_2 \rangle = 0$ for all $(v_1, v_2) \in \cV_1 \times \cV_2$. 
        This can be readily verified by observing $\bP_{\bX} = \bI_n$.
    \end{itemize}
    Let $\cX \coloneqq \big\{ \bX^{\top} \be_i: i \in \cI \big\} \cup \big\{ \bX^{\dagger} \be_i: i \in \cI^c \big\} \subset \RR^p$ and observe that $\cX$ is linearly independent due to (ii) and (iii). 
    Since $\cX \subseteq \rowsp(\bX)$ and $|\cX| = \dim \rowsp(\bX)$, it follows from (i) that $\vspan(\cX) = \rowsp(\bX)$. 
    Note that $\vspan(\cX) = \cV_1 + \cV_2$. 
    Therefore, $\cV_1$ is the orthogonal complement of $\cV_2$ in $\rowsp(\bX)$. 
    Since $\hbbeta = \bX^{\dagger} \by \in \rowsp(\bX)$, we have $\Pi_{\cV_1}(\hbbeta) = \Pi_{\cV_2^{\perp}}(\hbbeta)$. 
    Consequently, 
    \begin{align*}
        \hbbeta^{(\cI, \star)}
            &= \Pi_{\cV_1}(\hbbeta)
                && \because \eqref{eq:beta.hd} \\
            &= \Pi_{\cV_2^{\perp}}(\hbbeta)\\
            &= \left\{ \bI_p - \bX^{\dagger}_{\star,\cI^c} \big(\bX^{\dagger}_{\star,\cI^c}\big)^{\dagger} \right\} \cdot \hbbeta.
    \end{align*}
    This gives our desired result.

\end{proof}

\subsection{Proof of Corollary~\ref{cor:loo}} \label{sec:proof.alt.beta.loo} 

\textit{Proof 1: A simple proof of Corollary~\ref{cor:loo}.} 
Corollary~\ref{cor:loo} immediately follows from \eqref{eq:beta.hd.2} of Theorem~\ref{thm:subsampled_rows}. 

\begin{proof} [Proof of Corollary \ref{cor:loo}]
Let $\cI = \{i\}^c$ and observe that $\bX^{\dagger}_{\star, \cI^c} = \bX^{\dagger}\be_i$. 
Then it suffices to observe that $\bv^{\dagger} = \| \bv \|_2^{-2} \cdot \bv^{\top}$ for any nonzero vector $\bv$ by the definition of pseudoinverse.
\end{proof} 

\textit{Proof 2: Alternative proof of Corollary~\ref{cor:loo}} 
We also provide an alternative, direct proof of Corollary~\ref{cor:loo} in Appendix~\ref{sec:proof.alt.beta.loo}. 
This alternative proof relies on the generalized Sherman-Morrison formula for the pseudoinverse \citep{gen_inverse_meyers}, which may be of independent interest. 

We present an alternative, direct proof of Corollary~\ref{cor:loo}. 
First, we state a helpful classical result on the pseudoinverse of a rank-one perturbation of a matrix \cite[Theorem 6]{gen_inverse_meyers}. 
 \begin{lemma} 
 \label{lemma:meyer}
     Let $\bA \in \RR^{n \times p}$, $\bc \in \RR^n$, and $\bd \in \RR^p$. 
     If $\bc \in \colsp(\bA)$, $\bd \in \rowsp(\bA)$, and $1 + \bd^{\top} \bA^\dagger \bc = 0$, then
     \begin{align}
     	(\bA + \bc \bd^{\top})^\dagger = \bA^\dagger - \bk \bk^\dagger \bA^\dagger - \bA^\dagger \bh^{\dagger, \top} \bh^{\top} 
     	+ (\bk^\dagger \bA^\dagger \bh^{\dagger,\top}) \bk \bh^{\top}, \label{eq:meyer.1} 
     \end{align} 
     where $\bk = \bA^\dagger \bc \in \RR^{p}$ and $\bh = \bA^{\dagger, \top} \bd \in \RR^{n}$. 
 \end{lemma}

\begin{proof} [Alternative proof of Corollary \ref{cor:loo}]
Let $\bM = \bX \bX^{\top}$ be the Gram matrix of $\{ \bx_1, \dots, \bx_n \} \subset \RR^p$, the rows of $\bX$, such that $\bM_{ij} = \langle \bx_i, \bx_j \rangle$. 
Recall that $\bX_{\sim i} \coloneqq \bX_{\{i\}^c,\star} \in \RR^{(n-1) \times p}$ and $\by_{\sim i} = \by_{\{i\}^c} \in \Rb^{n-1}$ denote the leave-$i$-out data. 
Observe that
\begin{equation}\label{eqn:loo.products}
    \bX_{\si}^{\top} \bX_{\si} = \bX^{\top} \bX - \bx_i \bx_i^{\top}
    \qquad\text{and}\qquad
    \bX_{\si}^{\top} \by_{\si} = \bX^{\top} \by - y_i \bx_i = \bX^{\top} \big( \bI_n - \be_i \be_i^{\top} \big) \by,
\end{equation}
where we recall that $\be_i \in \RR^n$ is the $i$-th standard basis vector in $\RR^n$.

We apply Lemma \ref{lemma:meyer} with $\bA = \bX^{\top} \bX$, $\bc = \bx_i = \bX^{\top} \be_i$, and $\bd = -\bx_i = - \bX^{\top} \be_i$. Then, 
\[
    \bk = \bA^{\dagger} \bc = \big( \bX^{\top} \bX \big)^{\dagger} \bX^{\top} \be_i = \bX^{\dagger} \be_i
    \qquad\text{and}\qquad
    \bh = \bA^{\dagger,\top} \bd = - \bX^{\dagger} \be_i
\]
because $\bA^{\dagger} = \bX^{\dagger} (\bX^{\top})^{\dagger}$. 
Since $\bv^{\dagger} = \| \bv \|_2^{-2} \cdot \bv^{\top}$ for any vector $\bv$, we observe that 
\begin{align*}
    \bk \bk^{\dagger} \bA^{\dagger} 
         &= \frac{ \bk \bk^{\top} }{ \| \bk \|_2^2 } \bA^{\dagger}  
         = \frac{ \bX^{\dagger} \be_i \be_i^{\top} (\bX^{\dagger})^{\top} }{ \be_i^{\top} (\bX^{\dagger})^{\top} \bX^{\dagger} \be_i } \big( \bX^{\top} \bX \big)^{\dagger}
         = \bX^{\dagger} \frac{ \be_i \be^{\top}_i}{\be^{\top}_i \bM^{\dagger} \be_i } \bM^{\dagger} ( \bX^{\top} )^{\dagger},\\
    \bA^{\dagger} \bh^{\dagger, \top} \bh^{\top} 
         &= \bA^{\dagger} \frac{ \bh \bh^{\top} }{ \| \bh \|_2^2 }  
         =  \big( \bX^{\top} \bX \big)^{\dagger} \frac{ \bX^{\dagger} \be_i \be_i^{\top} (\bX^{\dagger})^{\top} }{ \be_i^{\top} (\bX^{\dagger})^{\top} \bX^{\dagger} \be_i } 
         =  \bX^{\dagger} \bM^{\dagger}  \frac{ \be_i \be^{\top}_i}{\be^{\top}_i \bM^{\dagger} \be_i } ( \bX^{\top} )^{\dagger},\\
    (\bk^\dagger \bA^\dagger \bh^{\dagger,\top}) \bk \bh^{\top}
         &= \frac{\bk^{\top} \bA^{\dagger} \bh }{ \| \bk\|_2^2 \| \bh \|_2^2 } \bk \bh^{\top}
         = \frac{ \be_i^{\top} (\bX^{\dagger})^{\top} \big( \bX^{\top} \bX )^{\dagger} \bX^{\dagger} \be_i }{ \big( \be_i^{\top} \bM^{\dagger} \be_i \big)^2 }
             \cdot \bX^{\dagger} \be_i \be_i^{\top} (\bX^{\dagger})^{\top}
         \\ & = \frac{ \be_i^{\top} {\bM^{\dagger}}^2 \be_i }{ \big( \be_i^{\top} \bM^{\dagger} \be_i \big)^2 }
             \cdot \bX^{\dagger} \be_i \be_i^{\top} (\bX^{\dagger})^{\top}.
\end{align*}
Thereafter, applying Lemma~\ref{lemma:meyer}, we obtain that
\begin{align}
    \big( \bX_{\si}^{\top} \bX_{\si} \big)^{\dagger}
         &= \big(\bX^{\top} \bX - \bx_i \bx_i^{\top}\big)^{\dagger} \\
         &= \big( \bX^{\top} \bX \big)^{\dagger}
             - \bX^{\dagger} \frac{ \be_i \be^{\top}_i}{\be^{\top}_i \bM^{\dagger} \be_i } \bM^{\dagger} ( \bX^{\top} )^{\dagger}
             - \bX^{\dagger} \bM^{\dagger}  \frac{ \be_i \be^{\top}_i}{\be^{\top}_i \bM^{\dagger} \be_i } ( \bX^{\top} )^{\dagger}
             + \frac{ \be_i^{\top} {\bM^{\dagger}}^2 \be_i }{ \big( \be_i^{\top} \bM^{\dagger} \be_i \big)^2 }
             \cdot \bX^{\dagger} \be_i \be_i^{\top} (\bX^{\dagger})^{\top}\\
         &= \bX^{\dagger} \left\{ \bI_n -  \frac{ \be_i \be^{\top}_i}{\be^{\top}_i \bM^{\dagger} \be_i } \bM^{\dagger} - \bM^{\dagger}  \frac{ \be_i \be^{\top}_i}{\be^{\top}_i \bM^{\dagger} \be_i } + \frac{ \be_i^{\top} {\bM^{\dagger}}^2 \be_i }{ \big( \be_i^{\top} \bM^{\dagger} \be_i \big)^2 } \be_i \be_i^{\top} \right\} (\bX^{\dagger})^{\top}.
         \label{eq:loo.pseudo}
\end{align}

Combining \eqref{eqn:loo.products} and \eqref{eq:loo.pseudo}, we have
\begin{align}
    \hbbeta^{(\si)} 
         &= \big( \bX_{\si}^{\top} \bX_{\si} \big)^{\dagger} \bX_{\si}^{\top} \by_{\si}\\
         &= \bX^{\dagger} \left\{ \bI_n -  \frac{ \be_i \be^{\top}_i}{\be^{\top}_i \bM^{\dagger} \be_i } \bM^{\dagger} - \bM^{\dagger}  \frac{ \be_i \be^{\top}_i}{\be^{\top}_i \bM^{\dagger} \be_i } + \frac{ \be_i^{\top} {\bM^{\dagger}}^2 \be_i }{ \big( \be_i^{\top} \bM^{\dagger} \be_i \big)^2 } \be_i \be_i^{\top} \right\} \underbrace{ (\bX^{\dagger})^{\top} \bX^{\top} }_{= \bI_n } \big( \bI_n - \be_i \be_i^{\top} \big) \by\\
         &= \bX^{\dagger} \left\{ \bI_n - \frac{ \be_i \be^{\top}_i}{\be^{\top}_i \bM^{\dagger} \be_i } \bM^{\dagger} \right\} \by\\
         &= \left\{ \bI_n - \frac{ \big( \bX^{\dagger} \be_i \big) \cdot \big( \bX^{\dagger} \be_i \big)^{\top} }{ \big\| \bX^{\dagger} \be_i \big\|_2^2 }  \right\} \bX^{\dagger} \by.       \qquad\qquad\because \bM^{\dagger} = (\bX^{\dagger})^{\top} \bX^{\dagger}
\end{align}
Observing $\hbbeta = \bX^{\dagger} \by$ completes the proof.
 \end{proof} 

\subsection{Proof of Corollary~\ref{cor:loo.shortcut}} \label{sec:proof_loo_residual}

\begin{proof} 
Recall from Assumption~\ref{assump:row_rank} that we have $\by = \bX \hbbeta$. 
Coupled with Corollary~\ref{cor:loo}, we rewrite the leave-$i$-out prediction residual as
\begin{align}
	\tvarepsilon_i &= \bx_i^\top \cdot \left( \hbbeta - \hbbeta^{(\si)} \right)
	\\
	&= \bx_i^\top \cdot \left\{ \frac{(\bX^\top \bX)^\dagger \bx_i \bx_i^\top (\bX^\top \bX)^\dagger}{\| (\bX^\top \bX)^\dagger \bx_i \|_2^2} \right\} \cdot \hbbeta
	\\
	&= \frac{\bx_i^\top (\bX^\top \bX)^\dagger}{\| (\bX^\top \bX)^\dagger \bx_i \|_2^2} \cdot  \hbbeta, \label{eq:loo.proof.1}
\end{align} 
where the final equality follows since $\bx_i^\top (\bX^\top \bX)^\dagger \bx_i = 1$, cf. Remark \ref{rem:hat_matrix}. 
Noting that $\bx_i = \bX^\top \be_i$ and recalling $\hbbeta = (\bX^\top \bX)^\dagger \bX^\top \by$, we further simplify \eqref{eq:loo.proof.1} as 
\begin{align}
	\tvarepsilon_i &= \frac{\be_i^\top (\bX \bX^\top)^{-1} \by}{\be_i^\top (\bX \bX^\top)^{-1} \be_i}. 
\end{align} 
To complete the proof, let $\bD = \diag(\bG_{\bX}) $ and observe that for all $i \in [n]$,
\begin{align}
    \be_i^\top \left( \bD^{-1} \cdot \bG_{\bX} \by \right)
        &= \frac{1}{ \be_i^\top \bD \be_i} \be_i^\top \cdot  (\bX \bX^\top)^{-1} \by
            &&\because \bD^{-1} \be_i = \frac{1}{\be_i^\top \bD \be_i} \be_i\\
        &= \frac{\be_i^\top \big( \bX \bX^\top \big)^{-1} \by }{ \be_i^\top \big( \bX \bX^\top \big)^{-1} \be_i }\\
        &= \tvarepsilon_i.
\end{align}
Therefore, $\tbvarepsilon =  \big[ \diag(\Gram_{\bX}) \big]^{-1} \cdot \Gram_{\bX}  \by$. 
\end{proof}

%

\section{Deferred proofs from Section~\ref{sec:subsampled_column}} \label{sec:proofs.columns}

\subsection{Proof of Theorem~\ref{thm:partially_regularized_OLS}} \label{sec:proof.partial.reg}
We prove Theorem~\ref{thm:partially_regularized_OLS} in two steps. 
First, we prove the formula \eqref{eq:fwl.j.hd.0}, and thereafter, we prove the additional outcomes in \eqref{eq:fwl.jc.hd} and Remark~\ref{remark:fwl.extension}.

\textit{I. Proof of the high-dimensional FWL theorem}
\begin{proof}[Proof of the formula \eqref{eq:fwl.j.hd.0}]
Recall that Assumption~\ref{assump:row_rank} implies $\rank(\bX) = n$.
As such, $\hbbeta$ satisfies the interpolating property: 
\begin{align}
	 \by = \bW \hbbeta^{[\cJ]}_{\cJ} + \bT \hbbeta^{[\cJ]}_{\cJ^c}. \label{eq:fwl.interpolate} 
\end{align} 
Next, we plug the decomposition $\bW = \bP_{\bT} \bW + \bP_{\bT}^{\perp} \bW$ into \eqref{eq:fwl.interpolate} to obtain 
\begin{align}
            \by &= \bP_{\bT}^{\perp} \bW \hbbeta^{[\cJ]}_{\cJ} + \bP_{\bT} \bW \hbbeta^{[\cJ]}_{\cJ} + \bT \hbbeta^{[\cJ]}_{\cJ^c}.  \label{eqn:interpol}
\end{align}
Multiplying $\bP_{\bT}^\perp$ to both sides of \eqref{eqn:interpol} from the left yields 
\begin{align}
    \bP_{\bT}^{\perp} \by 
        = \bP_{\bT}^{\perp} \bW \hbbeta^{[\cJ]}_{\cJ}    \label{eqn:interpol.2}
\end{align}
because 
(i) $(\bP_{\bT}^\perp)^2 = \bP_{\bT}^\perp$,
(ii) $\bP_{\bT}^{\perp}\bP_{\bT} = \bzero$,
and (iii) $\bP_{\bT}^\perp \bT = \bzero$. 
Consequently, it follows from the definition of the $\cJ$-partially regularized OLS estimator in \eqref{eqn:j_partial} that $\hbbeta^{[\cJ]}_{\cJ} $ is the minimum $\ell_2$-norm solution among the interpolators, i.e., $\hbbeta^{[\cJ]}_{\cJ} 
        = \olsmin ( \bP_{\bT}^{\perp} \bW, \bP_{\bT}^{\perp} \by ) = ( \bP_{\bT}^{\perp} \bW )^{\dagger} \bP_{\bT}^{\perp} \by$. 
%
\end{proof}

\textit{II. Proof of supplementary results in Theorem~\ref{thm:partially_regularized_OLS}}
We establish the additional results presented in \eqref{eq:fwl.jc.hd} and Remark~\ref{remark:fwl.extension} to complete the proof of Theorem \ref{thm:partially_regularized_OLS}. 
This part of the proof hinges on the utilization of block matrix inversion through the application of the Schur complement \cite[Chapter A.5.5]{boyd2004convex}, which is stated as a lemma below.
\begin{lemma}
\label{lem:schur}
    Let $\bM = \begin{bmatrix} \bA & \bB \\ \bB^{\top} & \bC \end{bmatrix}$ be a symmetric block partitioned matrix such that $\det \bA \neq 0$. 
    Let $\bS = \bC - \bB^{\top} \bA^{-1} \bB$ is the Schur complement of $\bA$ in $\bM$. 
    Then $\det \bM = \det \bA \cdot \det \bS$, and moreover, if $\det \bS \neq 0$, then
    \[
        \bM^{-1} 
            = \begin{bmatrix} \bA & \bB \\ \bB^{\top} & \bC \end{bmatrix}^{-1}
            = \begin{bmatrix} \bA^{-1} + \bA^{-1} \bB \bS^{-1} \bB^{\top} \bA^{-1} & - \bA^{-1} \bB \bS^{-1} \\
            -\bS^{-1} \bB^{\top} \bA^{-1} & \bS^{-1}
            \end{bmatrix}.
    \]
\end{lemma}

\begin{proof}[Completing the proof of Theorem \ref{thm:partially_regularized_OLS}]
    To ensure clarity and accessibility, we present this proof in three steps, which are outlined below. 
    In Step 1, we establish the necessary and sufficient conditions for the optimality of $\hbbeta^{[\cJ]}$ in the partially regularized OLS problem \eqref{eqn:j_partial}. 
    These conditions are expressed as a system of linear equations involving both the primal and dual variables, known as the Karush-Kuhn-Tucker (KKT) system; see \eqref{eqn:KKT_system.2} below. 
    In Step 2, we employ the Schur complement to compute the inverse of the KKT matrix, yielding the expression presented in \eqref{eqn:inverse_KKT_matrix}. 
    In Step 3, we utilize the inverse KKT matrix obtained in Step 2 to solve the KKT system from Step 1, thereby concluding the proof. 
    This step-by-step approach aims to enhance the understanding and accessibility of the proof.

    \textit{Step 1. Expressing the optimality condtions as a KKT system.}
    Let $q = |\cJ|$. Without loss of generality, we may assume $\cJ = [q] \subseteq [p]$ by an appropriate permutation of coordinates, if necessary. 
    Then we observe that Assumption \ref{assump:partial} implies Assumption \ref{assump:row_rank}, and thus, $\hbbeta^{[\cJ]}$ is the solution of the following optimization problem:
    \begin{equation}\label{eqn:prob_partial}
        \text{minimize}\quad \frac{1}{2} \bbeta^{\top} \bQ \bbeta
        \qquad
        \text{subject to}\quad \bX \bbeta = \by,
    \end{equation}
    where $\bQ = \begin{bmatrix} \bI_q & \bzero \\ \bzero & \bzero \end{bmatrix}$. 
    The Lagrangian of the problem \eqref{eqn:prob_partial} is given as
    \[
        \cL( \bbeta, \bnu ) = \frac{1}{2} \bbeta^{\top} \bQ \bbeta + \bnu^{\top}  \big( \bX \bbeta - \by ).
    \]
    It is well known that $\bbeta = \bbetaopt \in \RR^p$ is optimal for the problem \eqref{eqn:prob_partial} if and only if there exists a dual certificate $\bnuopt \in \RR^n$ such that 
    \begin{align*}
        \nabla_{\bbeta} \cL(\bbetaopt, \bnuopt) &= \bQ \bbetaopt + \bX^{\top} \bnuopt = \bzero\\
        \nabla_{\bnu} \cL(\bbetaopt, \bnuopt) &= \bX \bbetaopt - \by = \bzero,
    \end{align*}
    which are called the KKT conditions in the optimization literature.
    These optimality conditions can be written as a system of $p+n$ linear equations in the $p+n$ variables $\bbetaopt$, $\bnuopt$:
    \begin{equation}\label{eqn:KKT_system}
        \begin{bmatrix} \bQ & \bX^{\top} \\ \bX & \bzero \end{bmatrix}
        \begin{bmatrix} \bbetaopt \\ \bnuopt \end{bmatrix}
        =
        \begin{bmatrix} \bzero \\ \by \end{bmatrix}.
    \end{equation}

    Recall we introduced $\bW = \bX_{*,\cJ}$ and $\bT = \bX_{*,\cJ^c}$ for a shorthand notation. 
    We rewrite the KKT system \eqref{eqn:KKT_system} as
    \begin{equation}\label{eqn:KKT_system.2}
        \underbrace{\begin{bmatrix} \bI_q & \bzero & \bW^{\top} \\ \bzero & \bzero & \bT^{\top} \\ \bW & \bT & \bzero \end{bmatrix}}_{=: \bM}
        \begin{bmatrix} \bbetaopt_{\cJ} \\ \bbetaopt_{\cJ^c} \\ \bnuopt \end{bmatrix}
        =
        \begin{bmatrix} \bzero \\ \bzero \\\by \end{bmatrix}.
    \end{equation}

    \textit{Step 2. Computing the inverse of the KKT matrix $\bM$.} 
    Observe that $\bA := \bI_q$ is invertible and so is the Schur complement of the block $\bA$ of the matrix $\bM$, i.e., $\bM / \bA := \begin{bmatrix} \bzero & \bT^{\top} \\ \bT & \bzero \end{bmatrix} - \begin{bmatrix} \bzero \\ \bW\end{bmatrix} \bI_q^{-1} \begin{bmatrix} \bzero & \bW^{\top} \end{bmatrix} = \begin{bmatrix} \bzero & \bT^{\top} \\ \bT & - \bW \bW^{\top} \end{bmatrix}$. 
    Here we prove the invertibility of $\bM/\bA$.
    \begin{quote}
        \emph{Proof of the invertibility of $\bM/\bA$.}
        Assume that $\bM/\bA$ is not invertible. 
        Then $\cN(\bM/\bA) \neq \{\bzero\}$, and there exists a nonzero vector $\bv = \begin{bmatrix} \bv_1 \\ \bv_2 \end{bmatrix}$ such that $\bv_1 \in \RR^{p-q}$, $\bv_2 \in \RR^n$, and $(\bM/\bA) \bv = \bzero$. 
        That is,
        \[
            \begin{bmatrix} \bzero & \bT^{\top} \\ \bT & - \bW \bW^{\top} \end{bmatrix} \begin{bmatrix} \bv_1 \\ \bv_2 \end{bmatrix}
                = \begin{bmatrix} \bT^{\top} \bv_2 \\ \bT \bv_1 - \bW \bW^{\top} \bv_2 \end{bmatrix} 
                = \bzero.
        \]
        It follows that 
        \[
            \bv_2^{\top} \left( \bT \bv_1 - \bW \bW^{\top} \bv_2 \right)
                = \big( \bT^{\top} \bv_2 \big)^{\top} \bv_1 - \bv_2^{\top} \bW \bW^{\top} \bv_2 
                = - \big\| \bW^{\top} \bv_2 \big\|_2^2
                = 0.
        \]
        This implies that $\bW^{\top} \bv_2 = \bzero$. 
        Since $\rank(\bW) = n$ due to Assumption \ref{assump:partial}, we must have $\bv_2 = \bzero$, and therefore, $\bv_1 \neq \bzero$. 
        However, this yields $\bT \bv_1 - \bW \bW^{\top} \bv_2 = \bT \bv_1 = \bzero$, which contradicts the assumption that $\dim \rowsp(\bT) = \rank(\bT) = p-q$. 
        Consequently, $\cN(\bM/\bA) = \{\bzero\}$, and $\bM/\bA$ is invertible.
    \end{quote}
    By Lemma \ref{lem:schur}, since $\bA$ and $\bM/\bA$ are both invertible, $\bM$ is invertible, and moreover, its inverse can be computed using the Schur complement as
    \begin{align}
        \bM^{-1} 
            &= \begin{bmatrix} 
                \bA^{-1} + \bA^{-1} \bB \big( \bM/\bA\big)^{-1} \bB^{\top} \bA^{-1} &   - \bA^{-1} \bB \big( \bM/\bA \big)^{-1} \\
                - \big( \bM/\bA \big)^{-1} \bB^{\top} \bA^{-1}  & \big( \bM/\bA \big)^{-1}
                \end{bmatrix}
                \quad\text{where}\quad
                \bB =
                \begin{bmatrix} 0 & \bW^{\top}  \end{bmatrix}
                    \nonumber\\
            &= \begin{bmatrix}
                    \bI_q + \bB \big( \bM/\bA\big)^{-1} \bB^{\top} & - \bB \big( \bM/\bA \big)^{-1}\\
                    - \big( \bM/\bA \big)^{-1} \bB^{\top}   & \big( \bM/\bA \big)^{-1}
                \end{bmatrix}.
                    \label{eqn:KKT_inverse}
    \end{align}
    In order to compute the inverse of the Schur complement, $\big( \bM/\bA \big)^{-1}$, we let 
    \[
        \tbM := \bM/\bA = \begin{bmatrix} \bzero & \bT^{\top} \\ \bT & -\bW\bW^{\top} \end{bmatrix}
    \]
    and observe that $\bW \bW^{\top}$ is invertible due to the assumption that $\rank(\bW) = n$. 
    The Schur complement of $-\bW\bW^{\top}$ in $\tbM$, which we denote by $\tbS$, is given as 
    \[
        \tbS = \bzero - \bT^{\top} \big( - \bW\bW^{\top} \big)^{-1} \bT = \bT^{\top} \bW^{\dagger, \top} \bW^{\dagger} \bT = \big( \bW^{\dagger} \bT \big)^{\top} \big( \bW^{\dagger} \bT \big).
    \]
    By Lemma \ref{lem:schur}, we obtain
    \begin{align*}
        \tbM^{-1}
            = \begin{bmatrix}
                \tbS^{-1}   
                    &   -\tbS^{-1} \bT^{\top} \big(-\bW\bW^{\top} \big)^{-1}\\
                 - \big(-\bW\bW^{\top} \big)^{-1} \bT \tbS^{-1} 
                    & \big(-\bW\bW^{\top} \big)^{-1} + \big(-\bW\bW^{\top} \big)^{-1} \bT \tbS^{-1} \bT^{\top} \big(-\bW\bW^{\top} \big)^{-1}
            \end{bmatrix}.
    \end{align*}
    Then we observe that
    \begin{align*}
        -\tbS^{-1} \bT^{\top} \big(-\bW\bW^{\top} \big)^{-1}
            &= \tbS^{-1} \bT^{\top} \bW^{\dagger, \top} \bW^{\dagger}\\
            &= \big( \bW^{\dagger} \bT \big)^{\dagger} \big( \bW^{\dagger} \bT \big)^{\top, \dagger} \big( \bW^{\dagger} \bT \big)^{\top} \bW^{\dagger}\\
            &= \big( \bW^{\dagger} \bT \big)^{\dagger} \bW^{\dagger},
                \qquad \because \big( \bW^{\dagger} \bT \big)^{\top, \dagger} \big( \bW^{\dagger} \bT \big)^{\top} = \bP_{\bW^{\dagger}\bT} \\
        - \big(-\bW\bW^{\top} \big)^{-1} \bT \tbS^{-1}
            &= \bW^{\top, \dagger} \big( \bW^{\dagger} \bT \big) \big( \bW^{\dagger} \bT \big)^{\dagger} \big( \bW^{\dagger} \bT \big)^{\top, \dagger}\\
            &= \bW^{\top, \dagger} \big( \bW^{\dagger} \bT \big)^{\top, \dagger}.
    \end{align*}
    Likewise,
    \begin{align*}
        &\big(-\bW\bW^{\top} \big)^{-1} + \big(-\bW\bW^{\top} \big)^{-1} \bT \tbS^{-1} \bT^{\top} \big(-\bW\bW^{\top} \big)^{-1}
            \\ &= - \bW^{\top,\dagger}\bW^{\dagger} + \bW^{\top,\dagger}\bW^{\dagger} \bT \cdot \big( \bW^{\dagger} \bT \big)^{\dagger} \bW^{\dagger}\\
            &= - \bW^{\top,\dagger} \left( \bI_q - \bP_{\bW^{\dagger} \bT} \right) \bW^{\dagger} 
    \end{align*}
    As a result, we have
    \begin{equation}\label{eqn:inverse_schur_complement}
        \tbM^{-1}
            = \begin{bmatrix}
                 \big( \bW^{\dagger} \bT \big)^{\dagger}  \big( \bW^{\dagger} \bT \big)^{\top, \dagger}
                    & \big( \bW^{\dagger} \bT \big)^{\dagger} \bW^{\dagger}\\
                \bW^{\top, \dagger} \big( \bW^{\dagger} \bT \big)^{\top, \dagger}
                    & - \bW^{\top,\dagger} \left( \bI_q - \bP_{\bW^{\dagger} \bT} \right) \bW^{\dagger} 
            \end{bmatrix}.
    \end{equation}

    Finally, we insert the expression \eqref{eqn:inverse_schur_complement} for $\tbM^{-1} = \big( \bM / \bA \big)^{-1}$ into \eqref{eqn:KKT_inverse}. 
    Observe that 
    \begin{align*}
        - \bB \big( \bM/\bA\big)^{-1}
            &= - \begin{bmatrix} \bzero & \bW^{\top}  \end{bmatrix}
                \begin{bmatrix}
                 \big( \bW^{\dagger} \bT \big)^{\dagger}  \big( \bW^{\dagger} \bT \big)^{\top, \dagger}
                    & \big( \bW^{\dagger} \bT \big)^{\dagger} \bW^{\dagger}\\
                \bW^{\top, \dagger} \big( \bW^{\dagger} \bT \big)^{\top, \dagger}
                    & - \bW^{\top,\dagger} \left( \bI_q - \bP_{\bW^{\dagger} \bT} \right) \bW^{\dagger} 
            \end{bmatrix}\\
            &= \begin{bmatrix} 
                - \bW^{\top} \bW^{\top, \dagger} \big( \bW^{\dagger} \bT \big)^{\top, \dagger} 
                & \bW^{\top} \bW^{\top,\dagger} \left( \bI_q - \bP_{\bW^{\dagger} \bT} \right) \bW^{\dagger} 
                \end{bmatrix}\\
            &= \bP_{\bW^{\top}} \begin{bmatrix}  - \big( \bW^{\dagger} \bT \big)^{\top, \dagger}  &  \bP_{\bW^{\dagger} \bT}^{\perp} \bW^{\dagger}   \end{bmatrix},\\
        \bB \big( \bM/\bA\big)^{-1} \bB^{\top}
            &= \bP_{\bW^{\top}} \begin{bmatrix}  \big( \bW^{\dagger} \bT \big)^{\top, \dagger}  &  - \bP_{\bW^{\dagger} \bT}^{\perp} \bW^{\dagger}   \end{bmatrix}
                \begin{bmatrix} \bzero \\ \bW  \end{bmatrix}\\
            &= - \bP_{\bW^{\top}} \bP_{\bW^{\dagger} \bT}^{\perp} \bP_{\bW^{\top}}.
    \end{align*}
    Therefore, we obtain
    \begin{equation}\label{eqn:inverse_KKT_matrix}
        \bM^{-1}
            = \begin{bmatrix}
                \bI_q - \bP_{\bW^{\top}} \bP_{\bW^{\dagger} \bT}^{\perp} \bP_{\bW^{\top}}   &   - \bP_{\bW^{\top}} \big( \bW^{\dagger} \bT \big)^{\top, \dagger}  & \bP_{\bW^{\top}} \bP_{\bW^{\dagger} \bT}^{\perp} \bW^{\dagger} \\
                - \big( \bW^{\dagger} \bT \big)^{\dagger} \bP_{\bW^{\top}}    &   \big( \bW^{\dagger} \bT \big)^{\dagger}  \big( \bW^{\dagger} \bT \big)^{\top, \dagger}
                    & \big( \bW^{\dagger} \bT \big)^{\dagger} \bW^{\dagger}\\
                 \bW^{\dagger, \top} \bP_{\bW^{\top}} \bP_{\bW^{\dagger} \bT}^{\perp} &   \bW^{\top, \dagger} \big( \bW^{\dagger} \bT \big)^{\top, \dagger}
                    & - \bW^{\top,\dagger} \bP_{\bW^{\dagger} \bT}^{\perp} \bW^{\dagger} 
            \end{bmatrix}.
    \end{equation}

    \textit{Step 3. Concluding the proof.} 
    Lastly, we solve the KKT system \eqref{eqn:KKT_system.2} in Step 1, using the expression \eqref{eqn:inverse_KKT_matrix} for the inverse of the KKT matrix obtained in Step 2. 
    Specifically, solving this system yields 
    \begin{align*}
        \bbetaopt_{\cJ}
            &= \bP_{\bW^{\top}} \bP_{\bW^{\dagger} \bT}^{\perp} \bW^{\dagger} \by,\\ 
        \bbetaopt_{\cJ^c}
            &= \big( \bW^{\dagger} \bT \big)^{\dagger} \bW^{\dagger} \by.
    \end{align*}
    This completes the proof. 
\end{proof}

\subsection{Proof of Theorem~\ref{thm:sub_column_OLS}} \label{sec:proof_cochran}

\begin{proof} 
    Note that $\rank(\bX_{\star, \cJ}) = n$ by Assumption \ref{assump:partial}. 
    As a result, 
    \begin{align}
    	\min_{\balpha \in \Rb^{|\cJ|}} \| \by - \bX_{\star, \cJ} \balpha \|_2^2 &= 0,
	\\ \min_{\bDelta \in \Rb^{|\cJ| \times |\cJ^c|}} \| \bX_{\star, \cJ^c} - \bX_{\star, \cJ} \bDelta \|_F^2 &= 0. 
    \end{align} 
    Therefore, 
    \begin{align}
        \by &= \bX_{\star, \cJ} \hbalpha, \qquad \forall \hbalpha \in \cS_2,        \label{eqn:cochran.S2}\\
        \bX_{\star, \cJ^c} &= \bX_{\star, \cJ} \hbDelta, \qquad \forall \hbDelta \in \cS_3.      \label{eqn:cochran.S3}
    \end{align}
    Likewise, $\rank(\bX) = n$ as  $n = \rank(\bX_{\star, \cJ}) \leq \rank(\bX) \leq \min\{n, p\}$. 
    Thus, 
    \begin{equation}
        \by = \bX \hbbeta
            = \bX_{\star, \cJ} \hbbeta_{\cJ} + \bX_{\star, \cJ^c} \hbbeta_{\cJ^c},
                \qquad \forall \hbbeta \in \cS_1.       \label{eqn:cochran.S1}
    \end{equation}
    Combining \eqref{eqn:cochran.S3} and \eqref{eqn:cochran.S1} yields
    \[
        \by = \bX_{\star, \cJ} \left( \hbbeta_{\cJ} + \hbDelta \hbbeta_{\cJ^c} \right).
    \]
    This equation in combination with \eqref{eqn:cochran.S2} implies the first conclusion in \eqref{eq:cochran.hd}:
    \[
        \bX_{\star, \cJ} \hbalpha = \bX_{\star, \cJ} \left( \hbbeta_{\cJ} + \hbDelta \hbbeta_{\cJ^c} \right),
            \qquad \forall \big( \hbbeta, \hbalpha, \hbDelta \big) \in \cS_1 \times \cS_2 \times \cS_3.
    \]
    Suppose $\hbbeta \in \cS_1$, $\hbalpha \in \cS_2$, $\hbDelta \in \cS_3$ are each the unique minimum $\ell_2$-norm solutions. Then
    \[
        \hbbeta = \bX^{\dagger} \by,
        \qquad
        \hbalpha = \bX_{\star, \cJ}^{\dagger} \by,
        \qquad
        \hbDelta = \bX_{\star, \cJ}^{\dagger} \bX_{\star, \cJ^c}.
    \]
    Since $\hbbeta = \bX^{\dagger} \by$, we have $\hbbeta \in \rowsp(\bX)$ and thus, $\hbbeta_{\cJ} \in \rowsp(\bX_{\star,\cJ})$. 
    As such, we obtain
    \begin{align*}
        \hbbeta_{\cJ} + \hbDelta \hbbeta_{\cJ^c}
            &= \bX_{\star,\cJ}^{\dagger} \bX_{\star, \cJ} \hbbeta_{\cJ} + \bX_{\star, \cJ}^{\dagger} \bX_{\star, \cJ^c} \hbbeta_{\cJ^c}
                &&\because \hbbeta_{\cJ} = \bPi_{\bX_{\star,\cJ}^{\top}} \hbbeta_{\cJ} = \bX_{\star,\cJ}^{\dagger} \bX_{\star, \cJ} \hbbeta_{\cJ}\\
            &= \begin{bmatrix} \bX_{\star,\cJ}^{\dagger} \bX_{\star, \cJ} & \bzero \end{bmatrix} \hbbeta 
                + \begin{bmatrix} \bzero & \bX_{\star, \cJ}^{\dagger} \bX_{\star, \cJ^c} \end{bmatrix} \hbbeta\\
            &= \bX_{\star, \cJ}^{\dagger} \bX \bX^{\dagger} \by
                &&\because \hbbeta = \bX^{\dagger} \by\\
            &= \bX_{\star, \cJ}^{\dagger} \by
                &&\because \bX \bX^{\dagger} = \bI_n \text{ by Assumption \ref{assump:partial}}\\
            &= \hbalpha.
    \end{align*}
    This concludes the proof. 
\end{proof}

\subsection{Proof of Corollary~\ref{cor:loco}} \label{sec:additional_remarks}

%

\begin{proof}
%
We first state \eqref{eq:cochran.hd.2} of Theorem~\ref{thm:sub_column_OLS} for the specialized leave-one-covariate-out configuration. 
For any given $j \in [p]$, we let $\cJ = \{j\}^c = [p] \setminus \{j\}$ and $\cJ^c = \{j \}$ and obtain:  
\begin{align}
	\hbalpha^{(\sim j)} &= \hbbeta_{\{j\}^c} + \hbDelta^{(\sim j)} \cdot \hbeta_{j} \in \Rb^{p-1}, \label{eq:ovb.special}
\end{align} 
where $\hbbeta \in \cS_1$, $\hbalpha^{(\sim j)} \in \cS_2$, and $\hbDelta^{(\sim j)} \in \cS_3$ are each the minimum $\ell_2$-norm OLS solutions as defined in \eqref{eq:column.s1}. 
Recall $\hbbeta^{(\sim j)} \in \Rb^p$ is the zero-padded version of $\hbalpha^{(\sim j)}$, as defined in \eqref{eq:ovb.special.append}. 
%


%
We anchor on \eqref{eq:ovb.special} and \eqref{eq:ovb.special.append} to express the $\ell$-th coordinate as 
\begin{align}
	\left(\sum_{j = 1}^p \lambda_j \cdot \hbeta^{(\sim j)}\right)_\ell
        &= \sum_{j < \ell} \lambda_j \cdot \halpha^{(\sim j)}_{\ell-1}
            + \sum_{j > \ell} \lambda_j \cdot \halpha^{(\sim j)}_{\ell}
            &&\because \eqref{eq:ovb.special.append}
    	\\
        &=
    	\sum_{j < \ell} \lambda_j \cdot \left( \hbbeta_{\{j\}^c} + \hbDelta^{(\sim j)} \cdot \hbeta_{j} \right)_{\ell-1}
            + \sum_{j > \ell} \lambda_j \cdot \left( \hbbeta_{\{j\}^c} + \hbDelta^{(\sim j)} \cdot \hbeta_{j} \right)_{\ell}
            &&\because \eqref{eq:ovb.special}
    	\\
        &=
    	\sum_{j < \ell} \lambda_j \cdot \left( \hbeta_{\ell} + \hDelta^{(\sim j)}_{\ell-1} \cdot \hbeta_{j} \right)
            + \sum_{j > \ell} \lambda_j \cdot \left( \hbeta_{\ell} + \hDelta^{(\sim j)}_{\ell} \cdot \hbeta_{j} \right)
    	\\
    	&= (1 - \lambda_{\ell}) \hbeta_\ell  
            + \underbrace{
            \left( \sum_{j < \ell} \lambda_j  \hDelta^{(\sim j)}_{\ell-1} \hbeta_j
            + \sum_{j > \ell} \lambda_j  \hDelta^{(\sim j)}_{\ell} \hbeta_j \right)
            }_{\eqqcolon(*)}.
    	&&\because \sum_{j=1}^p \lambda_j = 1 \label{eq:guido.1}
\end{align}
Observe that $\hbDelta^{(\sim j)} = (\bX_{\star, \{j\}^c})^{\dagger} \bX_{\star, j}$, cf. Proposition \ref{prop:beta_hat}. 
We define $\tbDelta^{(\sim j)} \in \RR^p$ to be the zero-padded version of $\hbDelta^{(\sim j)}$, akin to how $\hbbeta^{(\sim j)}$ is defined in \eqref{eq:ovb.special.append}, such that 
\begin{equation} \label{eq:tilde_Delta}
	\tDelta^{(\sim j)}_\ell 
        = \begin{cases}
		  \hDelta^{(\sim j)}_\ell, & \text{if } \ell < j,\\
            0, & \text{if } \ell = j,\\
    		\hDelta^{(\sim j)}_{\ell-1}, & \text{if } \ell > j. 
	\end{cases} 
\end{equation}
It follows that $\tbDelta^{(\sim j)} = (\tbX_{(\sim j)})^{\dagger} \bX_{\star, j}$ where $\tbX_{(\sim j)} \in \RR^{n \times p}$ denotes the matrix $\bX$ with the $j$-th column replaced by $0$, i.e., a ``zero-padding'' of $\bX_{\star, \{j\}^c} \in \RR^{n \times (p-1)}$.
Recall the notation $\bG_{\bX} \coloneqq ( \bX \bX^{\top} )^{-1}$. 
Applying the Sherman–Morrison–Woodbury formula, we have
\begin{align}
    \left[\tbX_{(\sim j)} \tbX_{(\sim j)}^{\top} \right]^{-1}
        &= \left( \bX \bX^{\top} - \bX_{\star, j} \bX_{\star, j}^{\top} \right)^{-1}\\
        &= \bG_{\bX} + \frac{1}{1-h_{jj}} \bG_{\bX} \cdot \bX_{\star, j}\bX_{\star, j}^{\top} \cdot \bG_{\bX}.\label{eq:smw}
\end{align}
Then, it follows from \eqref{eq:smw} that
\begin{align}
    \tDelta^{(\sim j)}_{\ell} 
        &= \be_{\ell}^{\top} \tbX_{(\sim j)}^{\top} \left[\tbX_{(\sim j)} \tbX_{(\sim j)}^{\top} \right]^{-1}  \bX_{\star, j}\\
        &= (\bX_{\star, \ell})^{\top}
            \left( \bG_{\bX} + \frac{1}{1-h_{jj}} \bG_{\bX} \cdot \bX_{\star, j}\bX_{\star, j}^{\top} \cdot \bG_{\bX} \right)
            \bX_{\star, j}\\
        &= h_{\ell j} + \frac{h_{\ell j} h_{jj} }{1 - h_{jj}}\\
        &= \frac{ h_{\ell j}}{ 1 - h_{jj}}.
            \label{eq:Delta.simplify.1}
\end{align}

Armed with \eqref{eq:tilde_Delta} and \eqref{eq:Delta.simplify.1}, we continue to simplify the expression (*) in \eqref{eq:guido.1} as
\begin{align}
    \left( \sum_{j < \ell} \lambda_j  \hDelta^{(\sim j)}_{\ell-1} \hbeta_j
            + \sum_{j > \ell} \lambda_j  \hDelta^{(\sim j)}_{\ell} \hbeta_j \right)
        &= \sum_{j \neq \ell} \lambda_j  \tDelta^{(\sim j)}_{\ell} \hbeta_j\\
        &= \sum_{j \neq \ell} \frac{1 - h_{jj}}{p - n} \cdot \frac{h_{\ell j}}{1 - h_{jj}} \cdot \bX_{\star, j}^\top \cdot \bG_{\bX} \cdot  \by
    	\\
    	&= \frac{1}{p-n} \cdot \left( \sum_{j \neq \ell} h_{\ell j} \bX_{\star, j}^\top \right) \cdot \bG_{\bX} \cdot \by
    	\\
    	&= \frac{1}{p-n} \cdot \left( \sum_{j \neq \ell} \bX_{\star, \ell}^\top \bG_{\bX} \bX_{\star, j} \bX_{\star, j}^\top \right) \cdot \bG_{\bX} \cdot \by
    	\\
    	&= \frac{1}{p-n} \cdot  \bX_{\star, \ell}^\top \cdot \bG_{\bX} \cdot \left( \sum_{j \neq \ell} \bX_{\star, j} \bX_{\star, j}^\top \right) \cdot \bG_{\bX} \cdot \by
    	\\
    	&= \frac{1}{p-n} \cdot \bX_{\star, \ell}^\top \cdot \bG_{\bX} \cdot \left( \bX \bX^\top - \bX_{\star, \ell} \bX_{\star, \ell}^\top \right) \cdot \bG_{\bX} \cdot \by
    	\\
    	&= \frac{1}{p-n} \cdot \left( \bX_{\star, \ell}^\top \cdot \bG_{\bX} \cdot \by - \bX_{\star, \ell}^\top \cdot \bG_{\bX} \cdot \bX_{\star, \ell} \bX_{\star, \ell}^\top \cdot \bG_{\bX} \cdot \by \right)
    	\\
    	&= \frac{1}{p-n} \cdot \left( \hbeta_\ell - h_{\ell \ell} \cdot \hbeta_\ell \right) 
    	\\
    	&= \hbeta_\ell \cdot \left( \frac{1 - h_{\ell \ell}}{p-n} \right)
    	\\
	&= \hbeta_\ell \cdot \lambda_\ell. \label{eq:guido.2}
\end{align}
Combining \eqref{eq:guido.1} and \eqref{eq:guido.2}, we obtain our desired result. 
\end{proof} 

\section{Deferred proofs from Section~\ref{sec:stat_inf}} \label{sec:proofs.stats} 

%
%

\subsection{Proof of Theorem \ref{thm:Gauss_Markov}}

\begin{proof} 
	We provide a proof of Theorem~\ref{thm:Gauss_Markov.general}, which is a generalization of Theorem~\ref{thm:Gauss_Markov}.

    \medskip \noindent
    \emph{Proof of Part (a): Classical $(n>p)$.}
    We first prove the classical Gauss-Markov theorem in our notation. 
    If Assumption~\ref{assump:column_rank} holds, then there exists a left inverse $\bM$ of $\bX$ such that $\bM \bX = \bI_p$. 
    Observe that the set of left inverses of $\bX$ can be written as $\cS_{\uL} = \{ \bM \in \RR^{p \times n}: \bM = \bX^{\dagger} + \bN ~\text{with}~ \bN \bX = \bzero  \}$. 
    The set of unbiased linear estimators of $\bbeta$ is given as $\{ \tbbeta = \bM \by: \bM \in \cS_{\uL} \}$ as 
    \[
        \bbE[ \bM \by ] = \bbE[ \bM \bX \bbeta + \bM \bvarepsilon ] = \bM \bX \bbeta.
    \]
    Furthermore, for any deterministic matrix $\bM \in \RR^{p \times n}$, 
    \begin{align}\label{eqn:covar_tbbeta.0}
        \Cov( \bM \by )
            = \Cov (\bM \bvarepsilon)
            = \bM \cdot \Cov(\bvarepsilon) \cdot \bM^{\top} = \sigma^2 \cdot \bM \bM^{\top}
    \end{align}
    where the last equality follows from the homoskedasticity assumption $\Cov(\varepsilon) = \sigma^2 \bI_n$. 
    Therefore, for any $\bM = \bX^{\dagger} + \bN \in \cS_{\uL}$,
    \begin{align}
        \Cov( \bM \by )
            &= \sigma^2 \cdot \big( \bX^{\dagger} + \bN \big) \big( \bX^{\dagger} + \bN \big)^{\top}\\
            &= \sigma^2 \cdot \left( \bX^{\dagger}{\bX^{\dagger}}^{\top} + \bX^{\dagger} \bN^{\top} + \bN {\bX^{\dagger}}^{\top} + \bN \bN^{\top} \right)\\
            &= \sigma^2 \cdot \left( \bX^{\dagger}{\bX^{\dagger}}^{\top} + \bN \bN^{\top} \right)
                \quad \because \colsp(\bX) = \rowsp(\bX^{\dagger}) \implies \bN {\bX^{\dagger}}^{\top} = \bzero\\
            &= \Cov( \bX^{\dagger} \by ) + \sigma^2 \bN \bN^{\top}.
    \end{align}
    Since $\bN \bN^{\top}$ is positive semidefinite, $\Cov( \bX^{\dagger} \by ) \preceq \Cov( \bM \by )$ for all $\bM \in \cS_{\uL}$.

    \medskip
    \noindent
    \emph{Proof of Part (b): High-dimensional $(n \leq p)$.}
    
    \begin{itemize}
        \item
        \textit{Step 0 (Unbiased-through-$X$ $\iff$ right inverse).}
        Since $\bbE[\bX\tbbeta]=\bbE[\bX\bL\by]=\bX\bL\bX\bbeta$, the condition $\bbE[\bX\tbbeta]=\bX\bbeta$ for all $\bbeta$ is equivalent to  $\bX\bL\bX=\bX$. 
        Under Assumption~\ref{assump:row_rank}, $\colsp(\bX)=\RR^n$ and $\bX\bX^\dagger=\bI_n$, hence, $\bX\bL\bX=\bX$ implies $\bX\bL=\bX\bX^\dagger=\bI_n$. 
        Thus
        \[
            \bbE[\bX\tbbeta]=\bX\bbeta\ \ \text{for all }\bbeta
                \quad\Longleftrightarrow\quad
                \bX\bL=\bI_n.
        \]
        In particular, every such $\bL$ can be written as
        \[
            \bL=\bX^\dagger+\bN,\qquad \bX\bN=\bzero,
        \]
        i.e., $\bN$ has columns in $\ker(\bX)=\rowsp(\bX)^\perp$ (right-inverses)

        \item 
        \textit{Step 1 (Directional variances on $\rowsp(\bX)$).}
        For any deterministic matrix $\bA$, $\Cov(\bA\by)=\bA\bSigma\bA^\top$.
        Fix any $\bv\in\rowsp(\bX)$; then $\bv=\bX^\top\bc$ for some $\bc\in\RR^n$. Using $\bX\bX^\dagger=\bI_n$ and $\bX\bL=\bI_n$,
        \begin{align*}
            \bv^\top \Cov(\hbbeta)\,\bv
                &=\bc^\top\,\bX\,(\bX^\dagger\bSigma\bX^{\dagger,\top})\,\bX^\top\bc
                    =\bc^\top\,(\bX\bX^\dagger)\,\bSigma\,(\bX\bX^\dagger)^\top\bc
                    =\bc^\top\bSigma\,\bc,\\
            \bv^\top \Cov(\tbbeta)\,\bv
                &=\bc^\top\,\bX\,(\bL\bSigma\bL^\top)\,\bX^\top\bc
                    =\bc^\top\,(\bX\bL)\,\bSigma\,(\bX\bL)^\top\bc
                    =\bc^\top\bSigma\,\bc.
        \end{align*}
        Hence for every $\bv=\bX^\top\bc\in\rowsp(\bX)$,
        \[
            \bc^\top\bSigma\bc \;=\; \bv^\top\Cov(\hbbeta)\,\bv \;=\; \bv^\top\Cov(\tbbeta)\,\bv,
        \]
        proving part (b)-(i).

        \item 
        \textit{Step 2 (Directional variances on $\rowsp(\bX)^\perp$ and the equality condition).}
        Let $\bw\in\rowsp(\bX)^\perp$. Since $\bX\bw=\bzero$, we have $\bX^{\dagger,\top}\bw=(\bX\bX^\top)^{-1}\bX\bw=\bzero$, so
        \[
            \bw^\top\Cov(\hbbeta)\,\bw
                =\bw^\top(\bX^\dagger\bSigma\bX^{\dagger,\top})\bw
                =(\bX^{\dagger,\top}\bw)^\top\bSigma\,(\bX^{\dagger,\top}\bw)=0.
        \]
        For $\tbbeta=\bL\by=(\bX^\dagger+\bN)\by$ with $\bX\bN=\bzero$,
        \begin{align*}
            \bw^\top\Cov(\tbbeta)\,\bw
                &=\bw^\top(\bL\bSigma\bL^\top)\bw
                \\ &=\underbrace{\bw^\top(\bX^\dagger\bSigma\bX^{\dagger,\top})\bw}_{=\,0}
                    + \underbrace{\bw^\top(\bN\bSigma\bX^{\dagger,\top})\bw}_{=\,0}
                    + \underbrace{\bw^\top(\bX^\dagger\bSigma\bN^{\top})\bw}_{=\,0}
                    +\ \bw^\top\big(\bN\bSigma\bN^\top\big)\bw\\
                &=\big\|\bSigma^{1/2}\bN^\top\bw\big\|_2^2\ \ge\ 0,
        \end{align*}
        which proves $\bw^\top\Cov(\hbbeta)\,\bw \le \bw^\top\Cov(\tbbeta)\,\bw$ for every $\bw\in\rowsp(\bX)^\perp$.
        
        It remains to characterize equality for all such directions. Since $\bN^\top$ vanishes on $\rowsp(\bX)$ (because $\bN^\top\bX^\top=(\bX\bN)^\top=\bzero$),
        its image is generated by $\rowsp(\bX)^\perp$:
        \[
            \colsp(\bN^\top)=\bN^\top\!\big(\rowsp(\bX)^\perp\big).
        \]
        Thus $\bw^\top\Cov(\tbbeta)\,\bw=\|\bSigma^{1/2}\bN^\top\bw\|_2^2=0$ for all $\bw\in\rowsp(\bX)^\perp$
        if and only if $\colsp(\bN^\top)\subseteq\ker(\bSigma)$, i.e.,
        \[
            \rowsp(\bN)\subseteq\ker(\bSigma).
        \]
        Recalling $\bN=\bL-\bX^\dagger$, this is precisely $\rowsp(\bL-\bX^\dagger)\subseteq\ker(\bSigma)$, proving part (b)-(ii). 
    \end{itemize}
\end{proof}

%

\subsection{Proof of Theorem~\ref{thm:var.estimator}} 

\begin{proof} 
If Assumption \ref{assump:lm} holds with $\bSigma = \sigma^2 \bI$, then for any deterministic matrix $\bM \in \RR^{n \times n}$
\begin{align}
	\bbE[ \by^\top \bM \by] 
        &= \bbE\left[ \left(\bX \bbeta^* + \bvarepsilon \big)^\top \bM \big(\bX \bbeta^* + \bvarepsilon \right) \right]\\
        &= (\bbeta^*)^\top \bX^\top \bM \bX \bbeta^* + \bbE \left[ \bvarepsilon^\top \bM \bvarepsilon  \right]\\
        &= (\bbeta^*)^\top \bX^\top \bM \bX \bbeta^* + \sigma^2 \cdot \tr(\bM),   \label{eqn:exp_quadratic}
\end{align}
where the last equality follows from 
\[
    \bbE [ \bvarepsilon^\top \bM \bvarepsilon  ] = \bbE [ \tr( \bvarepsilon^\top \bM \bvarepsilon) ] = \bbE [ \tr( \bM \bvarepsilon \bvarepsilon^\top ) ] = \tr( \bM \cdot \bbE [  \bvarepsilon \bvarepsilon^\top ] ) = \sigma^2 \cdot \tr(\bM).
\]
We now present the proof for each data regime. 
\begin{enumerate} 
\item \emph{Classical $(n > p)$}: 
Recall from Corollary~\ref{cor:loo.shortcut} that the LOO residuals in the classical regime obey $\tbvarepsilon = \big[ \diag(\bP_{\bX}^\perp)\big]^{-1} \cdot \bP_{\bX}^\perp \by$. 
As such, we obtain 
\begin{align}
	\bbE[ \tbvarepsilon^\top \tbvarepsilon] 
        &= \bbE[ \by^\top \bP_{\bX}^\perp \cdot \big[ \diag(\bP_{\bX}^\perp)\big]^{-1} \cdot \big[ \diag(\bP_{\bX}^\perp)\big]^{-1} \cdot \bP_{\bX}^\perp \by ]\\
        &= \sigma^2 \cdot \tr \left( \bP_{\bX}^\perp \cdot \big[ \diag(\bP_{\bX}^\perp)\big]^{-1} \cdot \big[ \diag(\bP_{\bX}^\perp)\big]^{-1} \cdot \bP_{\bX}^\perp \right)
            &&\because \eqref{eqn:exp_quadratic}~\&~ \bP_{\bX}^\perp \bX = \bzero\\
        &= \sigma^2 \cdot \left\| \big[ \diag(\bP_{\bX}^\perp)\big]^{-1} \cdot \bP_{\bX}^\perp \right\|_{\F}^2
\end{align}
because $\tr( \bM^{\top} \bM ) = \| \bM \|_{\F}^2$.
Hence, 
\begin{align}
	\bbE\left[ \frac{ \left\| \tbvarepsilon \right\|_2^2 }{ \big\| \big[ \diag(\bP_{\bX}^\perp)\big]^{-1} \cdot \bP_{\bX}^\perp \big\|_{\F}^2 } \right] &= \sigma^2. 
\end{align} 

\item \emph{High-dimensional $(p \ge n)$}:  
Next, recall from Corollary~\ref{cor:loo.shortcut} that the LOO residuals in high-dimensions obey $\tbvarepsilon = \big[ \diag(\bG_{\bX}) \big]^{-1} \cdot \bG_{\bX} \by$.
Similarly as above, we apply \eqref{eqn:exp_quadratic} to obtain   
\begin{align}
	\bbE[ \tbvarepsilon^\top \tbvarepsilon] 
        &= \bbE[\by^\top \bG_{\bX} \cdot \big[ \diag(\bG_{\bX}) \big]^{-1} \cdot \big[ \diag(\bG_{\bX}) \big]^{-1} \cdot \bG_{\bX} \by]  \\
        &= (\bbeta^*)^\top \bX^\top \left( \bG_{\bX} \cdot \big[ \diag(\bG_{\bX}) \big]^{-1} \cdot \big[ \diag(\bG_{\bX}) \big]^{-1} \cdot \bG_{\bX} \right ) \bX \bbeta^*\\
            &\qquad + \sigma^2 \cdot \tr \left( \bG_{\bX} \cdot \big[ \diag(\bG_{\bX}) \big]^{-1} \cdot \big[ \diag(\bG_{\bX}) \big]^{-1} \cdot \bG_{\bX} \right)\\
        &= \big\| \bbE[\tbvarepsilon] \big\|_2^2 + \sigma^2 \cdot \big\| \big[ \diag(\bG_{\bX}) \big]^{-1} \cdot \bG_{\bX} \big\|_{\F}^2,
\end{align}
where the last equality follows from $\bbE[\tbvarepsilon] = \big[ \diag(\bG_{\bX}) \big]^{-1} \cdot \bG_{\bX} \bX \bbeta^*$. 
Hence, 
\begin{align}
	\bbE\left[ \frac{ \left\| \tbvarepsilon \right\|_2^2 }{ \big\| \big[ \diag(\bG_{\bX}) \big]^{-1} \cdot \bG_{\bX} \big\|_{\F}^2 } \right] 
        &= \sigma^2 + \frac{ \big\| \bbE[\tbvarepsilon] \big\|_2^2}{ \big\| \big[ \diag(\bG_{\bX}) \big]^{-1} \cdot \bG_{\bX} \big\|_{\F}^2 }. 
\end{align} 
\end{enumerate} 
This completes the proof. 
\end{proof}

\end{document}